\documentclass[12pt, a4paper,reqno]{amsart}
\numberwithin{equation}{section}
\usepackage{etex}
\usepackage{pstricks-add}
\usepackage{latexsym}
\usepackage{amssymb}
\usepackage{mathrsfs}
\usepackage{pictexwd}
\usepackage{amsmath}
\usepackage{fancybox,color}
\usepackage{enumerate}
\usepackage[utf8]{inputenc}
\usepackage{eurosym}
\usepackage{color}
\usepackage{url}
\usepackage{pgf,tikz}\usepackage{mathrsfs}\usetikzlibrary{arrows}

\numberwithin{figure}{section}

\newcommand{\kom}[1]{}
\renewcommand{\kom}[1]{{\bf [#1]}}

\newcommand{\be}{\begin{equation}}
\newcommand{\ee}{\end{equation}}

\usepackage{hyperref}
\hypersetup{
    colorlinks=true,                          
    linkcolor=blue, % Colors for internal links
    citecolor=red, % Colors for the biblio references
    urlcolor=blue  } % Colors for url
 \def\1{\raisebox{2pt}{\rm{$\chi$}}}

\newtheorem{theorem}{Theorem}[section]
\newtheorem{lem}[theorem]{Lemma}
\newtheorem{proposition}[theorem]{Proposition}
\newtheorem{remark}{Remark}[section]

\newcommand{\R}{{\mathbb R}}

\newcommand{\N}{{\mathbb N}}

\newcommand{\E}{{\mathbb E\,}}

 \newcommand{\eps}{\varepsilon}
 \def\1{\raisebox{2pt}{\rm{$\chi$}}}
 
\newcommand{\Rn}{\mathbb{R}^n}

\usepackage{pstricks-add}

\begin{document}
\title[Singularity formation at multiple times]
{Singularity formation and regularization \\ 
at multiple times in the viscous \\
Hamilton-Jacobi equation}

\date{\today}

\author[Mizoguchi]{Noriko Mizoguchi}%
\address{Department of Mathematics, Tokyo Gakugei University,
Koganei, Tokyo 184-8501, Japan }
\email{mizoguti@u-gakugei.ac.jp}

\author[Souplet]{Philippe Souplet}%
\address{Universit\'e Sorbonne Paris Nord,
CNRS UMR 7539, Laboratoire Analyse, G\'{e}om\'{e}trie et Applications,
93430 Villetaneuse, France}
\email{souplet@math.univ-paris13.fr}

 \begin{abstract}  
It is known that the Cauchy-Dirichlet problem for the super\-qua\-dratic viscous Hamilton-Jacobi equation, 
which has important applications in stochastic control theory,
admits a unique, global viscosity solution.
Solutions thus exist in the weak sense after the appearance of singularity in finite time, 
which occurs through gradient blow-up (GBU) on the boundary. 
Whereas the theory of viscosity solutions has been extensively studied and applied 
to many partial differential equations, 
there seem to be less results on refined behavior of solutions.
Although occurrence of two types of GBU with loss of boundary condition (LBC) and without LBC are known,
detailed behavior of viscosity solutions after GBU has remained open except 
for a strongly restricted special class of solutions in one space dimension. 
\vskip 1.5pt

In this paper, in general dimensions, for each $ m \ge 1 $ we construct solutions which undergo GBU 
and LBC at least at $ m $ times 
and then recover regularity. 
We also construct solutions that exhibit GBU without LBC at their first blowup time. 
In one space dimension, we obtain the complete classification of viscosity solutions at each time, 
which extends to radial cases in higher dimensions.  
Furthermore for each $ m \ge 2 $ and an arbitrarily given combination of GBU types with/without LBC at $ m $ times 
in an arbitrarily given order,
we show the existence of a solution which exhibits this exact combination of GBU. 
Some solutions display a new type of behavior called "bouncing".
\vskip 1.5pt

Global weak solutions of the viscous Hamilton-Jacobi equation with singularity at multiple times
turn out to display a larger variety of behaviors than those of the Fujita equation.  
We introduce a method based on an arbitrary
number of critical parameters, whose continuity 
requires a delicate argument.
Since we do not rely on any known special solution unlike in the proof in Fujita equation, 
our method is expected to apply to other equations.
\vskip 1.5pt

Singular behaviors at multiple times are completely new in the context not only of viscous Hamilton-Jacobi 
equation but also of stochastic control theory. In this framework our results imply that
for certain spatial distributions of rewards, if a controled Brownian particle starts near the boundary, 
then the net gain attains profitable values on different time horizons but not on some intermediate times.
 \end{abstract}

\maketitle

\tableofcontents

\section{Introduction and main results}

\subsection{Problem}

We consider the initial boundary value problem for the viscous Hamil\-ton-Jacobi equation: %%%
\be\label{vhj1}
\begin{cases}
u_t-\Delta u=|\nabla u|^p, & \quad x\in\Omega,\ t>0,\\
u =0, & \quad x\in\partial\Omega,\ t>0,\\
u(x,0) =\phi(x),  & \quad x\in\Omega  
\end{cases}
\ee
with $ p > 2 $, where $\phi\in X:=\bigl\{\phi\in C^1(\overline\Omega);\ \phi\ge 0,\ \phi=0 \hbox{ on } \partial\Omega\bigr\} $.
Throughout this paper, $\Omega$ is a smoothly bounded domain of $\R^n$.

By standard theory \cite{Fri}, problem~\eqref{vhj1} admits a unique, maximal classical solution $u$ satisfying $ u \ge 0 $
in $ \overline\Omega \times [0,T) $, where $T=T(\phi)\in (0,\infty]$ denotes the maximal existence time.
It is known that, for suitably large initial data, solutions blow up in finite time, 
i.e., $T(\phi)<\infty$ (see, e.g., \cite{A96, ABG89, S02}).
Since the solution satisfies 
$$\|u(t)\|_\infty\le \| \phi \|_\infty,\quad 0<t<T,$$ 
by the maximum principle, the classical solution can only cease to exist through {\bf gradient blowup} (GBU)
$$ \lim_{t\to T}\|\nabla u(t)\|_\infty=\infty.$$
However the solution survives after the blow-up time and can be continued as a generalized viscosity solution of \eqref{vhj1}. 
More precisely, by \cite{BDL04}, problem \eqref{vhj1} admits a unique, global (generalized) viscosity solution; 
this solution is nonnegative and continuous in $\overline\Omega\times[0,\infty)$, and it coincides with the (unique) classical 
solution in~$(0,T)$.
Throughout this paper, we shall also denote it by $u$, without risk of confusion
(or $u(\phi;\cdot,\cdot)$ when we need to emphasize its dependence on the initial data).

Whereas the theory of viscosity solutions (cf.~\cite{CIL}) has been extensively studied and applied to many partial differential 
equations, there seem to be less results on refined behavior of the weak solution. 
Our purpose is to investigate the behavior of the global viscosity solution $u$ for $t>T(\phi)$.
First of all, it is known that the solution $u$ enjoys interior regularity, i.e.~$u \in C^{1,2}(\Omega\times(0,\infty))$, 
and solves the PDE in the pointwise sense in $\Omega\times(0,\infty)$,
but it may lose the boundary conditions in the classical sense.  
Indeed such as possibility, which was suggested in \cite{BDL04}, was confirmed in \cite{PS2, QR16} 
where it was shown that, for suitably large initial data, 
the solution undergoes a loss of boundary conditions (LBC) at some times $t>T(\phi)$, i.e., 
$$\sup_{x\in\partial\Omega} u(x,t)>0.$$
However, some exceptional GBU solutions without LBC were also shown to exist in \cite{PS2, PS3},
found as separatrices between global solutions and GBU solutions with LBC.
We see that LBC solutions, which are meant to satisfy zero boundary conditions in the generalized viscosity sense
(see \cite{BDL04} for precise formulation), nevertheless have to continuously take on some positive boundary values.
This apparently paradoxical situation can however be interpreted in a more intuitive way,
when one recalls that the global viscosity solution can also be obtained as the limit 
of a sequence of global classical solutions of regularized versions of problem \eqref{vhj1}, with truncated nonlinearity
(see e.g.~\cite{PS3} and Section~2 below for details). Since this convergence is monotone increasing but not uniform up to the boundary,
the loss of boundary conditions can in this framework be seen as a more familiar
 boundary layer phenomenon.

On the other hand, it was shown in \cite{PZ} that $u$ 
becomes a classical solution again for all sufficiently large time, i.e. there exists $\tilde T\ge T$ such that 
$$\hbox{$u\in C^{2,1}(\overline\Omega\times (\tilde T,\infty))$,\ 
with $u=0$ on $\partial\Omega\times [\tilde T,\infty)$,}$$
 and furthermore $u$ decays exponentially in $C^1(\overline\Omega)$ as $t\to\infty$. 
The natural and important question is thus to describe the behavior of $u$ in the intermediate time range $[T,\tilde T]$.
In \cite{PS3}, for $n=1$ and $\Omega=(0,1)$, 
a strongly restricted special 
class of symmetric initial data was studied, 
for which the solution immediately loses the boundary conditions after $T$ 
and is immediately and permanently regularized after recovering the boundary condition. 
It has remained open whether or not there exists a viscosity solution undergoing GBU at multiple times. 
We affirmatively answer the question in general dimensions constructing a viscosity solution exhibiting GBU with LBC at multiple times.
We give a precise behavior in one dimension and in radial domains in higher dimensions.
There have been no results on the fundamental behaviors like finiteness of GBU times 
and existence of waiting time of recovery of regularity, except the special case in \cite{PS3}. 
We give a complete classification at each time.
Furthermore for each $ m \ge 2 $ and arbitrarily given combination of GBU types with/without LBC at $ m $ times, 
we construct a viscosity solution undergoing this exact combination of GBU 
 in the case of $ n = 1 $, which can be easily extended 
to radial case with $ n \ge 2 $.

In view of the very large literature devoted to the viscous Hamilton-Jacobi equation, this introduction 
makes no attempt to be exhaustive.
For other aspects of \eqref{vhj1} and related problems, we refer to, e.g.,
\cite{ABG89, FL94, CG96, S02, BSW, BKL, ARS04, HM04, SV, GH08, LS, ZL13, PS, QS07, FPS19, FLa, Esteve, AS20}
 and the references therein.

\subsection{Applications to stochastic control problems} 
Let us recall that \eqref{vhj1} arises in stochastic control problems. 
Namely, consider the controlled $n$-dimen\-sional stochastic differential equation
$$dX_s =\alpha_s ds+dW_s,\  \ s>0, \quad\hbox{ with } X_0 =x\in\Omega,$$
where the stochastic process $(X_s)_{s>0}$ represents the position or state of the system,
$(W_s)_{s>0}$ is a standard Brownian motion and 
$(\alpha_s)_{s>0}$ is the control
(in other words, the controller can choose the velocity of $X$).
The spatial distribution of rewards is given by a function $u_0\in C_0(\overline\Omega)$.
More precisely, at a given time horizon $s=t>0$, the final reward is
$u_0(X_t)$ if $X$ stays in $\Omega$ until time~$t$,
and $0$ otherwise. Finally, the cost of the control at each time $s$ is assumed to be
$k_p |\alpha_s|^q$ as long as $X_s$ stays in $\Omega$, 
where $q=p/(p-1)$ is the conjugate exponent of $p$ and $k_p>0$ is a normalization constant.
The goal of the controller is then to maximize the net gain
$$G_t=\chi_{\tau>t}u_0(X_t)-k_p\displaystyle\int_0^\tau|\alpha_s|^q\, ds,$$
where $\tau$ denotes the first exit time of $X$ from $\Omega$.
It is known (see \cite{BB, BDL04, FlSo} for details) 
that the maximal gain, also called value function of the stochastic control problem,
is given by the unique global (continuous) viscosity solution $u$ of \eqref{vhj1}, namely:
$$u(x,t)=\sup_{(\alpha_s)_s}  \E{\hskip 0.5pt}\bigl(G_t\, | \, X_0=x\bigr),$$
where $\E{\hskip 0.5pt}\bigl(\,\cdot\,| \, X_0=x\bigr)$ denotes the conditional expectation 
with respect to the event $\{X_0=x\}$,
and the supremum is taken over all (admissible) controls.
In this framework, the existence of solutions with multiple times of loss and recovery of boundary conditions,
obtained in Theorems~\ref{thm1}-\ref{conjz1} below have an interesting interpretation:  
 the system (controlled Brownian particle) starts near the boundary, 
for certain spatial distributions of rewards inside the domain,
the maximal gain attains profitable values on different time horizons but not on some intermediate times.

\subsection{Results in general domains}

Our first main result in general domains is the following. It shows the existence
of solutions undergoing GBU with losses and recoveries of boundary conditions at arbitrarily many times.

\goodbreak
\begin{theorem}\label{thm1} 
Let $p>2$ and let $\Omega$ be any smoothly bounded domain of $\R^n$.
For any integer $m\ge 2$, there exists $\phi\in X$ undergoing GBU with at least $m$ 
losses and recoveries of boundary conditions.
More precisely, we have 
$T(\phi)<\infty$ and there exist times $s_m>\cdots>s_1>T(\phi)$ and nonempty open subintervals $J_i\subset (s_i,s_{i+1})$
for $i=1,\dots,m-1$, such that
\be\label{conclthm1A}
\sup_{x\in\partial\Omega} u(x,s_i)>0 \hbox{ for each $i\in\{1,\dots,m\}$,}
\ee
\be\label{conclthm1B}
\begin{aligned}
&\hbox{$u$ is a classical solution on $J_i$ for each $i\in\{1,\dots,m-1\}$, i.e.}\\
&\qquad\quad\hbox{i.e., $u\in C^{2,1}(\overline\Omega\times J_i)$ \ \ and \ \ $u=0$ on $\partial\Omega\times J_i.$}
\end{aligned}
\ee
\end{theorem}

We can also show that the scenario in Theorem~\ref{thm1} can be preceded by the first GBU without LBC.

\begin{theorem}\label{thm1B} 
Let $p>2$ and let $\Omega$ be any smoothly bounded domain of $\R^n$.
For any integer $m\ge 2$, there exists $\phi\in X$ such that $u$ undergoes:

$\bullet$ GBU without LBC at $t=T(\phi)\,<\infty$,
i.e.~$u=0$ on $\partial\Omega\times [0,T(\phi)+\tau]$ for some $\tau>0$;

$\bullet$ and then at least $m$ losses and recoveries of boundary conditions;
 namely~there exist times $s_m>\cdots>s_1>T(\phi)+\tau$ 
and nonempty open subintervals $J_i\subset (s_i,s_{i+1})$ for $i=1,\dots,m-1$, such that 
\eqref{conclthm1A}-\eqref{conclthm1B} holds.
\end{theorem}

\subsection{Results in one space dimension}

We will obtain more precise results in the case of $ n=1$. 
We only consider the behavior at $x=0$.
Similar results at $x=1$ immediately follow by considering the reflected solution $\tilde u(x,t)=u(1-x,t)$.

Set
$$X_1:=\bigl\{\phi\in C^1([0,1]);\ \phi\ge 0,\ \phi(0)=\phi(1)=0\bigr\},$$ 
denote
$$U^*(x)=c_p\, x^\alpha,\quad x>0,\quad\hbox{ with $\alpha=\frac{p-2}{p-1}$,\ \ $c_p=(p-2)^{-1}(p-1)^{(p-2)/(p-1)}$,}$$
the singular steady state vanishing at $x=0$, and set
\be\label{defNt}
N(t):=z(u(\cdot,t)-U^*),
\ee
where $z$ denotes the number of sign changes on $[0,1]$, i.e.
$$z(v)=\sup\bigl\{m\in\N;\ \exists\, x_0<\dots<x_m\in (0,1),\ v(x_{i-1})v(x_i)<0,\ i=1,\dots,m\bigr\},$$
 with the convention $z(v)=-\infty$ in case $v\equiv 0$.
Let us also recall (see \cite[Lemma 5.1]{PS3}) that
\be\label{lemz3ahyp0}
u_x(0,t)=\lim_{x\to 0} u_x(x,t) \in \R\cup\{+\infty\},\quad t>0
\ee
(where the limit exists, possibly $+\infty$).
A crucial ingredient in our analysis is the study of the following ``transition'' set: 
 \be\label{singulartimes}
\mathcal{T}=\mathcal{T}_\phi:=\bigl\{t>0;\ u(0,t)=0\ \hbox{and}\ u_x(0,t)=\infty\bigr\},
\ee
in connection with the intersection number properties of the solution $u$ with the singular steady state $U^*$.
For future reference, we also denote the singular set
 \be\label{singulartimesT}
\mathcal{S}=\mathcal{S}_\phi:=\bigl\{t>0;\ u_x(0,t)=\infty\bigr\}.
\ee

The following two theorems are our main one-dimensional results.
The first one gives a classification of all possible behaviors at any time, for any viscosity solution to problem~\eqref{vhj1}.

\begin{theorem}\label{thmz1} 
Let $p>2$, $n=1$, $\Omega=(0,1)$, $\phi\in X_1$. 

(i) The set $\mathcal{T}=\mathcal{T}_\phi$ is at most finite.
 Moreover, if $N(0)<\infty$ then we have 
\be\label{cardS}
\#\,\mathcal{T}\le N(0).
\ee

(ii) On each interval between two consecutive times $t_1,t_2\in\mathcal{T}$, the solution is either:

\ \ $\bullet$ classical up to $x=0$, i.e.:
\be\label{classicalint}
\hbox{$u\in C^{2,1}([0,\frac12]\times(t_1,t_2))$ \quad and \quad $u=0$ on $\{0\}\times(t_1,t_2)$}
\ee

\ \ $\bullet$ or of LBC type at $x=0$, i.e.;
\be\label{LBCint}
\hbox{$u>0$ on $\{0\}\times(t_1,t_2)$.}
\ee
\end{theorem}

Our results also answer another question left open in \cite{PS3}: as a consequence of Theorem~\ref{thmz1}, 
we see that waiting time phenomena cannot occur, at least in one space dimension.
Once a solution undergoes gradient blowup, the solution either loses boundary conditions immediately or is immediately regularized.

The next result is in some sense the reciprocal of Theorem~\ref{thmz1}.
It shows that any given finite sequence of behaviors a priori permitted by Theorem~\ref{thmz1} is 
actually realized for suitable choice of initial data.
In the following, the letters $C, L$ respectively stand for ``Classical up to $x=0$'' and ``LBC at $x=0$''
(cf.~\eqref{classicalint}-\eqref{LBCint}).

\begin{theorem}\label{conjz1} 
Let $p>2$, $n=1$, $\Omega=(0,1)$ and let $\ell$ be a positive integer.   
Let $(\sigma_i)_{0\le i\le \ell}$ be any finite sequence with values in $\{C,L\}$, such that $\sigma_0=\sigma_{\ell+1}=C$,
and set $t_0=0$ and $t_\ell=\infty$.
There exist $\phi\in X_1$ and times $0<t_1<\dots<t_\ell\in\mathcal{T}$ such that

\begin{itemize}

\item[$\bullet$] for each $i\in \{0,\dots,\ell\}$, the behavior of $u$ on $(t_i,t_{i+1})$ is of type~$\sigma_i$
\smallskip

\item[$\bullet$]  $u$ is classical up to $x=1$ for all times, 
i.e.~$u\in C^{2,1}([\frac12,1]\times(0,\infty))$ and $u=0$ on $\{1\}\times(0,\infty)$.

\end{itemize}
\end{theorem}

Typical behaviors given in Theorems~\ref{thmz1} and \ref{conjz1} 
are illustrated in Fig.~\ref{FigLBC}--\ref{FigMixed}.

The name transition set for $\mathcal{T}$ is motivated by the fact that, in view of Theorems~\ref{thmz1} and \ref{conjz1},
the elements of $\mathcal{T}$ play the role of 
 transition times between intervals with behaviors C,L.
The passage from classical to LBC (resp., to classical)  
is associated with GBU with (resp., without) LBC.
As for the passage from LBC to classical it corresponds to the phenomenon of regularization.
Here we discover a new type of behavior, which we call {\it bouncing.} 
This is when the solution passes from LBC to LBC, through a single time of recovery of boundary conditions.
Whereas solutions with separated LBC time intervals should be rather stable, 
the phenomena of GBU without LBC or of bouncing are expected to be unstable.

\begin{figure}[h]
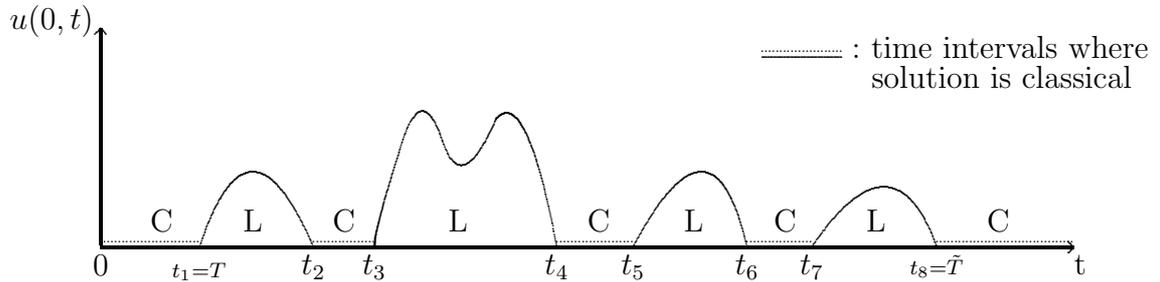

\[
\beginpicture
\setcoordinatesystem units <1cm,1cm>
\setplotarea x from -6 to 6, y from -1 to 4

\setdots <0pt>
\linethickness=1pt
\putrule from -5 0 to 7.8 0
\putrule from -5 0 to -5 2.9

\setdots <1.4pt> 
\setlinear
\plot -5 0.075 -3.7 0.075 /
\plot -2.2 0.075 -1.4 0.075 /
\plot 1 0.075 2 0.075 /
\plot 3.5 0.075 4.35 0.075 /
\plot 6 0.075 7.8 0.075 /

\setdots <0pt> 

\setquadratic

\plot -3.7 0 -3 1 -2.2 0  /

\plot -1.4 0  -1.3 0.5 -1 1.5 /
\plot -1 1.5 -0.75 1.8 -0.5 1.4 /
\plot -0.5 1.4 -0.2 1.1 0.2 1.7 /
\plot  0.2 1.7 0.6 1.5 1 0 /

\plot 2 0 2.9 1 3.5 0  /

\plot 4.35 0 5.3 0.8 6 0  /

\put {t} [lt] at 7.8 -.1
\put {$u(0,t)$} [rc] at -5.1 3

\setlinear 
\setdots <0pt> 
\plot 7.725 -.075 7.8 0 /
\plot 7.725 .075 7.8 0 /
\plot -5.075 2.825 -5 2.9 /
\plot -5 2.9 -4.925 2.825 /

\put {$0$}  [ct] at -5 -.1
\put {C}  [ct] at -4.2 0.5
\put {$^{t_1=T}$}  [ct] at -3.7 -.2
\put {L}  [ct] at -3 0.5
\put {$t_2$}  [ct] at -2.2 -.1
\put {C}  [ct] at -1.8 0.5
\put {$t_3$}  [ct] at -1.4 -.1
\put {L}  [ct] at -0.3 0.5
\put {$t_4$}   [ct] at 1 -.1
\put {C}  [ct] at 1.55 0.5
\put {$t_5$}   [ct] at 2 -.1
\put {L}  [ct] at 2.8 0.5
\put {$t_6$}   [ct] at 3.5 -.1
\put {C}  [ct] at 4 0.5
\put {$t_7$}   [ct] at 4.35  -.1
\put {L}  [ct] at 5.2 0.5
\put {$^{t_8=\tilde T}$}   [ct] at 6 -.1
\put {C}  [ct] at 6.8 0.5

\setlinear
\setdots <1.4pt> 
\plot 3.7 2.6 4.75 2.6 /
\setdots <0pt> 
\plot 3.7 2.525 4.75 2.525 /
\put {:~time intervals where }  [ct] at 6.9 2.8
\put {\ \phantom{a,}solution is classical}  [ct] at 6.62 2.4

\endpicture
\] 
\caption{A solution with exactly $4$ losses and recoveries of boundary conditions } 
\label{FigLBC}
\end{figure}

\begin{figure}[h]
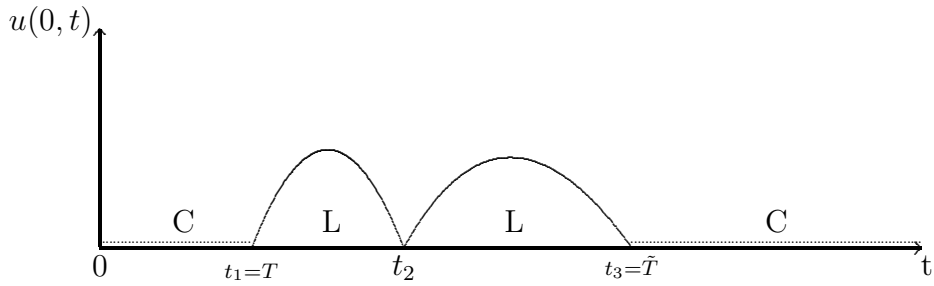

\[
\beginpicture
\setcoordinatesystem units <1cm,1cm>
\setplotarea x from -6 to 6, y from -1 to 3.5

\setdots <0pt>
\linethickness=1pt
\putrule from -5 0 to 5.8 0
\putrule from -5 0 to -5 2.9

\setdots <1.4pt> 
\setlinear
\plot -5 0.075 -3 0.075 /
\plot 2 0.075 5.8 0.075 /

\setdots <0pt> 

\setquadratic

\plot -3 0 -2 1.3 -1 0  /
\plot -1 0 0.4 1.2 2 0  /

\put {t} [lt] at 5.8 -.1
\put {$u(0,t)$} [rc] at -5.1 3

\setlinear 
\setdots <0pt> 
\plot 5.725 -.075 5.8 0 /
\plot 5.725 .075 5.8 0 /
\plot -5.075 2.825 -5 2.9 /
\plot -5 2.9 -4.925 2.825 /

\put {$0$}  [ct] at -5 -.1
\put {C}  [ct] at -3.9 0.5
\put {$^{t_1=T}$}  [ct] at -3 -.2
\put {L}  [ct] at -1.95 0.5
\put {$t_2$}  [ct] at -1 -.1
\put {L}  [ct] at 0.45 0.5
\put {$^{t_3=\tilde T}$}   [ct] at 2 -.1
\put {C}  [ct] at 3.9 0.5

\endpicture
\] 
\caption{A solution with exactly $1$ bouncing}
\label{FigBouncing}
\end{figure}

\begin{figure}[h]
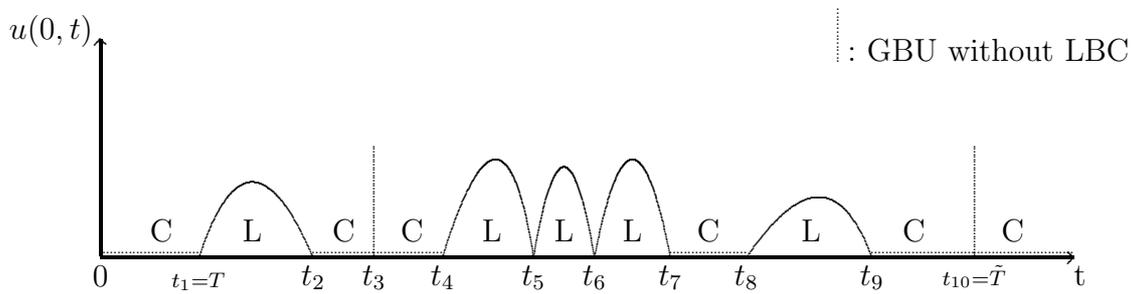

\[
\beginpicture
\setcoordinatesystem units <1cm,1cm>
\setplotarea x from -6 to 6, y from -1 to 4

\setdots <0pt>
\linethickness=1pt
\putrule from -5 0 to 7.8 0
\putrule from -5 0 to -5 2.9

\setdots <1.4pt> 
\setlinear
\plot -5 0.075 -3.7 0.075 /
\plot -2.2 0.075 -1.4 0.075 /
\plot -1.4 0.075 -0.5 0.075 /
\plot 2.5 0.075 3.5 0.075 /
\plot 5.15 0.075 6.5 0.075 /
\plot 6.5 0.075 7.8 0.075 /

\setdots <1pt> 

\plot -1.4 0  -1.4 1.5 /
\plot 6.5 0  6.5 1.5 /

\setdots <0pt> 

\setquadratic

\plot -3.7 0 -3 1 -2.2 0  /

\plot -0.5 0 0.2 1.3 0.7 0  /
\plot 0.7 0 1.1 1.2 1.5 0  /
\plot 1.5 0 2 1.3 2.5 0  /

\plot 3.5 0 4.45 0.8 5.15 0  /

\put {t} [lt] at 7.8 -.1
\put {$u(0,t)$} [rc] at -5.1 3

\setlinear 
\setdots <0pt> 
\plot 7.725 -.075 7.8 0 /
\plot 7.725 .075 7.8 0 /
\plot -5.075 2.825 -5 2.9 /
\plot -5 2.9 -4.925 2.825 /

\put {$0$}  [ct] at -5 -.1
\put {C}  [ct] at -4.2 0.5
\put {$^{t_1=T}$}  [ct] at -3.7 -.2
\put {L}  [ct] at -3 0.5
\put {$t_2$}  [ct] at -2.2 -.1
\put {C}  [ct] at -1.8 0.5
\put {$t_3$}  [ct] at -1.4 -.1
\put {C}  [ct] at -0.9 0.5
\put {$t_4$}   [ct] at -0.5 -.1
\put {L}  [ct] at 0.15 0.5
\put {$t_5$}   [ct] at 0.7 -.1
\put {L}  [ct] at 1.1 0.5
\put {$t_6$}   [ct] at 1.5 -.1
\put {L}  [ct] at 2 0.5
\put {$t_7$}   [ct] at 2.5  -.1
\put {C}  [ct] at 3 0.5
\put {$t_8$}   [ct] at 3.5 -.1
\put {L}  [ct] at 4.35 0.5
\put {$t_9$}   [ct] at 5.15 -.1
\put {C}  [ct] at 5.7 0.5
\put {$^{t_{10}=\tilde T}$}   [ct] at 6.5 -.1
\put {C}  [ct] at 7 0.5

\setlinear
\setdots <1.4pt> 
\plot 4.7 2.6 4.7 3.3 /
\put {:~GBU without LBC }  [ct] at 6.75 2.9

\endpicture
\] 
\caption{A solution with mixed behaviors (2 LBC, 2 GBU without LBC and 1 double bouncing)}
\label{FigMixed}
\end{figure}

\begin{remark} \label{nonmonot}
{\bf (Zero number properties)}
(i) Assume 
\begin{equation}
\label{phiinfty0} 
\| \phi \|_\infty \le c_p
\end{equation}
or $\phi$ symmetric and let $t\in \mathcal{T}$.
It follows from Proposition~\ref{zerob2} below that, if the solution changes type through time $t$, then
it must ``lose an odd 
number of intersections with $ U^*$'', more precisely, $\lim_{s\to t^-} N(s)-N(t)$ is odd. 
On the other hand, if it does not change type, then it must lose an even number of intersections with $ U^*$.
More generally this remains true provided $u(1,t)\ne c_p$. It $u(1,t)=c_p$, then
the above remains true if we replace $N$ with the number of sign changes of $u(\cdot,t)-U^*$ over $[0,b]$
 with $b<1$ close enough to $1$.
\end{remark}

\begin{remark} \label{higher-dim}
{\bf (Radial case in higher dimensions)}
Consider \eqref{vhj1} in the case that $ \Omega $ is a ball or an annulus in $\R^n$.
Let $ \Omega = (\rho,1) $ with $\rho\in(0,1)$.
Then \eqref{vhj1} for a radial solution turns out
\be\label{radial}
\begin{cases}
\displaystyle{ u_t- u_{rr} - \frac{n-1}{r} u_r =|u_r|^p }, & \quad \rho<r<1,\ t>0,\vspace{5pt} \\
u =0, & \quad r\in\{\rho, 1\},\ t>0,\\
\noalign{\vskip 1mm}
u(r,0) =\phi(r),  & \quad \rho<r<1,
\end{cases}
\ee
with $ r := |x| $. 
The singular steady state with singularity at $ r = 1 $ is given by
\be\label{Uradial}
U_{rad}^*(r) = \int_r^1 \xi^{1-n} 
\biggl[ \frac{(p-1)\bigl( \xi^{1 - (n-1)(p-1)}-1 \bigr)}{ (n-1)(p-1) - 1}  \biggr]^{ - \frac{1}{p-1} } d \xi,\quad
\rho<r<1.
\ee
It is immediate that $ U_{rad}^* $ behaves similarly to $ U^* $ near $ r = 1 $, i.e., 
\[
U_{rad}^*(r) = c_{n,p} (1-r)^\alpha + o \left( (1-r)^\alpha \right) \quad \mbox{ for } r<1 \mbox{ close to } 1
\]
with some $ c_{n,p} > 0 $.
The proofs of Theorems~\ref{thmz1} and \ref{conjz1} can be easily modified using $ U_{rad}^* $ instead of $ U^* $. 
Similar modifications work in the case $\Omega=B_1$, 
provided we assume that $\|\phi\|_\infty\le U_{rad}^*(0)$, given by formula \eqref{Uradial}.
\end{remark}

\begin{remark} \label{further-development}
{\bf (Further development)}
Our solutions given in Theorem \ref{conjz1} lose at most two intersections with $ U^*$ at each $ t \in {\mathcal T} $.
There should be other possibilities.
In addition, we do not deal with GBU rate and recovery rate at $ t \in {\mathcal T} $ in this paper.
If these questions are solved, then the behavior of global viscosity solutions will be completely clarified.
We used blow-up profile given in \cite{PS3} in the proofs of Theorems~\ref{thmz1} and \ref{conjz1}. 
However we could replace them by zero number arguments.  
These together with the iterative method based on the continuity of critical parameters in the proof of Theorem \ref{conjz1} 
will play a crucial role. 
This kind of methods will be applied to other equations and we will 
treat these issues in forthcoming papers.
\end{remark}

\subsection{Differences from Fujita equation}  

Blow-up problems in nonlinear parabolic equations have been studied the most extensively 
for the so-called Fujita equation
\be\label{fujita-eq}
\begin{cases}
u_t-\Delta u=u^q, & \quad x\in\R^n,\ t>0,\\
u(x,0)=\phi(x) \ge 0,  & \quad x\in\R^n 
\end{cases}
\ee
with $ q > 1 $.
A solution $ u $ of \eqref{fujita-eq} is said to blow up at $ t = T < \infty $ if $\displaystyle\limsup_{t \to T} \| u (t) \|_{L^\infty}$ $ = \infty$. 
There are critical values of the exponent $ q $ related to continuation after blow-up time in \eqref{fujita-eq}.
Denote by $q_S=(n+2)/(n-2)_+$ the Sobolev exponent and by $ q_{JL} $ the Joseph-Lundgren exponent, defined by 
\begin{equation}
\label{eq:JL}
q_{JL} = \left\{ 
\begin{array}{ll}
+ \infty & \quad \mbox{ if } n \leq 10, \vspace{5pt} \\
\displaystyle{ 1 + \frac{ 4 }{n-4 - 2 \sqrt{n-1} } } \;\; & \quad \mbox{ if } n \geq 11.
\end{array}
\right.
\end{equation}
When a solution $ u $ of \eqref{fujita-eq} blows up at $ t = T $, the blow-up is called complete if the proper continuation for $ t > T $ 
identically equals $ + \infty $ in  $ \R^n \times (T, \infty)  $, and incomplete otherwise.
We refer to \cite{GV97} for the definition of proper solution and its main properties. 
Roughly speaking, a proper solution is a weak solution defined as a limit of global classical solutions to an approximate equation.
From the point of view of regularization after blow-up, complete and incomplete blow-up in \eqref{fujita-eq} correspond to GBU 
with LBC and without LBC in \eqref{vhj1}, respectively.
In the case $ q < q_S $, only the complete blowup is possible by \cite{Baras-Cohen:JFA71}.
On the other hand, besides complete blow-up solutions, incomplete blow-up occurs when $ q \ge q_S $ 
\cite{Ni-Sacks-Tavantzis:JDE54, Mizoguchi:MZ239}.
We note that these works do not give information on the behavior of weak solutions after blow-up time.
The simplest incomplete blow-up solution is a peaking solution and its existence is known from
 \cite{GV97, Mizoguchi:JFA220} (see also \cite{FMP05,MM09}).
Here a global weak solution $u$ is called a peaking solution if $ u $  blows up at $ t = T < \infty $, recovers regularity just 
after $ t = T $ and exists for $ t > T $ in the classical sense. 
The GBU solution of \eqref{vhj1} without LBC constructed in \cite{PS3} is similar to a peaking solution of \eqref{fujita-eq}.
The natural question on the existence of a solution with multiple blow-up times has been paid attention in the field of the semilinear heat 
equation and was  answered affirmatively for $ q > q_{JL} $ in \cite{Miz1,Miz2}. 
Namely, if $ q > q_{JL} $, then for arbitrary $ m \in \N^* $ there exists a radial solution $ u $ 
of \eqref{fujita-eq} and $ 0 < t_1 < t_2 < \cdots < t_m $ such that $ u $ blows up at $ t = t_i $ and 
that $ u $ is classical in $ (t_{i-1}, t_i) $ for $ i = 1, 2, \cdots, m $, where $ t_0 := 0 $.
The radial solution with multiple blow-up times was constructed in \cite{Miz1,Miz2}
by carefully choosing initial data with $ m $ isolated parts such that 
each part corresponds to initial data of the special radial solution by \cite{Herrero-Velazquez:CRASP319} blowing up around $ t = t_i $ 
for $ i = 1, 2, \cdots, m $.
While the solutions due to \cite{Miz1,Miz2} blow up only at the origin at each blow-up time, a solution 
of different type, which undergoes incomplete blow-up at the origin at the first blow-up time and 
complete blow-up on a sphere at the second blow-up time, was constructed in \cite{Miz-Vaz}.

Whereas complete blow-up solutions of the Fujita equation \eqref{fujita-eq} cease to exist 
 in any weak sense at blow-up time, solutions of \eqref{vhj1} undergoing GBU with LBC 
continue to exist in viscosity sense after blow-up time and recover regularity after a while.
This makes the description of behavior of viscosity solutions after blow-up time more complicated than in \eqref{fujita-eq}.
In other words, when a radial solution of \eqref{fujita-eq} undergoes blow-up at multiple blow-up times, 
only the immediate regularization is possible except at final blow-up time.
On the contrary, various combinations of GBU with LBC and without LBC 
turn out to occur in \eqref{vhj1}, including the new type of behavior called bouncing.
Accordingly, as we shall explain in the next subsection, 
we introduce a method based on an arbitrary number of critical parameters, whose continuity 
requires a delicate argument.
Since we do not rely on any known special solution unlike in the proof for Fujita equation, 
our method is expected to apply to other equations.

\subsection{Ideas of proofs}
(i) The main idea of the proof of Theorem~\ref{thm1} is as follows. 
We construct a multiscale, compactly supported initial data, made of $m$, suitably rescaled, bumps 
which are located farther and farther from the boundary.
The distribution of the sizes and locations of the bumps in terms of the distance to the boundary is rather delicate
and the construction has to be made in a recursive way.
The bump which is closest to the boundary generates the first GBU and LBC.
The second bump is much farther from the boundary and its influence 
becomes significant only after some lapse of time, leaving enough time for the solution to get regularized by 
an effect of the diffusion, before producing a second GBU and LBC. 
Repeating the process, we construct an $m$-bump initial
data which leads to a solution undergoing GBU and LBC (at least) $m$ times.\footnote{Although the above 
description is more convenient for heuristic purposes,
the actual construction is done in the converse direction, first constructing the bump which is farthest from the boundary.}
The estimates leading to each LBC and regularization are provided by suitable rescaling and comparison arguments,
partly inspired in \cite{LS} and \cite{PZ},
that are the object of the preliminary lemmas in Section~3.
As for the proof of Theorem~\ref{thm1B}, it is based on a modification of the proof of Theorem~\ref{thm1} 
and a limiting argument.

(ii) The basic idea of the proof of Theorem~\ref{thmz1}(i) is to show that $N(t)$ is a nondecreasing function for $t\ge 0$
and that $N(t)$ has to drop at any transition time $t\in \mathcal{T}$.
These properties can be shown to be true provided one makes the additional assumption \eqref{phiinfty0}
(see Propositions~\ref{ZeroMonot} and \ref{zerob2}) and this then readily yields the bound \eqref{cardS}.
In the general case, without assumption \eqref{phiinfty0}, $N(t)$ need not be monotone
 (see Proposition~\ref{PropNonmonot})
and the proof requires additional arguments.
In particular, one needs to study the intersection properties of approximate solutions $u_k$ with regular steady states $U_j$
for suitable large $j,k$.
As for the proof of Theorem~\ref{thmz1}(ii) it is based on a further analysis of intersection properties of the solution with $U^*$,
combined with asymptotic estimates of the boundary profile of the solution at times $t\in \mathcal{S}$.

(iii) The proof of Theorem~\ref{conjz1} is rather long and delicate. 
It is based on a modification of the multibump construction in the proof of Theorems~\ref{thm1}
combined with deformation, zero number and recursion arguments.
More precisely we construct a multi-parameter family of initial data
based on suitable deformations of multibumps initial data, and the desired solution is obtained by iteratively selecting appropriate critical values of the parameters. 
The required continuity properties of the critical parameter functions rely upon zero number arguments
applied to the difference of two solutions, whereas the exact structure of the resulting solution
depends on dropping properties of the number of intersections with $ U^* $.
In the construction stated in (i), the behavior of solutions corresponding to each bump of initial data is rather stable, and does not 
give a serious effect to the others.
On the other hand, the behaviors associated with the critical parameters are unstable and sensitive to the other parts.
That is a reason why the proof is long and complicated.

\goodbreak

\section{Preliminary results I:  approximation, comparison and regularity}

Throughout the paper, we shall denote the function distance to the boundary by
$$\delta(x)={\rm dist}(x,\partial\Omega),\quad x\in\Omega,$$
and by $\nu$ the outer normal vector to $\partial\Omega$.

 We start by recalling (see,~e.g.,~\cite{PZ,PS3}) that \eqref{vhj1} 
can be approximated (away from the boundary) by the truncated problems
\be\label{vhj1k}
\begin{cases}
\partial_tu_k-\Delta u_k=F_k(|\nabla u_k|^2), & \quad x\in\Omega,\ t>0,\\
u_k=0, & \quad x\in\partial\Omega,\  t>0,\\
u_k(x,0) =\phi(x),  & \quad x\in\Omega, 
\end{cases}
\ee
where 
\be\label{DefFk}
F_k(s)=\min(s^{p/2},k^{(p-2)/2}s).
\ee
Namely, \eqref{vhj1k} admits a unique, global classical solution $u_k\in C^{1,2}(\overline\Omega\times(0,\infty))$
with $u_k,\nabla u_k\in C(\overline\Omega\times[0,\infty))$,
the sequence $u_k$ is nondecreasing, and
$$u_k\to u\ \hbox{ in $C^{2,1}_{loc}(\Omega\times(0,\infty))$, as $k\to\infty$.}$$

 We next give two versions of the comparison principle that will be used repeatedly in what follows.
 The first one is a standard comparison principle for sub-/super\-solutions (see e.g. \cite[Proposition 2.1]{SZ06})

\begin{proposition}\label{compP0}
Let $\omega$ be any bounded open subset 
of $\Rn$. Let $0\le t_1<t_2$, set $Q:=\omega\times(t_1,t_2)$
and denote by $\partial_PQ$ its parabolic boundary.
Assume that $u_i\in C^{2,1}(Q)\cap C(\overline Q)$, $i=1,2$, satisfy
$$\partial_t u_1-\Delta u_1-|\nabla u_1|^p
\leq \partial_t u_2-\Delta u_2-|\nabla u_2|^p\quad\hbox{ in $Q_T$,}$$
Then
$$\sup_{Q}(u_1-u_2) \leq \sup_{\partial_PQ} (u_1-u_2).$$
\end{proposition}

\medskip

The second one is a suitable form of the comparison principle for viscosity solutions. A key point is that comparison is guaranteed
 although we only assume $w\ge 0$ (in the classical sense) on  $\partial\Omega$ for the comparison function from above,
 whereas the viscosity solution $u$ may take positive pointwise values on the boundary
 (but of course it satisfies $u=0$ in the generalized viscosity sense or,
 equivalently, in the sense of monotone approximation by classical solutions of truncated problems).

\begin{proposition}\label{compP}
Let $\phi\in X$ and let $u$ be the corresponding global viscosity solution of \eqref{vhj1}.
Let $\omega$ be any open subset of $\Omega$, let $t_2>t_1\ge0$ and set $Q:=\omega\times(t_1,t_2)$.
Assume that $w\in C^{1,2}(Q)\cap C(\overline Q)$ satisfies
$$\begin{aligned}
w_t-\Delta w&\ge |\nabla w|^p\quad\hbox{ in $Q$,} \\
w&\ge 0 \quad\hbox{ on $(\omega\cap\partial\Omega)\times(t_1,t_2)$,} \\
w&\ge u \quad\hbox{ on $(\Omega\cap\partial\omega)\times(t_1,t_2)$,} \\
 w(\cdot,t_1) &\ge u(\cdot,t_1)\quad\hbox{ in $\omega$.}
\end{aligned}$$ 
Then $w\ge u$ in $\overline{Q}$.
 \end{proposition}
 
For convenience, we give a short proof.

\begin{proof} 
 Consider the smooth solutions $u_k$ of the truncated problems \eqref{vhj1k}.
Since $u\ge u_k$, our assumptions imply that $w\ge u_k$ on the parabolic boundary of $Q$.
Since $F_k(|\nabla v|^2)\le |\nabla v|^p$, it follows from Proposition~\ref{compP} that
$w\ge u_k$ on $Q$.
For each $(x,t)\in Q$, passing to the limit $k\to\infty$, we deduce that
$w(x,t)\ge u(x,t)$. 
Finally, since $u,w\in C(\overline\Omega\times[0,\infty))$, this extends to $\overline Q$.
\end{proof}

The solution $u$ satisfies the following continuous dependence estimate with respect initial data
(see e.g.~\cite[Theorem~3.1]{PS3}):
\be\label{contdep}
\|u(\phi_1;t)-u(\phi_2;t)\|_\infty\le \|\phi_1-\phi_2\|_\infty, \quad\hbox{ for all $t>0$.}
\ee
As another useful maximum principle property, let us recall (see, e.g.,~\cite[Lemma 3.1]{PS2} and \cite[Section~3]{PS2}) 
that for all $t_0>0$, 
there exist $M=M(t_0)>0$ and $k_0\ge 1$ such that
\be\label{Liptime0}
|u_t|\le M\quad\hbox{ in $\Omega\times [t_0,\infty)$}
\ee
 and the solutions $u_k$ of the truncated problems \eqref{vhj1k} satisfy
\be\label{Liptimek}
|\partial_tu_k|\le M \quad\hbox{ in $\Omega\times [t_0,\infty)$ for all $k\ge k_0$}.
\ee
Since $u\in C(\overline\Omega\times[0,\infty))$, \eqref{Liptime0} guarantees the following global Lipschitz estimate in time up to the boundary
\be\label{Liptime}
|u(x,t)-u(x,s)|\le M(t-s),\quad x\in\overline\Omega,\ t_0\le s<t<\infty.
\ee

We next recall (see \cite{BDL04}) that the existence of a unique global viscosity solution $u$ of~\eqref{vhj1},
with $u\in C^{1,2}(\overline\Omega\times(0,T))\cap C(\overline\Omega\times[0,\infty))$,
remains true for any initial data $\phi\in C_0(\Omega)$.
However, unlike in the case of $C^1$ initial data (namely $u_0\in X$), $u$ need then not be smooth up to the boundary,
nor achieve the boundary conditions in the classical sense, even for small time.
We first give the following proposition, which guarantees classical regularity for small time 
provided the initial data $\phi\in C_0(\Omega)$ is merely controlled above by the distance to the boundary,
even though $\phi$ is not~$C^1$.
This result will be a useful tool in the proof of Theorem~\ref{thm1},
to ensure that the solution becomes classical again between two losses of boundary conditions.
Set
$$Y:=\Bigl\{\phi\in C(\overline\Omega);\ \phi=0 \hbox{ on $\partial\Omega$ \ and \ $\displaystyle\sup_{\Omega}\frac{\phi}{\delta}<\infty$}\Bigr\}.$$

\begin{proposition}\label{prop1} 
Let $\phi\in Y$. There exists $\tau>0$ such that problem~\eqref{vhj1} admits a unique classical solution $U$ on $[0,\tau]$,
with $U\in C^{1,2}(\overline\Omega\times(0,\tau])\cap C(\overline\Omega\times[0,\tau])$.
Moreover, we have
\be\label{estimnablau}
\sup_{t\in (0,\tau]} t^{1/2} \|\nabla U(t)\|_\infty<\infty,
\ee
and $U$ coincides with the unique viscosity solution $u$ on $[0,\tau]$.
\end{proposition}

\begin{remark}\label{remprop1} 
The time $\tau$ can be chosen uniform in terms of 
\be\label{defMphi}
M:=\max\Bigl\{\displaystyle\sup_{\Omega}\delta^{-1}\phi,\bigl|\inf_{\Omega}\phi\bigr|\Bigr\}.
\ee
\end{remark}

Altough Proposition~\ref{prop1} does not seem to have appeared in the literature,
it is essentially a consequence of Bernstein-type and approximation arguments from \cite{PZ}.
Since the motivation in \cite{PZ} was ultimate regularization for large time,
we need to adapt them to our present purpose and we give a full proof for completeness.

\begin{proof}[Proof of Proposition~\ref{prop1}]
Let $\psi\in C^2(\overline\Omega)$ be the unique solution of the linear problem
\be\label{defpsi}
\begin{cases}
-\Delta\psi=1, & \quad x\in\Omega,\\
\psi=0, & \quad x\in\partial\Omega.
\end{cases}
\ee
By Hopf's lemma, there exists $c_1>0$ such that 
\be\label{hopfpsi}
\psi\ge c_1\delta\quad\hbox{ in $\Omega$.}
\ee
By assumption, we have
$\phi\le M\delta\le c_1M\psi$, where $M$ is given by \eqref{defMphi}.
Let $w$ be the solution of problem~\eqref{vhj1} with initial data $c_1M\psi\in X$.
There exist $\tau>0$ and $C_1>0$ depending only on $M, p, \Omega$, such that 
$w$ is classical on $(0,\tau]$ and
\be\label{supgradw}
\sup_{t\in (0,\tau]}\|\nabla w(t)\|_\infty\le C_1.
\ee

As in the proof of Proposition~\ref{compP}, we next introduce a sequence of truncated problems,
this time with slightly better behaved truncated nonlinearities.
For each integer $k\ge 1$, let $F_k$ be given by
\be\label{eqaaa}
F_k(s)=\begin{cases}
s^{p/2},\quad 0\le s\le k, \\
\noalign{\vskip 1mm}
    k^{p/2}+\frac{p}{2}k^{(p-2)/2}(s-k),\quad s>k.  
\end{cases}
\ee
Note that $F_k\in C^1([0,\infty))$, with $F_k'$ locally H\"older continuous, 
\be\label{eqaa0}
0\le F_k(s)\le |s|^{p/2},\quad s\ge 0
\ee
and
\be\label{eqaa}
2sF'(s)\ge F(s),\quad s\ge 0.
\ee
Since $F_k(|\nabla v|^2)$ has at most quadratic growth at infinity with respect to $\nabla v$,
it is well known that problem \eqref{vhj1k} admits a unique, global classical solution $v=v_k$.
By the maximum principle, we have 
\be\label{estim0vhj1k}
|v|\le \|\phi\|_\infty
\ee
and, in view of \eqref{eqaa0}, it follows from the comparison principle that
$$v\le w\quad\hbox{ in $Q:=\Omega\times(0,\tau)$}.$$
Consequently, using \eqref{supgradw}, we get
\be\label{estim1vhj1k}
\frac{\partial v}{\partial\nu}\ge -C_1\quad\hbox{ in $Q$}.
\ee
On the other hand, we have $v\ge z:=e^{t\Delta}(\phi_-)$,
where $\phi_-=\min(\phi,0)$.
Since, by standard heat semigroup estimate, $\|\nabla z(t)\|_\infty\le C(\Omega)t^{-1/2}\|\phi_-\|_\infty\le C(\Omega)Mt^{-1/2}$, we deduce that
\be\label{estim2vhj1k}
\frac{\partial v}{\partial\nu}\le C(\Omega)t^{-1/2}M\quad\hbox{ in $Q$}.
\ee

Now, following \cite{PZ}, in order to obtain gradient estimates for $v=v_k$ away from $t=0$, uniformly in $k$,
we introduce the following Bernstein-type auxiliary function:
$$J=t|\nabla v|^2+(v+\|\phi\|_\infty)^2.$$
By \eqref{estim0vhj1k}-\eqref{estim2vhj1k}, we have 
$$J\le C_2:=4\|\phi\|_\infty^2+ \max(C^2(\Omega)M^2,C_1^2\tau)
\quad\hbox{ on $\partial_P Q$}$$
(where $\partial_P$ denotes the parabolic boundary).
We compute, in $Q$:
$$J_t=|\nabla v|^2+2t\nabla v\cdot \nabla v_t+2(v+\|\phi\|_\infty)v_t$$
and, using $\Delta(|\nabla v|^2)=2\nabla v\cdot \nabla \Delta v+2\sum_{ij}(v_{ij})^2$, 
$$\Delta J\ge 2t\nabla v\cdot \nabla \Delta v+2|\nabla v|^2+2(v+\|\phi\|_\infty)\Delta v.$$
Therefore,
$$
\begin{aligned}
J_t-\Delta J
&\le -|\nabla v|^2+2t\nabla v\cdot \nabla (F(|\nabla v|^2))+2(v+\|\phi\|_\infty)F(|\nabla v|^2).
\end{aligned}$$
Since
$$
\begin{aligned}
2t\nabla v\cdot \nabla (F(|\nabla v|^2))
&=4t\sum_i v_i\nabla v\cdot \nabla v_iF'(|\nabla v|^2)
=2\nabla v\cdot \nabla(t|\nabla v|^2)F'(|\nabla v|^2)\\
&=2F'(|\nabla v|^2)\nabla v\cdot \nabla J-4(v+\|\phi\|_\infty)|\nabla v|^2F'(|\nabla v|^2),
\end{aligned}$$
it follows that
$$J_t-\Delta J-2F'(|\nabla v|^2)\nabla v\cdot \nabla J
\le -|\nabla v|^2+2(v+\|\phi\|_\infty)\bigl[F(|\nabla v|^2)-2|\nabla v|^2F'(|\nabla v|^2)\bigr].$$ 
Using \eqref{eqaa}, \eqref{estim0vhj1k}, we deduce that
$$J_t-\Delta J-2F'(|\nabla v|^2)\nabla v\cdot \nabla J\le 0.$$
It follows from the maximum principle that
\be\label{controlJ}
\sup_Q J\le \sup_{\partial_P Q} J\le C_2.
\ee
Since $C_2$ independent of $k$, this guarantees that the sequence $(\nabla v_k)$ is uniformly bounded on 
$Q_\eps:=\Omega\times[\eps,\tau]$ for each $\eps\in (0,\tau)$,
and it follows from parabolic estimates
 that $(v_k)$ is relatively compact in $C^{1,2}(\overline{Q_\eps})$.
By a diagonal procedure, we deduce that some subsequence of $v_k$ converges 
to a classical solution $U$ of \eqref{vhj1}$_1$-\eqref{vhj1}$_2$ in $\overline\Omega\times(0,\tau]$, with convergence 
in $C^{1,2}(\overline{Q_\eps})$ for each $\eps\in (0,\tau)$.
Moreover, for each $t\in (0,\tau]$, \eqref{controlJ} implies $t\|\nabla U(\cdot,t)\|_\infty^2\le \limsup_k t \|\nabla v_k(\cdot,t)\|_\infty^2\le C_2$, 
hence estimate \eqref{estimnablau}.

It remains to show that 
\be\label{contU}
U\in C(\overline\Omega\times[0,\tau])\quad\hbox{ with $U(0)=\phi$.}
\ee
Since $\phi\in C_0(\Omega)$, for each $\eta>0$, we may pick a function $\psi_\eta\in C^\infty_0(\Omega)$
such that $\|\phi-\psi_\eta\|_\infty\le \eta$.
Setting $K=\sup_{\Omega} (\Delta\psi_\eta+|\nabla \psi_\eta|^q)$ and $\overline v(x,t)=\psi_\eta+\eta+Kt$, 
and using \eqref{eqaa0}, we have
$$\overline v_t-\Delta \overline v -F_k(|\nabla \overline v|^2)=
K-\Delta\psi_\eta-F_k(|\nabla \psi_\eta|^2)\ge K-\Delta\psi_\eta-|\nabla \psi_\eta|^p \ge 0.$$
Since $\overline v(x,0)=\psi_\eta+\eta\ge \phi$, it follows from the comparison principle that 
$$e^{t\Delta}\phi\le v_k\le \psi_\eta+\eta+Kt\le \phi+2\eta+Kt.$$
Consequently,
$\|v_k(\cdot,t)-\phi\|_\infty\le 3\eta$ for $t>0$ small,
hence \eqref{contU}.

Finally, since $U$ is in particular a viscosity solution, it has to coincide with the unique viscosity solution $u$.
This in turn guarantees the uniqueness of $U$.
The proof of Proposition~\ref{prop1} is complete.
\end{proof} 

\section{Preliminary results II: LBC and regularity properties in time based on local properties of initial data}

For each $\eps>0$, we set
$$\Omega_\eps=\bigl\{x\in\Omega;\ \delta(x)<\eps\bigr\},\quad
\Omega^\eps=\bigl\{x\in\Omega;\ \delta(x)>\eps\bigr\},\quad
\Gamma_\eps=\bigl\{x\in\Omega;\ \delta(x)=\eps\bigr\}$$
(observing that $\Omega_\eps=\Omega$ and $\Omega^\eps,\Gamma_\eps=\emptyset$ for $\eps>0$ sufficiently large).
 The goal of this section is to establish the following two propositions,
 which are important ingredients for the proofs of Theorems~\ref{thm1}, \ref{thm1B} and \ref{conjz1}. 
 
In our first proposition, we show that LBC in short time occurs for suitable initial data,
with support concentrated near the boundary and precisely controled, small $L^\infty$ norm.
We also give a useful, slightly more precise version in the case $n=1$. 

\begin{proposition}\label{lem2} 
(i) There exists $\eps_0\in (0,1)$ such that, for all $\eps\in (0,\eps_0)$, 
there exist a compactly supported $\psi_\eps\in X$ 
with the following properties:
\be\label{suppcond}
{\rm Supp}(\psi_\eps)\subset \Omega_{2\eps}\setminus\Omega_\eps
\ee
\be\label{suppsi}
\|\psi_\eps\|_\infty\le K_1\eps^\alpha,
\ee
\be\label{Tpsieps}
T(\psi_\eps)<\tau_\eps\,=c_0\eps^2, 
\ee
\be\label{LBC0}
\sup_{x\in\partial\Omega} u(\psi_\eps;x,\tau_\eps)>0,
\ee
where $K_1,c_0>0$ are constants independent of $\eps$.

(ii) Let $n=1$, $\Omega=(0,1)$. 
There exist $c_2>c_1>0$ and $C_0=C_0(p)>0$ and, for each $\eta\in (0,1)$, 
there exists $\tilde K=\tilde K(\eta)>0$ such that, for any $\eps\in (0,1/2)$, 
if 
\be\label{suppcondn1}
\phi\ge \tilde K\eps^\alpha\quad\hbox{on $((1-\eta)\eps,(1+\eta)\eps)$,}
\ee
 then
\be\label{LBC0n1}
u(0,t)\ge C_0\eps^\alpha\quad\hbox{ for all $t\in [c_1\eps^2,c_2\eps^2]$.}
\ee
\end{proposition}

Our second proposition gives a regularity property for the solution at suitable times, 
depending on the behavior of the initial data at given distance from the boundary.

\begin{proposition}\label{lem3bis} 
Let $\phi\in X\cap C^2(\Omega)$. There exists a constant $L=L(\Omega,p)>0$ and for each $\eps,M>0$, there exists a constant
$\gamma_0=\gamma_0(\Omega,p,\eps,M)>0$
with the following property.
If for some function $H\in C^2(\overline\Omega)$, $H\ge 0$, such that
\be\label{hyplem3a00H}
\|H\|_{C^2(\overline\Omega^{\eps/2})}\le M
\ee
and some constant $\gamma\in [0,\gamma_0)$ we have
\be\label{hyplem3a00}
\phi\le 
\begin{cases}
H &\hbox{ in $\Omega^\eps$,} \\
\noalign{\vskip 1mm}
\gamma &\hbox{ in $\Omega_\eps$,}
\end{cases}
\ee
then $u$ is a classical solution for $t\in [L\gamma,L\gamma_0]$. 
\end{proposition}

\begin{proof} [Proof of Proposition~\ref{lem2}.]
Consider the following problem
\be\label{vhj1b}
\begin{cases}
v_t-\Delta v=|\nabla v|^p, & \quad x\in B_1,\ t>0,\\
v =0, & \quad x\in\partial B_1,\ t>0,\\
v(x,0) =v_0(x),  & \quad x\in B_1, 
\end{cases}
\ee
where $0\le v_0\in C^\infty_0(\R^n)$ is radially symmetric with ${\rm Supp}(v_0)\subset B_{1/4}$.
Denote by $v$ the unique global viscosity solution of \eqref{vhj1b} and by $T_0:=T(v_0)$ 
its maximal existence time as a classical solution.
After replacing $v_0$ by a sufficiently large multiple, it follows from \cite[Theorem~4]{PS2} that $T_0<\infty$ and 
that there exists $t_0>T_0$ such that
\be\label{gbuv}
\sup_{x\in\partial B_1} v(x,t_0)>0.
\ee

Next pick some point $a\in \Omega$, let $b$ be its projection onto $\partial \Omega$, set $d=|a-b|$
and let $e:=d^{-1}(a-b)$ be the inward unit normal vector to $\partial\Omega$ at the point $b$.
For each $\eps\in(0,d)$, we set $a_\eps=b+\eps e$.
Then $D_\eps:=B(a_\eps,\eps)\subset \Omega$ and $b\in\partial\Omega$.

Now, following \cite{LS}, we consider the following rescaled functions:
\be\label{reslem3a}
\phi_\eps(x):=\eps^\alpha v_0(\eps^{-1}(x-a_\eps))\ge 0,\quad x\in\R^n
\ee
(which is supported in $D_\eps$) and 
$$v_\eps(x,t)=\eps^\alpha v\bigl(\eps^{-1}(x-a_\eps),\eps^{-2}t\bigr),\quad (x,t)\in \overline D_\eps\times (0,\infty].$$
Then $v_\eps$ solves problem~\eqref{vhj1} with $\Omega$ replaced by $B(a_\eps,\eps)$
and initial data $\phi_\eps$.
Let $u=u(\phi_\eps;\cdot,\cdot)$ be the solution of problem~\eqref{vhj1} (in $\Omega$) with initial data $\phi=\phi_\eps$.
By Proposition~\ref{compP}, using $u\ge 0$, we get
\begin{equation*}
    u\geq v_\eps\ \ \ {\rm in }\ D_\eps  \times (0,\infty).
\end{equation*}
Since $v$ is radially symmetric, it follows from \eqref{gbuv} that
\be\label{reslem3b}
u(\phi_\eps;b,\eps^2t_0)\ge v_\eps(b,\eps^2t_0)=\eps^\alpha v\bigl(\eps^{-1}(b-a_\eps),t_0\bigr)
=\eps^\alpha v\bigl(-e,t_0\bigr)>0.
\ee
On the other hand, we note that 
\be\label{reslem3c}
\begin{aligned}
\phi_\eps(x)\ne 0 
&\Longrightarrow |\eps^{-1}(x-a_\eps)|<1/4 \\
&\Longrightarrow |\delta(x)-\eps|=|\delta(x)-\delta(a_\eps)|\le |x-a_\eps|<\eps/4\\
&\Longrightarrow \frac{3\eps}{4}<\delta(x)<\frac{5\eps}{4}.
\end{aligned}
\ee
Finally setting $\psi_\eps=\phi_{4\eps/3}$, we deduce from \eqref{reslem3a}-\eqref{reslem3c} that
 \eqref{suppcond}-\eqref{LBC0} are satisfied with 
$$c_0=\frac{16}{9}t_0,\qquad K_1=\Bigr(\frac{4}{3}\Bigl)^\alpha\|v_0\|_\infty,$$
and assertion (i) is proved with $\eps_0=\frac43 d$.

(ii) The argument is similar, now considering the problem
\be\label{vhj1bn1}
\begin{cases}
v_t-v_{xx}=|v_x|^p, & \quad x\in (0,2),\ t>0,\\
v =0, & \quad x\in\{0,1\},\ t>0,\\
v(x,0)=K\psi(x),  & \quad x\in (0,2), 
\end{cases}
\ee
where $K>0$ and $\psi\in C^\infty_0(\R)$ is symmetric with respect to $x=1$, nontrivial with ${\rm Supp}(\psi)\subset(1-\eta,1+\eta)$
and $0\le \psi\le 1$.
Denote by $v$ the unique global viscosity solution of \eqref{vhj1b} and by $T_0:=T(v_0)$ 
its maximal existence time as a classical solution.
By \cite[Theorem~4]{PS2}, 
we may take $K= \tilde K(\eta)>0$ sufficiently large, such that $T_0<\infty$ and
that there exist  $c_2>c_1>T_0$ and $C_0>0$ such that
$v(0,s)>C_0$ for all $s\in [c_1,c_2]$.
Now consider
$$v_\eps(x,t)=\eps^\alpha v\bigl(\eps^{-1}x,\eps^{-2}t\bigr),\quad (x,t)\in [0,2\eps]\times [0,\infty).$$ 
Then $v_\eps$ solves problem~\eqref{vhj1} with $(0,2)$ replaced by $(0,2\eps)$
and initial data $v_\eps(x,0):=K\eps^\alpha\psi(\eps^{-1}x)$.
Under assumption \eqref{suppcondn1}, since ${\rm Supp}(v_\eps(\cdot,0))\subset((1-\eta)\eps,(1+\eta)\eps)$
and $v_\eps(\cdot,0)\le K\eps^\alpha$, we have $\phi\ge v_\eps(\cdot,0)$ in $(0,2\eps)$.
By Proposition~\ref{compP}, using $u\ge 0$, it follows that
$u\geq v_\eps$ in $[0,2\eps]\times (0,\infty)$,
hence
$u(0,s\eps^2)\ge v_\eps(0,s\eps^2)= \eps^\alpha v\bigl(0,s\bigr)=C_0\eps^\alpha$ for all $s\in [c_1,c_2]$.
\end{proof} 

For the proof of Proposition~\ref{lem3bis}, we need two lemmas.
Our first simple lemma gives a pointwise control from above with linear growth in time.

\begin{lem}\label{lem1} 
Let $h\in X\cap C^2(\overline\Omega)$ and set
\be\label{deflambda}
\lambda=\lambda_h:=\displaystyle\sup_{\overline\Omega} \bigl[\Delta h+|\nabla h|^p\bigr].
\ee
For any $\gamma\ge 0$, if  $\phi\in X$ satisfies $\phi\le h+\gamma$ in $\Omega$, then
$$u\le h + \gamma+\lambda t\quad\hbox{ in $\Omega\times(0,\infty)$.}$$ 
\end{lem}

\begin{proof} 
 Set $\overline u(x,t)=h(x)+ \gamma+\lambda t$. We have
$$P\overline u:=\overline u_t-\Delta \overline u-|\nabla \overline u|^p=\lambda-\Delta h-|\nabla h|^p\ge 0
\quad\hbox{ in $\overline\Omega\times(0,\infty)$.}$$
Since also $\overline u(0,\cdot)\ge \phi$ by assumption and $\overline u\ge 0$ on $\partial\Omega\times(0,\infty)$,
the conclusion follows from Proposition~\ref{compP}.
\end{proof}

Our second lemma shows that the solution becomes classical again after some short time,
provided it is suitably controled in amplitude\footnote{We stress that, unlike
in more standard criteria, we do not assume any control of the {\it gradient} of the solution
near the boundary.}
near the boundary. It is a localized version (in space near the boundary and in time)
of an observation from~\cite{PZ}.

\begin{lem}\label{lem3} 
There exist $K_2,K_3>0$ depending only on $\Omega, p$ such that if,
for some $\eps, \eta, s>0$, 
\be\label{hyplem3a}
u(\cdot,s)\le \eta\quad\hbox{ in $\Omega_\eps$,}
\ee
and (in case $\Gamma_\eps\neq\emptyset$)
\be\label{hyplem3}
u\le K_2\eps\quad\hbox{ in $\Gamma_\eps\times[s,\hat s]$,\ where $\hat s=s+K_3\eta$,}
\ee
then
$$\hbox{$u$ is a classical solution for $t>\hat s$ close to $\hat s$.}$$
\end{lem}

\begin{proof} 
We modify a comparison argument from \cite{PZ} (see \cite[Lemma 3.2]{PZ}).
Let $\psi\in C^2(\overline\Omega)$ be the unique solution of the linear problem \eqref{defpsi}
and let $c_1>0$ be the constant of inequality \eqref{hopfpsi}.
Let $\kappa=(2\|\nabla\psi\|_\infty)^{-p/(p-1)}$, $K_3=\kappa^{-1}$, $\hat s=s+K_3\eta$,
and define
$$w(x,t)=\eta+\kappa\bigl[2\psi(x)+s-t\bigr],\quad (x,t)\in \overline Q, \quad\hbox{ with } Q:=\Omega_\eps\times(s,\hat s].$$
We compute
$$w_t-\Delta w-|\nabla w|^p=-\kappa+2\kappa-\bigl(2\kappa|\nabla\psi|\bigr)^p\ge 0\quad\hbox{ in $Q$.}$$ 
On the other hand, by the choice of $\hat s$, we have $w\ge 2\kappa\psi$ in $Q$, 
hence in particular $w\ge 0$ on $\partial\Omega\times(s,\hat s)$
and,  using \eqref{hopfpsi}, $w\ge 2\kappa c_1\eps$ on $\Gamma_\eps\times(s,\hat s)$.
 Moreover we have 
$w(x,s)\ge \eta$ in $\Omega_\eps$.
Choosing $K_2=2\kappa c_1$,  in view of assumptions \eqref{hyplem3a}-\eqref{hyplem3},
 it follows from Proposition~\ref{compP} that $u\le w$ in~$Q$. 
In particular, we have
$u(\cdot,\hat s)\le 2\kappa\psi$ in $\Omega_\eps$.
Since also $u\le \|\phi\|_\infty$ 
 in $\overline\Omega\times[0,\infty)$ by Proposition~\ref{compP},  
using \eqref{hopfpsi} we get
$$u(\cdot,\hat s)\le C_1\psi\quad\hbox{ in $\Omega$,\quad with }  C_1:=\max\bigl(2\kappa, (c_1\eps)^{-1}\|\phi\|_\infty\bigr).$$
It then follows from Proposition~\ref{prop1}  
 and uniqueness of the viscosity solution 
that $u$ is a classical solution for $t>\hat s$ close to~$\hat s$.
\end{proof}

\begin{proof} [Proof of Proposition~\ref{lem3bis}.]
Fix $\theta_\eps\in C^2(\overline\Omega)$ such that $\theta_\eps=0$ in $\overline\Omega_{\eps/2}$,
 $\theta_\eps=1$ in $\overline\Omega^\eps$ and $0\le \theta_\eps\le 1$.
 Setting $h=\theta_\eps H$, assumption \eqref{hyplem3a00} guarantees that
$$\phi\le \gamma+h \quad\hbox{ in $\Omega$.}$$
It follows from Lemma~\ref{lem1} that
$$u(x,t)\le \gamma+h(x)+\lambda t\quad\hbox{ in $\Omega\times(0,\infty),$}$$
where $\lambda=\lambda_h$ is given by \eqref{deflambda},
hence in particular
\be\label{uxtgammalambda}
u(x,t)\le \gamma+\lambda t\quad\hbox{ in $\overline\Omega_{\eps/2}\times(0,\infty).$}
\ee
 Let
\be\label{defgamma0}
\overline\gamma:=\textstyle\frac14 \min\bigl\{1,(\lambda K_3)^{-1}\bigr\}K_2\eps,
\ee
where $K_2,K_3$ are given by Lemma~\ref{lem3}, and note that $\overline\gamma$ is uniformly positive
for $\eps$ fixed and $\|H\|_{C^2(\overline\Omega^{\eps/2})}$ bounded.
Assume 
\be\label{condgamma1}
0<\gamma\le  \overline\gamma 
\ee
and take any $\tau$ such that
\be\label{condgamma2}
K_3\gamma\le\tau\le \textstyle K_3\overline\gamma.
\ee
Set $s=(1+K_3\lambda)^{-1}(\tau-K_3\gamma)\in [0,\tau)$ and $\eta=\gamma+\lambda s$,
hence 
\be\label{condgamma3}
s+K_3\eta =\tau.
\ee
By \eqref{uxtgammalambda} we have
\be\label{condeta1}
u(\cdot,s)\le \eta\quad\hbox{ in $\Omega_{\eps/2}.$}
\ee
Moreover,  since \eqref{defgamma0}-\eqref{condgamma2} 
guarantee that $\tau\le \frac14 \lambda^{-1}K_2\eps\le \lambda^{-1}(\frac12 K_2\eps-\gamma)$, hence  $\gamma+\lambda\tau\le K_2\eps/2$,
 inequality \eqref{uxtgammalambda} also implies 
 \be\label{uxtgammalambda2}
u(x,t)\le \gamma+\lambda \tau\le K_2\eps/2\quad\hbox{ in $\Gamma_{\eps/2}\times[s,\tau].$} 
\ee
It then follows  from \eqref{condgamma3}-\eqref{uxtgammalambda2} 
and Lemma~\ref{lem3} that $u$ is a classical solution for $t>\tau$ close to $\tau$.
Since this is true for any $\tau$ satisfying \eqref{condgamma2}, 
we have shown that $u$ is classical on $(K_3\gamma,K_3\overline\gamma]$.

Finally set $\gamma_0=\overline\gamma/2$ and $L=2K_3$. If $\phi$ satisfies \eqref{hyplem3a00} with $\gamma\in[0,\gamma_0)$, 
then the above implies that $u$ is classical on $(K_3\gamma,K_3\overline\gamma]=(L\gamma/2,L\gamma_0]\supset [L\gamma,L\gamma_0]$.
The proposition is proved. 
\end{proof}

\section{Proofs of Theorems~\ref{thm1} and \ref{thm1B}} 

\begin{proof}[Proof of Theorem~\ref{thm1}.]
Let $m\ge 1$. 
 We shall construct a suitable, multibump initial data of the form 
$\phi=\sum_{i=1}^m \phi_i$, with compactly supported $\phi_i$.
Recall that the constants $\eps_0,K_1,L, c_0$ are defined in Propositions~\ref{lem2} and~\ref{lem3bis}.
 In view of the application of Proposition~\ref{lem3bis}, bumps will be sequentially 
added by moving toward the boundary. Therefore, we proceed by backward induction
and recursively define functions $\phi_1,\dots,\phi_m \in C^2(\overline\Omega)$ and numbers $\eps_1>\dots>\eps_m>0$ 
as follows.

Fix $\eps_m\in (0,\eps_0)$ and set
\be\label{LBC21}
\phi_m=\psi_{\eps_m},
\ee
where $\psi_{\eps_m}$ is given by Proposition~\ref{lem2}(i).
Take $i\in\{2,\dots,m\}$ and assume that $\phi_i,\dots,\phi_m \in C^2(\overline\Omega)$ 
and $\eps_m>\dots>\eps_i >0$ 
have already been chosen.
Set
$$h_i:=\sum_{j=i}^m \phi_j\in C^2(\overline\Omega)$$
and
\be\label{defgammai}
\gamma_i=\min\Bigl\{\textstyle\frac12\gamma_0\bigl(\Omega,p,\eps_i,\|h_i\|_{C^2(\overline\Omega)}\bigr),
\frac14 c_0L^{-1}\eps_i^2\Bigr\},
\ee
where the function $\gamma_0$ is given by Proposition~\ref{lem3bis}.
We next choose
\be\label{choiceeps}
\eps_{i-1}=\min\Bigl\{\bigl(K_1^{-1}\gamma_i\bigr)^{1/\alpha}, 
\frac12\bigl(c_0^{-1}L\gamma_i\bigr)^{1/2}\Bigr\}  <\textstyle\frac12\eps_i
\ee
and we set 
$\phi_{i-1}=\psi_{\eps_{ i-1}}$ 
where $\psi_{\eps_{i-1}}$ 
is given by Proposition~\ref{lem2}(i).
Recalling \eqref{LBC21}, we thus have
\be\label{LBC2}
\phi_i=\psi_{\eps_i}, \quad i=1,\dots,m
\ee
and
\be\label{LBC3}
 \eps_1>\dots>\eps_m>0.
\ee

Now set 
\be\label{defphi}
\phi:=\displaystyle\sum_{j=1}^m \phi_j
\ee
and
$$\begin{aligned}
s_i&:=c_0\eps_i^2,\qquad i=1,\dots,m,\\
\hat s_i&:=L\gamma_i\in [4s_{i-1},\textstyle\frac14 s_i],\qquad  i=2,\dots,m
\end{aligned}$$
(where we used \eqref{defgammai}, \eqref{choiceeps}). Let us verify that the initial data $\phi$ has all the desired properties.
Since $\phi\ge\phi_i$  for each $i\in\{1,\dots,m\}$, it follows from \eqref{LBC0}, \eqref{LBC2} 
 and Proposition~\ref{compP} that
\be\label{conclthm1}
\sup_{x\in\partial\Omega} u(\phi;\cdot,s_i)\ge \sup_{x\in\partial\Omega} u(\phi_i;\cdot,s_i)>0,
\qquad i\in\{1,\dots,m\}.
\ee
Next for each $i\in\{2,\dots,m\}$, by \eqref{suppcond}, \eqref{suppsi}, \eqref{choiceeps}-\eqref{LBC3},  
we have
$$\phi=\underbrace{\displaystyle\sum_{j=1}^{i-1} \phi_j}_{\hbox{$0$ in $\Omega^{\eps_i}$}}
+\ \displaystyle\sum_{j=i}^m \phi_j
=h_i\quad\hbox{ in $\Omega^{\eps_i}$}$$
and 
$$\phi\le K_1\eps^\alpha_{i-1}+\underbrace{\displaystyle\sum_{j=i}^m \phi_j}_{\hbox{$0$ in $\Omega_{\eps_i}$}}
= K_1\eps^\alpha_{ i-1} \le \gamma_i\quad\hbox{ in $\Omega_{\eps_i}$.}$$ 
This along with \eqref{defgammai} and Proposition~\ref{lem3bis} guarantees that
\be\label{conclthm2}
\hbox{$u$ is a classical solution for $t\in (L\gamma_i,2L\gamma_i]=(\hat s_i,2\hat s_i]$,}
\qquad i\in\{2,\dots,m\}.
\ee
In view of \eqref{conclthm1}-\eqref{conclthm2} and $s_{i-1}<\hat s_i<2\hat s_i<s_i$, the theorem is thus proved.

For future reference, we note that since $\phi=0$ in $\Omega_{\eps_{1}}$,
\eqref{defgammai} and Proposition~\ref{lem3bis} also guarantee that
\be\label{conclthm2modif}
\hbox{$u$ is a classical solution for $t\in (0,2\hat s_1)$.}
\ee

\end{proof}

We next turn to the proof of Theorem~\ref{thm1B}.
We shall use the following uniform lower bound from \cite{FPS19} on the boundary gradient blow-up profile
of maximal classical solutions.
First recall that, thanks to the regularity of $\Omega$, we can find $\delta_0>0$ such that
every $x\in\Omega_{\delta_0}$ has a unique projection $P(x)$ onto $\partial\Omega$.
In this way, we can in particular extend the outer normal vector field to $\Omega_{\delta_0}$ by setting
$\nu_x=\nu_{P(x)}$, and we then denote $u_\nu(x,t)=\nabla u(x,t)\cdot \nu_x$ 
for all $(x,t)\in\Omega_{\delta_0}\times (0,\infty)$.

\begin{proposition}\label{lemCentral2}
Let $p>2$, $\eps\in(0,1)$ and $M>0$. There exists $\eta=\eta(p,\Omega,M,\eps)\in(0,\delta_0)$
such that for any $\phi\in X\cap C^2(\overline\Omega)$ with $\|\phi\|_{C^2(\overline\Omega)}\le M$ and $T:=T(\phi)<\infty$
and any GBU point $a$ of $u$
(i.e., $\limsup_{t\to T^-,\, x\to a} |\nabla u(x,t)|=\infty$),
 then
\begin{equation}
\label{HypTypeILemSmall2}
r^{1/(p-1)} u_\nu(a+r\nu_a,T) \le -(1-\eps)d_p\quad\hbox{for all $r\in(0,\eta]$}.
\end{equation}
\end{proposition}

Proposition~\ref{lemCentral2} was essentially proved in \cite{FPS19}
(see Proposition 5.2 and estimate (1.13)), except for the uniform dependence of $\eta$.
The latter can be easily checked along the proof, using the fact that the Bernstein estimate
\be\label{estBernstein1}
|\nabla u(x,t)|\le C_1\delta^{-1/(p-1)}(x)+C_2\quad\hbox{ in $\Omega\times [0,T)$} 
\ee
from \cite[Theorem~3.2]{SZ06} and the time derivative estimate $\sup_{\Omega\times(0,\infty)} |u_t|\le C_3$
(see \cite[Proposition~2.4]{SZ06}), hold with uniform constants $C_1=C_1(n,p)$, 
$C_2=C_2(p,\Omega,M)$ and $C_3=C_3(p,\Omega,M)$.

\begin{proof}[Proof of Theorem~\ref{thm1B}.]
It is based on a modification of the proof of Theorem~\ref{thm1} and a limiting argument.
Let $m\ge 2$ and let $\phi_i$, $s_i$ and $\hat s_i$ be as the proof of Theorem~\ref{thm1},
which satisfy $s_{i-1}<\hat s_i<2\hat s_i<s_i$ for all $i\in\{2,\dots,m\}$.
We define the family of initial data
\be\label{defphiL}
\Phi_\lambda:=\lambda\phi_1+\displaystyle\sum_{j=2}^m \phi_j,\quad \lambda\in[0,1],
\ee
and denote by $u_\lambda$ the corresponding global viscosity solution of~\eqref{vhj1}.
For any $\lambda\in[0,1]$, by the arguments leading to \eqref{conclthm1} and \eqref{conclthm2}, we have
\be\label{conclthm1BB}
\sup_{x\in\partial\Omega} u_\lambda(\cdot,s_i)>0,\quad i\in\{2,\dots,m\}
\ee
and
\be\label{conclthm2BB}
\hbox{$u_\lambda$ is a classical solution for $t\in(\hat s_i,2\hat s_i)$ and $i\in\{2,\dots,m\}$.}
\ee
Also, for $\lambda=1$, we have
\be\label{conclthm1BBm}
\sup_{x\in\partial\Omega} u_1(\cdot,s_1)>0
\ee
whereas, for $\lambda=0$, since $\Phi_0=\sum_{j=2}^m \phi_j$,
we may apply \eqref{conclthm2modif} with $\eps_1,\hat s_1$ replaced by $\eps_2,\hat s_2$
to get that
\be\label{conclthm2BBm}
\hbox{$u_0$ is a classical solution for $t\in(0,2\hat s_2)$.}
\ee

Next let
$$\lambda^*= \inf E,\qquad E:=\bigl\{\lambda\in[0,1];\ T(\Phi_\lambda)< \hat s_2\bigr\}.$$
Note that $1\in E$ by \eqref{conclthm1BBm}, so that $\lambda^*\in[0,1]$ is well defined.
If $\lambda^*>0$ then, for all $\lambda\in[0,\lambda^*)$, we have
$u_\lambda=0$ on $\partial\Omega\times [0,\hat s_2]$, 
hence
$$u_\lambda=0\quad\hbox{ on $\partial\Omega\times [0,2\hat s_2]$}$$
by \eqref{conclthm2BB}.
By continuous dependence (cf.~\eqref{contdep}), it follows that
\be\label{ustarzero}
u_{\lambda^*}=0\quad\hbox{ on $\partial\Omega\times [0,2\hat s_2]$.}
\ee
Moreover, \eqref{ustarzero} remains true in case $\lambda^*=0$ due to \eqref{conclthm2BBm}.

On the other hand, by definition, there exists a sequence $\lambda_j\to \lambda^*$
with $\lambda_j\ge \lambda^*$, such that
$T_j:=T(\Phi_{\lambda_j})<\hat s_2$.
Since $\partial\Omega$ is compact, $u_{\lambda_j}$ admits at least a GBU point $a_j\in\partial\Omega$.
Observing that $\sup_{\lambda\in [0,1]} \|\Phi_\lambda\|_{C^2(\overline\Omega)}<\infty$,
 we may apply \eqref{HypTypeILemSmall2} with $\eps=\frac12$ and,
 after integrating along the normal direction, we obtain
$$ u_{\lambda_j}(a_j+r\nu_{a_j},T_j) \ge \frac{c_p}{2}r^\alpha\quad\hbox{for all $r\in(0,\eta]$}.$$
After extracting a subsequence, we may assume that $a_j\to a\in\partial\Omega$ and $T_j\to t_0\in [0,\hat s_2]$.
Passing to the limit by means of \eqref{contdep}  and of the continuity of $u_{\lambda^*}$, we deduce that
$$u_{\lambda^*}(a+r\nu_a,t_0) \ge \frac{c_p}{2}r^\alpha\quad\hbox{for all $r\in(0,\eta]$},$$
hence in particular $u_{\lambda^*}(\cdot,t_0)\not\in C^1(\overline\Omega)$.
Therefore $T(\Phi_{\lambda^*})\le t_0\le \hat s_2$.
Taking \eqref{ustarzero} into account, it follows that $u_{\lambda^*}$ undergoes gradient blow-up at $T(\Phi_{\lambda^*})$ 
without loss of boundary conditions.
In view of \eqref{conclthm1BB}, \eqref{conclthm2BB}, this completes the proof 
(since $m\ge 2$ is arbitrary).
\end{proof}

\section{Estimates for $n=1$}

Let us first recall the following estimates from \cite{PS3}
(see \cite[Lemmas 5.1-5.4]{PS3}).

\begin{lem}\label{basic-prop0} 
  Let $\phi\in X_1$ and $t_0>0$. There exists a constant $K>0$ (depending on $\phi$ and $t_0$) with the following properties.
  
(i) For all $t\ge t_0$,
  \be\label{SGBUprofileUpperEstxx}
u_{xx}(x,t) \le K, \quad 0<x<1,
\ee
  \be\label{SGBUprofileUpperEst2}
u_x(x,t) \le U^*_x(x)+\, Kx, \quad 0<x<1,
\ee
\be\label{SGBUprofileUpperEst3}
u_x(x,t) \ge -U^*_x(1-x)-\, K(1-x), \quad 0<x<1.
\ee
In particular, there exists $\bar x\in (0,\frac12)$ (depending on $\phi$ and $t_0$) such that
\be\label{SGBUprofileUpperEst4}
u_x(x,t) \ge -U^*_x(x), \quad 0<x<\bar x.
\ee

(ii) Let $t\ge t_0$ and assume that there exists a sequence $(x_j, t_j)\to (0,t)$ such that $u_x(x_j,t_j)\to \infty$.
Then 
\be\label{profile2}
|u(x,t)-u(0,t)-U^*(x)| \le K x^2, \qquad 0< x\le \textstyle\frac12,
\ee
\be\label{profile}
|u_x(x,t)-U_x^*(x)| \le K x, \qquad 0<x\le \textstyle\frac12.
\ee
Moreover, \eqref{profile2}-\eqref{profile} are true whenever $u(0,t)>0$.
\end{lem}

We shall use also the following simple lemma, which connects the sign of $u_t$ near the boundary
with the sign of $u_x-U^*_x$ and is a variant of \cite[Lemma 5.6]{PS3}.

 \begin{lem}\label{lemz3a} 
Let $\phi\in X_1$, $t>0$ and $a\in (0,\textstyle\frac12)$.
Then 
 \be\label{lemz3impl1}
 \hbox{$u_t(\cdot,t)\le 0$ in $(0,a)$} 
\ \Longrightarrow\ \hbox{$u_x(\cdot,t)\le U_x^*$ in $(0,a)$}
\ee
 and
 \be\label{lemz3impl2}
 \hbox{$u_t(\cdot,t)\ge 0$ in $(0,a)$ and $u_x(0,t)=\infty$}  
\ \Longrightarrow\ \hbox{$u_x(\cdot,t)\ge U_x^*$ in $(0,a)$.}
\ee
  \end{lem}
 
 \begin{proof} Set $v(x)=u(x,t)$. 
 
 First assume $u_t(\cdot,t)\le 0$ in $(0,a)$. Thus we have  $-v_{xx}\ge |v_x|^p$ in $(0,a)$.
 In particular $v_x$ is nonincreasing,
and we may assume $\lim_{x\to 0}v_x>0$ since otherwise the conclusion is obvious.
Therefore the set $\{x\in (0,a);\ v_x>0\}$ is an interval of the form $(0,b)$ with $b\in (0,a]$.
Then we have
 $$[(v_x)^{1-p}]_x=-(p-1)(v_x)^{-p}v_{xx}\ge p-1,\quad 0<x<b.$$
 By integration, it follows that
 $$(v_x)^{1-p}(x)\ge (v_x)^{1-p}(y)+(p-1)(x-y)\ge (p-1)(x-y),\quad 0<y<x<b.$$ 
 Letting $y\to 0$, we obtain
$$v_x(x)\le [(p-1)x]^{-1/(p-1)}=U^*_x(x),\quad 0<x<b,$$
 which gives the desired conclusion.

Next assume $u_t(\cdot,t)\ge 0$ in $(0,a)$ and $u_x(0,t)=\infty$. 
By \eqref{lemz3ahyp0} we have
 \be\label{lemz3lim}
 \lim_{y\to 0}u_x(y,t)=\infty,
 \ee
hence in particular $v_x>0$ for $x>0$ small.
 Let 
 $$b:=\sup\{x_0\in (0,a];\ v_x>0 \hbox{ on } (0,x_0)\}\in (0,a].$$
Since $-v_{xx}\le (v_x)^p$ in $(0,b)$, the same calculation as above now yields
 $$(v_x)^{1-p}(x)\le (v_x)^{1-p}(y)+(p-1)(x-y),\quad 0<y<x<b.$$
Letting $y\to 0$ and using \eqref{lemz3lim}, we get $v_x\ge U^*_x$ in $(0,b)$.
In view of the definition of~$b$, this in turn implies $b=a$. The lemma is proved.
 \end{proof}

The following lemma will enable us to compare $u_x$ and $U^*_x$ in the singular regime.

\begin{lem}\label{lemz3b} 
Let $\phi\in X_1$, $t_2>t_1>0$ and $a\in (0,\textstyle\frac12)$.
Set 
$$Q:=(0,a)\times (t_1,t_2],\qquad \Gamma:=\bigl((0,a]\times\{t_1\}\bigr) \cup \bigl(\{a\}\times (t_1,t_2]\bigr).$$

(i)  Assume that 
 \be\label{lemz3bhyp2}
 \hbox{$u_x-U^*_x\le 0$ on $\Gamma$.}
\ee
Then 
 \be\label{lemz3bconcl1}
 \hbox{$u_x-U^*_x\le 0$ in $Q$.}
 \ee

(ii) Assume that 
 \be\label{lemz3bhyp2b}
 \hbox{$u_x-U^*_x\ge 0$ on $\Gamma$}
\ee
and
 \be\label{lemz3bhyp}
 \hbox{ $u_x(0,s)=\infty$ for all $s\in [t_1,t_2]$.}
 \ee
Then 
 \be\label{lemz3bconcl2}
 \hbox{$u_x-U^*_x\ge 0$ in $Q$.}
 \ee
  \end{lem}

 \begin{proof} 
First observe that if the function $W:=u_x-U^*_x$ attains a local extremum at some point $(x,t)\in Q$,
then at this point we have $u_{xx}=U^*_{xx}$ hence
$$W_t-W_{xx}= h(u_x)u_{xx}-h(U^*_x)U^*_{xx}=\bigl[h(u_x)-h(U^*_x)\bigr]U^*_{xx}$$
where $h(s)=p|s|^{p-2}s$, that is,
 \be\label{lemz3WtWxx}
 W_t-W_{xx}=-\bigl[h(W+U^*_x)-h(U^*_x)\bigr](U^*_x)^p.
 \ee

Let us first assume \eqref{lemz3bhyp2}. Suppose for contradiction that \eqref{lemz3bconcl1} fails, i.e.~$\sup_Q W\in (0,\infty]$.
By \eqref{SGBUprofileUpperEst2} in Lemma~\ref{basic-prop0}, we have $W\le Kx$ in $Q$.
Therefore, $W$ must attain a positive maximum at some point $(x,t)\in Q$.
Since $h$ is an increasing function, at this point we have $h(W+U^*_x)>h(U^*_x)$, hence \eqref{lemz3WtWxx} yields
$0\le W_t-W_{xx}<0$: a contradiction.

Let us next assume \eqref{lemz3bhyp2b} and \eqref{lemz3bhyp}. 
Suppose for contradiction that \eqref{lemz3bconcl2} fails, i.e. $\inf_Q W\in [-\infty,0)$.
In view of \eqref{lemz3bhyp}, estimate \eqref{profile} in Lemma~\ref{basic-prop0} guarantees that $W\ge -Kx$ in $Q$. 
Therefore, $W$ must attain a negative minimum at some point $(x,t)\in Q$ and we reach a contradiction similarly as before.

The lemma is proved.
  \end{proof}

The next lemma shows that the solutions $u_k$ of the truncated problems \eqref{vhj1k}
have good enough approximation properties 
in the singular regime too. Namely $u_k$ has large space derivative near $0$ whenever $u_x(0,t)=\infty$ and $k$ is large.

\begin{lem}\label{lemapprox} 
Let $\phi\in X_1$. Then we have
$$\lim_{k\to\infty\atop \eta\to 0}\Bigl[ \inf\bigl\{\partial_x u_k(x,t);\ x\in [0,\eta],\ t\in \mathcal{S}\bigr\}\Bigr]=\infty,$$
where $\mathcal{S}$ is defined in \eqref{singulartimesT}.
  \end{lem}

 \begin{proof} 
 Let $M, k_0$ be given by \eqref{Liptimek} with $t_0=T$.
 By \eqref{DefFk}, there exist $A_0>0$ and $k_1\ge k_0$ such that $F_k(s)>M$ whenever $s\ge A_0$ and $k\ge k_1$. 
 Let $A\ge A_0$.
 Then it follows from \eqref{vhj1k}, \eqref{Liptimek} that, for any $k\ge k_1$ and $(x,t)\in [0,1]\times[T,\infty)$,
 \be\label{ineqlemapprox}
 \partial_x u_k(x,t)\ge A\  \Longrightarrow [\partial_x^2 u_k](x,t) <0 
 \  \Longrightarrow \inf_{[0,x]}\partial_x u_k(\cdot,t) \ge A.
 \ee
On the other hand, by \eqref{profile}, for any $t\in \mathcal{S}$, we have
$$u_x(x,t)\ge U^*_x(x)-Kx,\quad 0<x<\textstyle\frac12.$$
Therefore there exists $\eta>0$ such that, for any $t\in \mathcal{S}$, $u_x(\eta,t)\ge A+1$.
Moreover, we know from \cite{PZ} that there exists $\tilde T\in(T,\infty)$ such that $\mathcal{S}\subset[T,\tilde T]$.
Since $\partial_xu_k\to u_x$ locally uniformly in $(0,1)\times [T,\tilde T]$, there exists 
$k_2\ge k_1$ such that for all $k\ge k_2$, 
$$\partial_xu_k(\eta,t)\ge A\ \hbox{ for any $t\in \mathcal{S},$}$$
hence $\displaystyle\inf_{[0,\eta]\times \mathcal{S}}\partial_x u_k\ge A$ by \eqref{ineqlemapprox}. This proves the lemma. \end{proof}

\section{Zero number properties I: monotonicity}

We start with the following intersection properties.
Note that these properties are well known for classical solutions,
but are not standard in the context of viscosity solutions.

\begin{proposition}\label{ZeroMonot}
Let $\phi\in X_1$.

(i) For any $t_0\in (0,T(\phi))$, we have 
\be\label{nondegenerateID1}
 0\le  N(t)\le N(t_0)<\infty,\quad t>t_0.
\ee
Moreover, if $N(0)<\infty$ then we may take $t_0=0$ in \eqref{nondegenerateID1}.

(ii) For any $\tau_2>\tau_1>0$ such that $u\ne c_p$ on $\{0\}\times [\tau_1,\tau_2]$, 
the function $N$ is nonincreasing and right continuous on $[\tau_1,\tau_2]$.
Moreover, if $\|\phi\|_\infty\le c_p$, then $N$ is nonincreasing and right continuous on $[0,\infty)$.

(iii) For all $t\in \mathcal{T}$ with $u(1,t)\ne c_p$, the limits 
$N(t^\pm):=\displaystyle\lim_{s\to t^\pm} N(s)$ are well defined and finite and we have
\be\label{NttNt}
N(t^-)\ge N(t)=N(t^+).
\ee
 \end{proposition}

We shall also establish a monotonicity property for the number of intersections of two solutions,
which will be used in the proof of Theorem~\ref{conjz1}.
Actually, we will need to compare solutions of problem~\eqref{vhj1} on different space intervals.
To this end, for any $b\in (0,1]$, we set 
$$X_b:=\bigl\{\phi\in C^1([0,b]);\ \phi\ge 0,\ \phi(0)=\phi(b)=0\bigr\}$$
and the number of sign changes $z_{|[0,b]}$ is defined similarly as in
\eqref{defNt}, replacing $1$ with~$b$.
Also in this and the next section, beside the singular steady state $U^*$, we shall also use 
the regular steady states $U_a$, such that $U_a(0)=0$ and $U_a'(0)=a$, with $a>0$.
Namely,
\be\label{defUa}
U_a(x)=U^*(x+k)-U^*(k),\quad\hbox{ where $k=a^{1-p}/(p-1)$.}
\ee

\begin{proposition}\label{lemz20A} 
Let $b\in(0,1]$, $\phi_1\in X_1$, $\phi_2\in X_b$. 
Denote by $u,v$ the corresponding viscosity solutions of \eqref{vhj1}, 
respectively with $\Omega=(0,1)$ and $\Omega=(0,b)$.
Assume that, for all $t\ge 0$,
\be\label{lemz20AHyp}
\begin{cases}
\hbox{$u(\cdot,t)\not\equiv v(\cdot,t)$ on $[0,1]$,} &\hbox{ if $b=1$,}\\
\noalign{\vskip 1mm}
\hbox{$v(b,t)=0$ (in the classical sense),} &\hbox{ if $b\in(0,1)$,}\\
\end{cases}
\ee
and set
$$\mathcal{N}(t)=z_{|[0,b]}(u(\cdot,t)-v(\cdot,t)),\quad t\ge 0.$$
Then
\be\label{lemz20Aa}
\hbox{$\mathcal{N}(t)$ is finite and nonincreasing on $(0,\infty)$.}
\ee
If moreover $\mathcal{N}(\phi_1-\phi_2)<\infty$, then \eqref{lemz20Aa} remains true 
on $[0,\infty)$.
\end{proposition}

The next result shows that the assumption $\|\phi\|_\infty\le c_p$ in Proposition~\ref{ZeroMonot} is not purely technical.
We actually suspect that for any $M>c_p$, the monotonicity of $N(t)$ fails for some $\phi\in X_1$ with $\|\phi\|_\infty=M$.

\begin{proposition}\label{PropNonmonot} 
There exists $\phi_1\in X_1$, with $\|\phi\|_\infty>c_p$ and $T=T(\phi)<\infty$,
and there exist $t_2>t_1>T$ such that $N(t_1)=0$ and $N(t_2)=1$.
\end{proposition}

In the proof of Proposition~\ref{ZeroMonot} we shall use the following two lemmas.
 
\begin{lem}\label{lemz1} 
 Let $a,b\in\R$, $a<b$ 
 and assume that $v\in C(a,b)$, with $v\not\equiv 0$, is such that $z(v)<\infty$.
 Let $(v_j)_j$ be a sequence of $C(a,b)$ such that 
 $v_j\to v$ in $C_{loc}(a,b)$. Then $\liminf_{j\to\infty} z(v_j)\ge z(v)$.
\end{lem}
 
 \begin{proof} Let $m=z(v)$. By assumption there exist $a<x_0<\dots<x_m<b$ such that
 $v(x_{i-1})v(x_i)<0$ for $i=1,\dots,m$. Since $v_j\to v$ in $C([x_0,x_m])$,
 there exists $j_0$ such that, for all $j\ge j_0$, we have $v_j(x_{i-1})v_j(x_i)<0$ for $i=1,\dots,m$,
 hence $z(v_j)\ge m$. The lemma follows.
 \end{proof}
 
\begin{lem}\label{lemz1b} 
Let $\phi\in X_1$ and let $0<t_0<t_1<T:=T(\phi)$. Then $N(t_0)<\infty$ and there exists $a_1>0$ such that
 \be\label{ZeroMonot1lem}
0\le z(u(\cdot,t_1)-U_a)\le N(t_0),\quad a\ge a_1.
 \ee
Moreover \eqref{ZeroMonot1lem} remains true with $t_0=0$ provided $N(0)<\infty$.
 \end{lem}

 \begin{proof}
Set $U_\infty:=U^*$. Let us first observe that, for any $a\in (0,\infty]$, we have
 \be\label{nonident0}
 u(\cdot,t)-U_a\not\equiv 0\ \hbox{on $[0,1]$,\ for each $t\ge 0$}
 \ee
 (i.e.,~$0\le z(u(\cdot,t)-U_a)\le\infty$). 
Indeed, this is obvious if $u(1,t)\ne U_a(1)$ whereas, if $u(1,t)=U_a(1)>0$ then, 
by applying Lemma~\ref{basic-prop0}(ii) to 
the solution $w(x,t)=u(1-x,t)$, we get $\lim_{y\to 1}u_x(y,t)=-\infty$, hence $u(y,t)\ne U_a(y)$ for $y<1$ close to $1$.

Fix $t_0\in (0,T)$, or $t_0\in [0,T)$ if $N(0)<\infty$, and let $t_1\in(t_0,T)$.
Since 
$$\sup_{s\in [0,t_1]}\|u_x(s)\|_\infty<\infty,$$
there exist $\eta\in (0,1)$ and $a_0>0$ such that 
 \be\label{ZeroMonot1a}
 u<U_{a_0}<U^*\quad\hbox{ in $(0,\eta]\times [0,t_1]$.}
 \ee
Therefore, we have $N(s)=z_{|[\eta,1]}(u(\cdot,s)-U^*)$ for all $s\in [0,t_1]$. 
Since $U^*$ is smooth on $[\eta,1]$ and $U^*(1)>0$, it follows from standard properties of the zero number 
 (cf.~\cite{An88} and see also \cite[Proposition 6.1]{PS3}) that
$N$ is finite and nonincreasing on $[t_0,t_1]$
and there exists $t_2\in (0,t_1)$ such that all zeros of $u(\cdot,t_2)-U^*$ in $[\eta,1]$ are nondegenerate.
Since $U_a$ increases with $a$ and $U_a\to U^*$ in $C^1([\eta,1])$ as $a\to\infty$,
recalling~\eqref{ZeroMonot1a}, there exists $a_1>a_0$ such that
 \be\label{ZeroMonot1c}
 z(u(\cdot,t_2)-U_a)=z(u(\cdot,t_2)-U^*)=N(t_2)\le N(t_0),\quad a\ge a_1.
 \ee
On the other hand, since $u$ and $U_a$ are smooth on $[0,1]\times(0,T)$
with $u=U_a=0$ at $x=0$ and $u=0<U_a$ at $x=1$, the zero number principle
guarantees that $z(u(\cdot,t_2)-U_a)$ is finite and nonincreasing on $(t_0,t_1]$.
This along with \eqref{ZeroMonot1c} yields the desired conclusion.
\end{proof}

 \begin{proof}[Proof of Proposition~\ref{ZeroMonot}(i)]
First of all, we have $0\le N(t)\le\infty$ by \eqref{nonident0}.
Set $T=T(\phi)$ and fix $t_0\in(0,T)$, or $t_0\in[0,T)$ in case $N(0)<\infty$.
Let $t>t_0$ and pick any $t_1\in (t_0,\min(T,t))$.
By Lemma~\ref{lemz1b}, we have
 \be\label{ZeroMonot1c2}
 z(u(\cdot,t_1)-U_a)\le N(t_0)<\infty,\quad a\ge a_1.
 \ee
Next, since $\sup_{s\in [0,t_1]}\|u_x(s)\|_\infty<\infty$, there exists $k_0\ge 0$ such that the solutions $u_k$ of the truncated problems \eqref{vhj1k} satisfy
 \be\label{ZeroMonot1d}
 \hbox{ $u_k\equiv u$ on $[0,1]\times[0,t_1]$ for all $k\ge k_0$.}
 \ee
Moreover, for any integer $j\ge a_1$, there exists $k_j\ge \max(k_0,j)$ such that $U_j$ is a steady state of problem \eqref{vhj1k} with $k=k_j$.
Since $u_{k_j}$ and $U_j$ are smooth, 
with $u_{k_j}=U_j=0$ at $x=0$ and $u_{k_j}=0<U_j$ at $x=1$, we can apply the zero number principle
to infer that
$$z(u_{k_j}(\cdot,t)-U_j)\le z(u_{k_j}(\cdot,t_1)-U_j).$$
This combined with \eqref{ZeroMonot1c2}, \eqref{ZeroMonot1d} yields
$$z(u_{k_j}(\cdot,t)-U_j)\le z(u(\cdot,t_1)-U_j)\le N(t_0).$$
Since $u_{k_j}(\cdot,t)-U_j\to u(\cdot,t)-U^*$ in $C_{loc}(0,1)$ as $j\to\infty$, property \eqref{nondegenerateID1} follows from Lemma~\ref{lemz1}.
The assertion is proved.
\end{proof}

In view of the proof of assertions (ii)(iii), we prepare
the following lemma which gives a more general property of the zero number on intervals $[0,b]$
and will be useful also in the proof of Theorem~\ref{thmz1}.
 In what follows, for any $b\in (0,1]$, we denote
$$N_b(t):=z_{|[0,b]}(u(\cdot,t)-U^*),\quad t>0.$$

\begin{lem} \label{zerob}
Let $\phi\in X_1$,
$b\in (0,1]$ and $\tau_2>\tau_1>0$ be such that 
 \be\label{Openneighb00}
u-U^*\ne 0\quad\hbox{on $\{b\}\times [\tau_1,\tau_2]$.}
 \ee
Then the function $N_b(t)$ is nonincreasing and right continuous on $[\tau_1,\tau_2]$.
\end{lem}

 \begin{proof}
It suffices to consider the case $b\in(0,1)$.
Indeed if $b=1$ then, by continuity, for $\eta\in(0,1)$ small,
assumption \eqref{Openneighb00} remains true for $b=1-\eta$ and we have $N_1(t)=z_{|[0,1-\eta]}(u(\cdot,t)-U^*)$ for all $t\in[\tau_1,\tau_2]$.

 {\bf Step~1.} {\it Upper semicontinuity on the right.}
 
We claim that for each $t\in [\tau_1,\tau_2)$, 
 \be\label{Openneighb}
 \limsup_{s\to t^+} N_b(s) \le N_b(t).
 \ee
Set $w=u-U^*$. Since $0\le N(t)<\infty$ by Proposition~\ref{ZeroMonot}(i), 
 there exist $\sigma\in\{-1,1\}$ and $a\in (0,1)$ such that 
   $$\sigma w(\cdot,t)\ge 0 \hbox{ on $[0,a]$ \quad and \quad }\sigma w(a,t)>0.$$
By continuity, there exists $\eps\in (0,\tau_2-t)$ such that
 $ \sigma  w(a,s)>0$
  for all $s\in [t,t+\eps]$.
  Applying Proposition~\ref{compP0}  if $\sigma=1$ and 
  Proposition~\ref{compP} if $\sigma=-1$, we then have
 \be\label{Openneighb1a}
\sigma w(\cdot,t)\ge 0 
 \hbox{ on $[0,a]\times [t,t+\eps]$ \quad and \quad}
  \sigma w(a,t) >0
   \hbox{ on $[t,t+\eps]$.}
 \ee
We may suppose $b>a$, since otherwise $N_b= 0$ on $[t,t+\eps]$ and we are done. 
By \eqref{Openneighb1a} 
we have
 \be\label{Openneighb2}
N_b(s)=z_{|[a,b]}(u(\cdot,s)-U^*),\quad t\le s\le t+\eps.
 \ee
 Since $u, U^*$ are smooth in $[a,b]\times[t,t+\eps]$ and $u\ne U^*$ on $\{a,b\}\times[t,t+\eps]$ 
due to \eqref{Openneighb00}, \eqref{Openneighb1a}, we may apply the zero number principle 
 on $[a,b]$. In view of \eqref{Openneighb2} this yields $N_b(s)\le N_b(t)$ for all $s\in [t,t+\eps]$,
hence \eqref{Openneighb}.
 
\smallskip

 {\bf Step~2.} {\it Monotonicity in $[\tau_0,\tau_1]$.}
Let $\tau_1\le t_1<t_2\le \tau_2$ and suppose for contradiction that 
 \be\label{ZeroMonot20}
N_b(t_2)\ge N_b(t_1)+1.
 \ee
 Set $t_3:=\inf\{t\in (t_1,t_2];\, N_b(t)\ge N_b(t_1)+1\}\in [t_1,t_2]$. 
 We claim that 
 \be\label{ZeroMonot2}
 N_b(t_3)\le N_b(t_1).
 \ee
 Indeed this is obvious if $t_3=t_1$ whereas, 
if $t_3>t_1$, then $N_b(t)\le N_b(t_1)$ on $[t_1,t_3)$, so that Lemma~\ref{lemz1} implies \eqref{ZeroMonot2}.
By \eqref{ZeroMonot20}, \eqref{ZeroMonot2} we thus have $t_3<t_2$,
hence there exists a sequence $t_j\to t_3^+$ such that $N_b(t_j)\ge N_b(t_1)+1$.
But this contradicts \eqref{Openneighb}. Consequently, $N_b$ is nonincreasing on $[\tau_0,\tau_1]$.
  \end{proof}

  \begin{proof}[Proof of Proposition~\ref{ZeroMonot}(ii)(iii)]
 The first part of assertion (ii) is a direct consequence of Lemma~\ref{zerob}.
The second part follows from the fact that, by the strong maximum principle, 
 \be\label{phiinfty}
 \|\phi\|_\infty\le c_p\ \Longrightarrow\ \|u(t)\|_\infty<c_p \ \hbox{ for all $t>0$.}
 \ee

To prove assertion (iii), taking $t\in\mathcal{T}$, we have $u \ne c_p$ on $\{0\}\times[t-\eps,t+\eps]$ for $\eps>0$ small
 by continuity. It thus follows from assertion (ii) that $N$ is nonincreasing and right continuous on $[t-\eps,t+\eps]$. Therefore $N(t^\pm)$ exist
 and \eqref{NttNt} is true.
 \end{proof}

  \begin{proof}[Proof of Proposition~\ref{lemz20A}]
Let $t_0\in (0,T_0)$, with $T_0:=\min(T(\phi_1),T(\phi_2))$,
or $t_0=0$ if $\mathcal{N}(\phi_1-\phi_2)<\infty$.
 Let $u_k,v_k$ be the global classical solutions of the truncated problems \eqref{vhj1k}
 with $\Omega=(0,1)$, $\phi=\phi_1$ and $\Omega=(0,b)$, $\phi=\phi_2$, respectively.
Since $\sup_{s\in [0,t_0]}(\|u_x(s)\|_\infty+\|v_x(s)\|_\infty)<\infty$,
 there exists $k_0\ge 0$ such that 
 \be\label{ZeroMonot1d2}
 \hbox{ $u_k\equiv u$ on $[0,1]\times[0,t_0]$ and  $v_k\equiv v$ on $[0,b]\times[0,t_0]$ for all $k\ge k_0$.}
\ee
The solutions $u_k,v_k$ are smooth, with $u_k=v_k=0$ at $x=0$, $u_k=v_k=0$ at $x=1$ if $b=1$ and
$u_k>0=v_k$ at $x=b$ if $b\in(0,1)$. Consequently, we can apply the zero number principle
to infer that
$$z_{|[0,b]}\bigl((u_k-v_k)(\cdot,t)\bigr)\le z_{|[0,b]}\bigl((u_k-v_k)(\cdot,t_0)\bigr)
=\mathcal{N}(t_0)<\infty,\quad t>t_0$$
(where the equality is due to \eqref{ZeroMonot1d2}).
Since $(u_k-v_k)(\cdot,t)\to (u-v)(\cdot,t)$ in $C_{loc}(0,1)$ as $k\to\infty$, we deduce
from Lemma~\ref{lemz1} and assumption \eqref{lemz20AHyp} that 
$\mathcal{N}(t)$ is finite (i.e., $\mathcal{N}(t)\in\N$) for all $t\ge t_0$.

We next claim that for each $t\ge t_0$, 
 \be\label{OpenneighbW}
 \limsup_{s\to t^+} \mathcal{N}(s) \le \mathcal{N}(t).
 \ee
Set $w=u-v$. Since $\mathcal{N}(t)$ is finite,
by the argument in Step 1 of the proof of Lemma~\ref{zerob}, it follows that there exist 
  $a_0\in (0,1)$, $\eps_0>0$ and $\sigma_0\in\{-1,1\}$ such that
   \be\label{Openneighb1aW1}
 \sigma_0 w(\cdot,t)\ge 0 \hbox{ on $[0,a_0]\times[t,t+\eps_0]$ \quad and \quad}
 \sigma_0  w(a_0,t)> 0\hbox{ on $[t,t+\eps_0]$.}
 \ee
 If $b=1$, then we similarly find $a_1\in (0,1)$, $\eps_1>0$  and $\sigma_1\in\{-1,1\}$ such that
   \be\label{Openneighb1aW2}
 \sigma_1 w(\cdot,t)\ge 0 \hbox{ on $[a_1,1]\times[t,t+\eps_1]$ \quad and \quad}
 \sigma_1  w(a_1,t)> 0\hbox{ on $[t,t+\eps_1]$.}
 \ee
 If $b\in (0,1)$ then, in view of assumption \eqref{lemz20AHyp} and since $u\ge e^{t\Delta}\phi>0$ in $(0,1)\times(0,\infty)$, 
there exist $a_1\in (0,b)$, $\eps_1>0$ such that
   \be\label{Openneighb1aW3}
w(\cdot,t)>0 \hbox{ on $[a_1,b]\times[t,t+\eps_1]$.}
 \ee
Letting $\eps=\min(\eps_0,\eps_1)$ we may suppose $a_1>a_0$, 
since otherwise $\mathcal{N}=0$ on $[t,t+\eps]$ and we are done. 
By \eqref{Openneighb1aW1}-\eqref{Openneighb1aW3} we have
 \be\label{Openneighb2W}
\mathcal{N}(s)=z_{|[a_0,a_1]}(w(\cdot,s)),\quad t\le s\le t+\eps.
 \ee
 Since $u,v$ are smooth on $[a_0,a_1]$ and $u\ne v$ on $\{a_0,a_1\}\times[t,t+\eps]$ 
due to \eqref{Openneighb1aW1}-\eqref{Openneighb1aW3}, we may apply the zero number principle 
 on $[a_0,a_1]$. In view of \eqref{Openneighb2W} this yields $\mathcal{N}(s)\le \mathcal{N}(t)$ for all $s\in [t,t+\eps]$,
hence \eqref{OpenneighbW}.

The monotonicity of $\mathcal{N}$ then follows exactly as in Step 2 of the proof of Lemma~\ref{zerob}.
  \end{proof}

\begin{proof}[Proof of Proposition~\ref{PropNonmonot}]
Fix $\phi_0\in X_1$, with 
\be\label{hypPropNonmonot}
\hbox{$\phi_0$ symmetric and nondecreasing on $[0,\frac12]$, }
\ee
such that the corresponding solution $v$ of \eqref{vhj1} satisfies $v(0,t_1)=v(1,t_1)>0$ for some $t_1>0$.
Let $\phi=\lambda \phi_0$ with $\lambda>1$.
Setting $\mu=\lambda^{-1}$ and $z=\mu u$, we have
 $$z_t-z_{xx}-|z_x|^p=\mu(u_t-u_{xx}-\mu^{p-1}|u_x|^p)=\mu(1-\mu^{p-1})|u_x|^p\ge 0
 \ \hbox{ in $(0,1)\times(0,\infty)$.}$$
It thus follows from Proposition~\ref{compP} that $z\ge v$, hence $u\ge \lambda v$ in $[0,1]\times[0,\infty)$.

Now taking $\lambda>\max\{1,(v(1,t_1))^{-1}c_p\}$, we get $u(1,t_1)>c_p$.
Also, assumption \eqref{hypPropNonmonot} guarantees that $u(\cdot,t)$ is symmetric and 
nondecreasing on $[0,\frac12]$ for each $t>0$ (arguing on the truncated problems 
\eqref{vhj1k} and passing to the limit).
Therefore, $u(\cdot,t_1)>c_p\ge U^*$ in $[0,1]$, hence $N(t_1)=0$.
But we know from \cite{PZ} that $u$ eventually decays to~$0$ in $C^1$ norm, hence $N(t)=0$ for $t\gg 1$.
This implies that, at some $t_2>t_1$, $0<u(0,t_2)=u(1,t_2)<c_p$, hence $N(t_2)\ge 1$. 
  \end{proof}

\section{Zero number properties II: analysis of the transition set $\mathcal{T}$}

Define the sets:
$$\Sigma_+= \bigl\{t>0;\ \exists\, a\in (0,1),\,u(\cdot,t)\ge U^* \hbox{ on $(0,a)$ and $u(a,t)>U^*(a)$}\bigr\},$$ 
$$\Sigma_-= \bigl\{t>0;\ \exists\, a\in (0,1),\,u(\cdot,t)\le U^* \hbox{ on $(0,a)$ and $u(a,t)<U^*(a)$}\bigr\}.$$
We have $\Sigma_+\cap\Sigma_-=\emptyset$ and, as a consequence of Proposition~\ref{ZeroMonot}(i),
\be\label{cupSigplusminus}
(0,\infty)=\Sigma_+\cup \Sigma_-.
\ee
 Moreover, we have
\be\label{SigplusT}
\Sigma_+\subset \mathcal{S}
\ee
(this is obvious if $u(0,t)=0$ and follows from Lemma~\ref{basic-prop0}(ii) if $u(0,t)>0$).

Next, for all $t_0\in (0,T(\phi))$, we have
\be\label{zutmonot}
z(u_t(\cdot,t))\le z(u_t(\cdot,t_0))<\infty,\quad t>t_0.
\ee
 Indeed, since $u$ is smooth on $[0,1]\times (0,T(\phi))$, we have $z(u_t(\cdot,t_0))<\infty$ by standard properties of the zero number 
applied to the equation for $u_t$, and \eqref{zutmonot} then follows from \cite[Proposition 6.2]{PS3}.

 The goal of this section is to prove the following two propositions. The first one guarantees 
immediate loss of boundary conditions (resp., regularization)
 at any time at which the solution dominates (resp., is dominated by) the singular steady state near the boundary.
 Moreover, the recovery of boundary conditions or loss of regularity
 cannot occur as long as the intersection number remains constant.

\begin{proposition}\label{lemz6}
Let $\phi\in X_1$, $b\in (0,1]$ and $t_2>t_1>0$ be such that 
\be\label{hypNconst}
\hbox{$N_b(t)=z_{|[0,b]}(u(\cdot,t)-U^*)$ is constant on $[t_1,t_2]$.}
\ee

(i) If $t_1\in \Sigma_+$, then $u>0$ on $\{0\}\times(t_1,t_2]$.

(ii) If $t_1\in \Sigma_-$, then $u=0$ on $\{0\}\times(t_1, t_2]$ and $u$ is a classical solution on~$[0,\frac12]\times (t_1, t_2]$.
\end{proposition}

Our second proposition shows that the intersection number has to drop at any transition time $t\in\mathcal{T}$,
provided no intersection occurs at $x=1$. We actually need a slightly more general version on intervals $[0,b]\subset [0,1]$.

\begin{proposition}\label{zerob2}
Let $\phi\in X_1$, $b\in (0,1]$ and 
assume that $t\in\mathcal{T}$ satisfies $u(b,t)\ne U^*(b)$. Then $N_b(t^-)$ is well defined and we have
\be\label{singulartimes2}
N_b(t^-)\ge N_b(t)+1.
\ee
Moreover, if $\|\phi\|_\infty\le c_p$ or if $\phi$ is symmetric, then \eqref{singulartimes2} with $b=1$ is true for any $t\in \mathcal{T}$.
\end{proposition}

In view of the proof of Propositions~\ref{lemz6} and \ref{zerob2} we need the following two lemmas.

\begin{lem}\label{lemz20} 
Let $\phi\in X_1$,
$b\in (0,1]$ and $t_2>t_1>0$ satisfy \eqref{hypNconst}. 
Then there exists $a>0$ such that 
$$\hbox{$u>U^*$ in $(0,a]\times (t_1,t_2]$}$$
or
$$\hbox{$u<U^*$ in $(0,a]\times (t_1,t_2]$.}$$
 \end{lem}

 \begin{proof} 
Set $w=u-U^*$.  
We first note that, for any 
$t_1\le s<t\le t_2$, $d\in (0,1)$ and $\sigma\in \{-1,1\}$,
\be\label{lemz20b}
\left.\begin{aligned}
&\hbox{$\sigma w\le 0$ in $[0,d]\times [s,t]$}\\
&\hbox{$\sigma w<0$ on $\{d\}\times (s,t]$}
\end{aligned}
\ \ \right\}
 \quad\Longrightarrow\quad 
 \hbox{$\sigma w<0$ in $(0,d]\times (s,t],$}
\ee
as a consequence of the strong maximum principle
(applied on each interval $[\eta,d]$ with $\eta\in(0,d)$,
where $u$ and $U^*$ are smooth).

Next, by \eqref{cupSigplusminus}, there exist $c\in (0,1)$ and $\bar\sigma\in \{-1,1\}$ such that
$\bar\sigma w(\cdot,t_1)\le 0$ in $[0,c]$ and $\bar\sigma w(c,t_1)<0$.
By continuity, there exists $t_3\in (t_1,t_2)$ such that $\bar\sigma w(c,t)<0$ for all $t\in [t_1,t_3]$.
Applying Proposition~\ref{compP0} if $\bar\sigma=-1$ and
Proposition~\ref{compP} if $\bar\sigma=1$, we deduce that 
$\bar\sigma w\le 0$ in $[0,c]\times [t_1,t_3]$,
hence $\bar\sigma w<0$ in $(0,c]\times (t_1,t_3]$, by \eqref{lemz20b}.

Now, the set
$$\hbox{$J:=\bigl\{t\in (t_1,t_2];\ \bar\sigma w<0$ in $(0,d]\times (t_1,t]$ for some $d>0\bigr\}$}$$
is a nontempty interval and we put $\tau=\sup J\in (t_1,t_2]$. 
 Since $N_b(\tau)=m:= N_b(t_1)$, there exist
 $0<x_1<\dots<x_{m+1}<b$ and $\hat\sigma\in \{-1,1\}$ such that
$$(-1)^i \hat\sigma w(x_i,\tau)>0,\quad i=1,\dots,m+1.$$
By continuity, there exists $\eps\in (0,\tau-t_1)$ such that
\be\label{lemz20c}
(-1)^i \hat\sigma w(x_i,t)>0,
\quad i=1,\dots,m+1,\quad\hbox{ for all $t\in[\tau-\eps,\bar\tau]$,}
\ee
where $\bar\tau=\min(t_2,\tau+\eps)$.
It follows from \eqref{lemz20c} and $N_b(t)=m$ 
that
$$\hbox{$\hat\sigma w\le 0$ \  in $[0,x_1]\times [\tau-\eps,\bar\tau]$ \quad
and \quad $\hat\sigma w<0$ \ on $\{x_1\}\times [\tau-\eps,\bar\tau]$},$$
hence
\be\label{lemz20d}
\hat\sigma w<0\quad \hbox{ in $(0,x_1]\times (\tau-\eps,\bar\tau]$}
\ee
by \eqref{lemz20b}. 
Now, since $\tau-\eps\in J$, there exists $a\in (0,x_1)$ such that
$\bar\sigma w<0$ in $(0,a]\times  (t_1,\tau-\eps]$.
By \eqref{lemz20d} and continuity, 
we deduce that $\hat\sigma=\bar\sigma$, hence 
\be\label{lemz20e}
\hbox{$\bar\sigma w<0$ in $(0,a]\times (t_1,\bar\tau]$.}
\ee
Therefore, $\bar\tau\in J$ hence
$\bar\tau\le\tau$, i.e. $\bar\tau=t_2$,
so that \eqref{lemz20e} proves the lemma.
\end{proof}

\begin{lem}\label{lemz7} 
Let $\phi\in X_1$, $b\in (0,1]$ and let $t_2>t_1>0$ satisfy \eqref{hypNconst}. 

(i) If $u>0$ on $\{0\}\times[t_1,t_2)$, then $u(0,t_2)>0$.

(ii) If $u(0,t_1)=0$ and $u_x(0,t_1)<\infty$,  then $u=0$ on $\{0\}\times(t_1,t_2]$ and $u$ is a 
classical solution on $ [0,\frac12]\times(t_1,t_2]$.
 \end{lem}

\begin{proof} 
(i) By Lemma~\ref{lemz20} there exists $a>0$ such that 
 \be\label{lemz7p1}
 \hbox{$u>U^*$ in $(0,a]\times (t_1,t_2]$.}
 \ee
For $\lambda\in (0,1)$, let $u_\lambda$ be the global viscosity solution of problem~\eqref{vhj1} 
with initial data $\lambda\phi$.
As a consequence of Proposition~\ref{compP}, we have 
 \be\label{lemz7stabil}
 \|u(\cdot,t)-u_\lambda(\cdot,t)\|_\infty\le \|u(\cdot,0)-u_\lambda(\cdot,0)\|_\infty= (1-\lambda)\|\phi\|_\infty,\quad t>0
 \ee
(this follows by taking $w=u+\|u(\cdot,0)-u_\lambda(\cdot,0)\|_\infty$ 
as comparison function, and then exchanging the roles of $u$ and $u_\lambda$).
Since $z:=\lambda u$ satisfies
$$z_t-z_{xx}-|z_x|^p=\lambda(u_t-u_{xx}-\lambda^{p-1}|u_x|^p)=\lambda(1-\lambda^{p-1})|u_x|^p\ge 0 
\quad\hbox{in $(0,1)\times (0,\infty)$}$$
it follows from Proposition~\ref{compP} that 
 \be\label{lemz7p2}
 u_\lambda\le \lambda u\quad\hbox{in $(0,1)\times (0,\infty)$.}
 \ee

Fix $t_3\in(t_1,t_2)$ and define the compact set $\Gamma=\bigl([0,a]\times \{t_3\}\bigr)\cup \bigl(\{a\}\times[t_3,t_2]\bigr)$.
Owing to \eqref{lemz7p1} and $u(0,t_3)>0$, we have $\eta=\displaystyle\min_\Gamma\, (u-U^*)>0$.
Taking $\lambda$ close enough to $1$, so that $(1-\lambda)\|\phi\|_\infty<\eta$ and using \eqref{lemz7stabil}, we get
$$u_\lambda-U^*\ge 0\quad\hbox{ on $\Gamma$.}$$
Since $u_\lambda\ge 0$ on $\{0\}\times[t_3,t_2]$
it follows from Proposition~\ref{compP0} that
$$u_\lambda\ge U^*\quad\hbox{ in $(0,a)\times (t_3,t_2)$.}$$
Combining this with \eqref{lemz7p2}, we obtain
$$u(\cdot,t_2)\ge \lambda^{-1} u_\lambda(\cdot,t_2)\ge \lambda^{-1}U^*\quad\hbox{ in $(0,a)$.}$$
By Lemma~\ref{basic-prop0}(ii), we conclude that $u(0,t_2)>0$.

(ii) By our assumptions and \eqref{lemz3ahyp0}, there exist $c\in (0,1)$ and 
$\mu_0>0$ such that, for all $\mu\ge \mu_0$,
 \be\label{lemz7b0a}
 u(\cdot,t_1)<U_\mu<U^*\quad\hbox{ in $(0,c)$,}
 \ee
where the regular steady state $U_\mu$ is defined in \eqref{defUa}.
Consequently, by Lemma~\ref{lemz20}, there exists $a\in (0,c)$ such that 
$$u<U^* \quad\hbox{ in $(0,a]\times [t_1,t_2]$,}$$
hence in particular $u=0$ on $\{0\}\times(t_1,t_2)$.
We may thus choose $\lambda>\mu_0$ sufficiently large so that
 \be\label{lemz7b0b}
 u\le U_\lambda \quad\hbox{ on $\{0,a\}\times [t_1,t_2]$.}
 \ee
It then follows from \eqref{lemz7b0a}, \eqref{lemz7b0b} and Proposition~\ref{compP0} that 
$$u(x,t)<U_\lambda(x)\le \lambda x\quad\hbox{ for all $(x,t)\in (0,a]\times [t_1,t_2]$.}$$
Let $(x,t)\in (0,a]\times [t_1,t_2]$. By the mean-value theorem there exists $y\in (0,x)$ such that $u_x(y,t)=x^{-1}u(x,t)\le \lambda$.
Since $s\mapsto u_x(s,t)-Ks$ is nonincreasing by \eqref{SGBUprofileUpperEstxx} in Lemma~\ref{basic-prop0}(i) (applied with $t_0=t_1$), it follows that
$u_x(x,t)-Kx\le u_x(y,t)-Ky \le \lambda$, hence $u_x(x,t)\le \lambda+K$.
This combined with \eqref{SGBUprofileUpperEst2}-\eqref{SGBUprofileUpperEst3} guarantees that 
$$\displaystyle\sup_{(0,\frac34)\times (t_1,t_2)}|u_x|<\infty.$$
By parabolic estimates we conclude that $u$ is a classical solution on $[0,\frac12]\times(t_1,t_2]$.
 \end{proof}

 With Lemmas~\ref{lemz20} and \ref{lemz7} at hand we can now turn to the proof of Propositions~\ref{lemz6} and \ref{zerob2}.

\begin{proof} [Proof of Proposition~\ref{lemz6}]
Set $v=u-U^*$. 

(i) Assume for contradiction that $u(0,t_3)=0$ for some $t_3\in (t_1,t_2]$.
We consider the following two cases.

{\it Case 1:} there exists $t_4\in (t_1,t_3)$ such that $u(0,t_4)>0$.
Letting
$$t_5=\inf\{s>t_4;\, u(0,s)=0\}\in (t_4,t_3],$$
we then have $u(0,t_5)=0$ by continuity, whereas $u>0$ on $\{0\}\times (t_4,t_5)$.
But this contradicts Lemma~\ref{lemz7}(i).

{\it Case 2:} 
 \be\label{lemz6p0b}
 u(0,s)=0\quad\hbox{ for all $s\in (t_1,t_3]$.}
 \ee
By our assumption, there exists $a\in(0,1)$ such that  
 \be\label{lemz6p1}
\hbox{$v(\cdot,t_1)\ge 0$ in $[0,a]$ \quad and \quad $v(a,t_1)>0$.}
\ee
By continuity there exists $t_4\in (t_1,t_2)$ 
such that $v(a,s)>0$ for all $s\in (t_1,t_4]$.
Applying the comparison principle (Proposition~\ref{compP0}) and then the strong maximum principle, it follows that
 \be\label{lemz6p2}
\hbox{$v>0$ in $(0,a]\times (t_1,t_4]$.}
\ee
In particular \eqref{lemz6p0b} and \eqref{lemz6p2} imply that
 \be\label{lemz6p2a}
\hbox{$u_x(0,s)=\infty$ for all $s\in (t_1,t_4]$.}
\ee
Pick $t_5\in (t_1,t_4)$.
By \eqref{zutmonot} we have $z(u_t(\cdot,t_5))<\infty$. Consequently, we have either
$u_t(\cdot,t_5)\ge 0$ or $u_t(\cdot,t_5)\le 0$ for $x>0$ small.
We deduce from Lemma~\ref{lemz3a} and \eqref{lemz6p2a} that either
 \be\label{lemz6p2b}
 \hbox{$u_x(\cdot,t_5)\ge U_x^*\ \ $ or $\ \ u_x(\cdot,t_5)\le U_x^*\ \ $ for $x>0$ small.}
 \ee
The second alternative in \eqref{lemz6p2b} cannot hold in view of \eqref{lemz6p0b}, 
since this would lead to $v(\cdot,t_5)\le 0$ for $x>0$ small,
contradicting \eqref{lemz6p2}.

Therefore, there exists $c\in (0,a)$ such that
$$
 u_x(\cdot,t_5)\ge U_x^*\quad\hbox{ in $(0,c)$}
$$
 and $u_x(b,t_5)>U_x^*(c)$.
By continuity, there exist $t_6\in (t_5,t_4)$ such that 
$$
 u_x(c,s)>U_x^*(c)\quad\hbox{ for all $s\in [t_5,t_6]$.}
$$
Lemma~\ref{lemz3b}(ii) and \eqref{lemz6p2a} then guarantee 
that $u_x\ge U^*_x$ in $Q=:(0,c)\times (t_5,t_6]$. Consequently, the function $v$ satisfies
$$v_t-v_{xx}=(u_x)^p-(U^*_x)^p\ge 0\quad\hbox{ in $Q$,}$$
with $v>0$ in $Q$.
By comparison with the solution of the heat equation, it follows that
$v(\cdot,t_6)\ge \eps x$ in $(0,c/2)$ for some $\eps>0$. But,
recalling \eqref{lemz6p0b}, this contradicts estimate \eqref{profile2} in Lemma~\ref{basic-prop0}.
Assertion (i) is proved.

(ii) By our assumption, there exists $a\in(0,1)$ such that  
 \be\label{lemz6p1b}
\hbox{$v(\cdot,t_1)\le 0$ in $[0,a]$ \quad and \quad $v(a,t_1)<0$.} 
\ee
By continuity there exists $t_3\in(t_1,t_2)$ such that $v(a,s)<0$ for all $s\in [t_1,t_3]$.
It follows from Proposition~\ref{compP} that
$u\le U^*$ in $[0,a]\times[t_1,t_3]$, hence in particular 
\be\label{lemz6zeroBC}
\hbox{$u(0,s)=0$ for all $s\in[t_1,t_3]$.}
\ee
Moreover, by the strong maximum principle, 
we get
 \be\label{lemz6p2bis}
\hbox{$v<0$ in $(0,a]\times (t_1,t_3]$.}
\ee
We consider the following two cases.

{\it Case 1:} there exists a sequence $s_i\to t_1$ with $s_i> t_1$, such that $u_x(0,s_i)<\infty$.
By Lemma~\ref{lemz7}(ii), it follows that $u$ is a classical solution on $[0,\frac12]\times(s_i,t_2]$ for each $i$, 
hence on~$[0,\frac12]\times(t_1,t_2]$, which is the desired result.

{\it Case 2:} there exists $t_4\in (t_1,t_3)$ such that 
\be\label{lemz6p2a2}
\hbox{$u_x(0,s)=\infty$ for all $s\in (t_1,t_4]$.}
\ee
Pick $t_5\in (t_1,t_4)$. 
By \eqref{zutmonot} we have $z(u_t(\cdot,t_5))<\infty$. Consequently, we have either
$u_t(\cdot,t_5)\ge 0$ or $u_t(\cdot,t_5)\le 0$ for $x>0$ small.
We deduce from Lemma~\ref{lemz3a} and \eqref{lemz6p2a2} that either
 \be\label{lemz6p2b2}
 \hbox{$u_x(\cdot,t_5)\ge U_x^*\ \ $ or $\ \ u_x(\cdot,t_5)\le U_x^*\ \ $ for $x>0$ small.}
 \ee
The first alternative in \eqref{lemz6p2b2} cannot hold, since this would lead to $v(\cdot,t_5)\ge 0$ for $x>0$ small,
contradicting \eqref{lemz6p2bis}.

Let $\bar x$ be given by Lemma~\ref{basic-prop0}(i) with $t_0=t_1$.
Then there exists $c\in (0,\min(\bar x,a))$ such that
$$u_x(\cdot,t_5)\le U_x^*\ \hbox{ in $(0,c)$ \quad and \quad $u_x(c,t_5)<U_x^*(c)$.}$$
By continuity, there exists $t_6\in (t_5,t_4)$ such that 
$$
 u_x(c,s)<U_x^*(c)\quad\hbox{ for all $s\in [t_5,t_6]$.}
$$
Lemma~\ref{lemz3b}(i) then guarantees
that $u_x\le U^*_x$ in $Q:=(0,c)\times (t_5,t_6]$. 
Since also $u_x\ge -U^*_x$ in $Q$ by \eqref{SGBUprofileUpperEst4}, the function $v$ thus satisfies
$$v_t-v_{xx}=|u_x|^p-(U^*_x)^p\le 0\quad\hbox{ in $Q$,}$$
with $v<0$ in $Q$.
By comparison with the solution of the heat equation, it follows that
$v(\cdot,t_6\le -\eps x$ in $(0,c/2)$ for some $\eps>0$. But
recalling \eqref{lemz6p2a2}, this contradicts estimate \eqref{profile2} in Lemma~\ref{basic-prop0}.
The lemma is proved.
\end{proof}

\begin{proof} [Proof of Proposition~\ref{zerob2}]
By continuity, there exists $\eps>0$ such that $u(s,b)\ne U^*(b)$ for all $s\in [t-\eps,t+\eps]$
and Lemma~\ref{zerob} thus guarantees that the function $N_b(t)$ is nonincreasing and right continuous on $[t-\eps,t+\eps]$.
Assume for contradiction that $N_b(t^-)=N_b(t)$. 
Since $N_b(t)$ is integer valued, there exists $\hat\eps\in(0,\eps)$ such that 
\be\label{NsNt}
\hbox{$N_b(s)=N_b(t)$ for all $s\in [t-\hat\eps,t+\hat\eps]$.}
\ee
If $t-\hat\eps\in \Sigma_+$ then, by Proposition~\ref{lemz6}(i), $u(0,s)>0$ for all $s\in(t-\eps,t+\eps)$,
contradicting $u(0,t)=0$.
If $t-\hat\eps\in \Sigma_-$ then, by Proposition~\ref{lemz6}(ii), $u$ is a classical solution on $[0,\frac12]\times(t-\hat\eps,t+\hat\eps)$,
contradicting $u_x(0,t)=\infty$.

 The last statement of the proposition follows from the fact that, 
when $\|\phi\|_\infty\le c_p$ or $\phi$ is symmetric, the condition $u(1,t)\ne U^*(1)$ is satisfied for all $t\in \mathcal{T}$, 
owing respectively to \eqref{phiinfty} or to $u(1,t)=u(0,t)=0$.
\end{proof} 

We end this section with the following variant of Lemma~\ref{lemz20}, 
concerning the number of intersections of two solutions, which will be useful in the proof
of Theorem~\ref{conjz1}.

\begin{lem}\label{lemz20B} 
Let $0<b\le 1$, $\phi_1\in X_1$, $\phi_2\in X_b$,
and denote by $u,v$ the corresponding viscosity solutions of \eqref{vhj1}, 
respectively with $\Omega=(0,1)$ and $\Omega=(0,b)$.
Let $t_2>t_1>0$ and assume that
$$z_{|[0,b]}(u(\cdot,t)-v(\cdot,t))\ \hbox{ is finite and constant on $[t_1,t_2]$.}$$
Then there exists $a\in (0,b)$ such that 
$$\hbox{$u>v$ in $(0,a]\times (t_1,t_2]$}$$
or
$$\hbox{$u<v$ in $(0,a]\times (t_1,t_2]$.}$$
 \end{lem}

\begin{proof} [Proof of Lemma~\ref{lemz20B}]
It is completely similar to the proof of Lemma~\ref{lemz20}, replacing $U^*$ with $v$.
\end{proof}

\section{Proof of Theorem~\ref{thmz1}} 

We first give a simpler proof of Theorem~\ref{thmz1}(i)  in a special but already representative case.

\begin{proof} [Proof of Theorem~\ref{thmz1}(i) under the assumption $\|\phi\|_\infty\le c_p$.]
Set $T=T(\phi)$ and fix $t_0=0$ if $N(0)<\infty$, or any $t_0\in (0,T)$ otherwise. 
By Propositions~\ref{ZeroMonot}(ii) and \ref{zerob2}, $N(t)$ is nonincreasing on $[t_0,\infty)$ and we have
$N(t^-)\ge N(t)+1$ for all $t\in \mathcal{T}$.
Since $\mathcal{T}\cap [0,T)=\emptyset$, it follows that $\#\mathcal{T}\le N(t_0)$.
This proves the assertion.
\end{proof} 

We now turn to the proof of Theorem~\ref{thmz1}(i) in the general case.
Since $N$ need not be monotone in general 
 (cf.~Proposition~\ref{PropNonmonot}),
the proof is more delicate.  
We shall show that, around any transition time $t\in \mathcal{T}$, the number of intersections of the approximate solutions $u_k$
with regular steady states $U_j$ (cf.~\eqref{defUa}) has to drop for suitably large $j,k$.
This is the purpose of the following proposition.
In what follows, for any positive integers $j, k$, we set
$$Z_{k,j}(t)=z_{[0,1]}\bigl(u_k(\cdot,t)-U_j\bigr),\quad t\ge 0.$$

\begin{proposition} \label{approxdrop}
Let $t\in \mathcal{T}$. For any $\tau\in (0,t)$, there exist $j_0\ge 1$ and 
a sequence of integers $(\kappa_j)_{j\ge j_0}$ (depending on $t, \tau$) such that,
 for all $j\ge j_0$ and $k\ge\kappa_j$,
\be \label{approxdropB}
\hbox{$Z_{k,j}$ is finite and nonincreasing on $(0,\infty)$}
\ee
 and 
\be \label{approxdropA}
Z_{k,j}(t-\tau)\ge Z_{k,j}(t+\tau)+1.
\ee
\end{proposition}

\begin{proof} 
 We need to carefully control the zeros of $u_k-U_j$ and of $u-U^*$ near $x=0$, near $x=1$, 
and on the remaining interval.
We thus proceed in three steps.

{\bf Step~1.} {\it Behavior near $x=1$.}
We consider two cases.

{\it Case A.} $u(1,t)<c_p$. Then there exists $\eps\in(0,\tau)$ such that $u(1,s)<c_p$ for all $s\in [t-\eps,t+\eps]$.
We may thus choose $b\in (0,1)$ and $j_1\ge 1$ such that
\be \label{approxdrop1}
u(\cdot,s)-U^*<0 \quad\hbox{in $[b,1]$ for all $s\in[t-\eps,t+\eps]$}
\ee
and $u_k(\cdot,s)-U_j<0$ in $[b,1]$ for all $s\in[t,t+\eps]$, $j\ge j_1$ and $k\ge 1$, hence
\be \label{approxdrop2}
Z_{k,j}(s)=z_{[0,b]}\bigl(u_k(\cdot,s)-U_j\bigr)\ \hbox{for all $s\in[t-\eps,t+\eps]$, $j\ge j_1$ and $k\ge 1$.}
\ee

{\it Case B.} $u(1,t)\ge c_p$.
As a consequence of Lemma~\ref{lemapprox} (applied to the reflected solution $\tilde u(x,t)=u(1-x,t)$),
there exist $b\in (0,1)$ and $k_1\ge 1$ such that,  for all~$s>0$,
$$\begin{aligned}
u(1,s)>0 
&\ \Longrightarrow\ \hbox{$\partial_x u_k(\cdot,s)\le 0$ on $[b,1)$ for all $k\ge k_1$}\\
&\ \Longrightarrow\ \hbox{$u_x(\cdot,s)\le 0$ on $[b,1)$.}
\end{aligned}$$
We may thus choose $\eps\in(0,\tau)$ small enough so that 
\be \label{approxdrop3}
u(b,s)\ge u(1,s)>U^*(b) \quad\hbox{for all $s\in[t-\eps,t+\eps]$.}
\ee
Since $u_k\to u$ locally uniformly on $(0,1)\times[0,\infty)$ 
as $k\to\infty$, there exists $k_2\ge k_1$, such that
for all $s\in[t-\eps,t+\eps]$, $j\ge 1$ and $k\ge k_2$, the function $u_k(\cdot,s)-U_j$ is 
decreasing on $[b,1]$, positive at $x=b$ and negative at $x=1$ (recalling that $u_k(1,s)=0$).
Consequently
\be \label{approxdrop4}
Z_{k,j}(s)=z_{[0,b]}\bigl(u_k(\cdot,s)-U_j\bigr)+1\ \hbox{for all $s\in[t-\eps,t+\eps]$, $j\ge 1$ and $k\ge k_2$.}
\ee

In each case, by Lemma~\ref{zerob} and Proposition~\ref{zerob2} there exists $t_1\in (t,t+\eps)$ such that 
\be\label{hypNconst2}
\hbox{$N_b(\sigma)-1\ge N_b(t)=N_b(s)$ \ for $t-\eps\le\sigma<t<s\le t_1$.}
\ee

{\bf Step~2.} {\it Behavior near $x=0$.}
We set $t_2=(t+t_1)/2$ and consider the cases $t\in \Sigma_-$ and $t\in\Sigma_+$ separately.

{\it Case 1.} $t\in \Sigma_-$. By Proposition~\ref{lemz6}(ii) and \eqref{hypNconst2}, $u$ is a classical solution on~$[0,\frac12]\times(t,t_1]$.
Consequently, since 
 $\sup_{[0,\frac12]\times[t_2,t_1]}|u_x|<\infty$,
there exist $a\in (0,b)$ and $j_2\ge j_1$ such that
\be \label{approxdrop5}
u<U_{j_2}<U^* \quad\hbox{in $(0,a]\times [t_2,t_1]$,}
\ee
hence
\be \label{approxdrop6}
u_k(\cdot,s)-U_j<0 \quad\hbox{in $(0,a]$ for all $s\in [t_2,t_1]$, $j\ge j_2$ and $k\ge 1$.}
\ee

{\it Case 2.} $t\in \Sigma_+$. 
By Lemma~\ref{lemz20} and \eqref{hypNconst2} there exists $a\in (0,1)$ such that
\be \label{approxdrop7}
u>U^* \quad\hbox{in $(0,a]\times (t,t_1]$.}
\ee 
hence in particular $[t,t_1]\subset\Sigma_+\subset\mathcal{S}$ by \eqref{SigplusT}. We then deduce from 
Lemma~\ref{lemapprox} that for each $j\ge 1$, there exist $\eta_j\in (0,a)$ and $\kappa_1(j)\ge k_2$ such that
$$u_k> jx\ge U_j \quad\hbox{in $(0,\eta_j]\times(t,t_1]$ for all $k\ge \kappa_1(j)$.}$$
Since $u_k\to u$ locally uniformly on $(0,1)\times[t,t_1]$ as $k\to\infty$ and $U>U_j$ on $(0,1]$, 
we infer from \eqref{approxdrop7} the existence of $\kappa_2(j)\ge \kappa_1(j)$ such that
$u_k> U_j$ in $[\eta_j,a]\times[t,t_1]$ for all $k\ge \kappa_2(j)$. Therefore, we have
\be \label{approxdrop8}
u_k(\cdot,s)-U_j>0 \quad\hbox{in $(0,a]$ for all $s\in [t,t_1]$, $j\ge 1$ and $k\ge \kappa_2(j)$.} 
\ee

{\bf Step~3.} {\it Conclusion.}
For any integer $j\ge 1$, there exists $\kappa_3(j)\ge \kappa_2(j)$ 
such that $U_j$ is a steady state of problem \eqref{vhj1k} for all $k\ge \kappa_3(j)$.
Since $u_k, U_j$ are smooth, $u_k=U_j=0$ at $x=0$ and $u_k=0\ne U_j$ at $x=1$, it follows from the zero number principle that
\be \label{approxdrop9}
\hbox{$Z_{k,j}$ is finite and nonincreasing in $(0,\infty)$ for all $j\ge 1$ and $k\ge \kappa_3(j)$.}
\ee
Also, since $u$ and $U^*$ are classical solutions of $u_t-u_{xx}=|u_x|^p$ on $[a,b]$, in view of 
\eqref{approxdrop1}, \eqref{approxdrop3}, \eqref{approxdrop5}, \eqref{approxdrop7},
the function $z_{[a,b]}\bigl(u(\cdot,t)-U^*\bigr)$ is finite and nonincreasing on $(t_2,t_1)$ and it drops at any time $s$ 
such that $u(\cdot,s)-U^*$ has a degenerate zero.
Consequently, there exist $\hat t\in (t_2,t_1)$ such that all zeros of $u(\cdot,\hat t)-U$  on $[a,b]$
are nondegenerate.
Since $u_k(\cdot,\hat t)-U_j$ converges to $u(\cdot,\hat t)-U^*$ in $C^1([a,b])$ as $k,j\to\infty$, 
there exists $k_3\ge \max(j_2,k_2)$ such that
$$z_{[a,b]}\bigl(u_k(\cdot,\hat t)-U_j\bigr)=z_{[a,b]}\bigl(u(\cdot,\hat t)-U^*\bigr) \quad\hbox{ for all $k,j\ge k_3$.}$$
For any $j\ge k_3$ and $k\ge \max(k_3,\kappa_3(j))$, using \eqref{approxdrop5}-\eqref{approxdrop8}, we obtain
\be \label{approxdrop12}
z_{[0,b]}\bigl(u_k(\cdot,\hat t)-U_j\bigr)=z_{[0,b]}\bigl(u(\cdot,\hat t)-U^*\bigr)=N_b(\hat t).
\ee
On the other hand, by Lemma~\ref{lemz1}, there exists $j_0\ge k_3$ such that
 $z_{[0,b]}\bigl(u_k(\cdot,t-\eps)-U_j\bigr)\ge N_b(t-\eps)$ for all $k,j\ge j_0$.
 For any $j\ge j_0$ and $k\ge \kappa_j:=\max(j_0,\kappa_3(j))$, this along with 
 \eqref{hypNconst2}, \eqref{approxdrop12} guarantees that
\be \label{approxdrop10}
z_{[0,b]}\bigl(u_k(\cdot,t-\eps)-U_j\bigr) \ge N_b(t)+1 \ge N_b(\hat t)+1=z_{[0,b]}\bigl(u_k(\cdot,\hat t)-U_j\bigr)+1.
\ee
In case A, by \eqref{approxdrop2}, \eqref{approxdrop9}, we get
$$Z_{k,j}(t-\eps)\ge Z_{k,j}(\hat t)+1\ge Z_{k,j}(t+\eps)+1.$$
In case B, by \eqref{approxdrop4} and \eqref{approxdrop9}, \eqref{approxdrop10} yields
$$\begin{aligned}
Z_{k,j}(t-\eps)-1
&= z_{[0,b]}\bigl(u_k(\cdot,t-\eps)-U_j\bigr) \\
& \ge z_{[0,b]}\bigl(u_k(\cdot,\hat t)-U_j\bigr)+1= Z_{k,j}(\hat t)\ge Z_{k,j}(t+\eps).
\end{aligned}$$ 
In view of \eqref{approxdrop9} and $\eps\in(0,\tau)$, this proves the proposition.
\end{proof} 

\begin{proof} [Proof of Theorem~\ref{thmz1}(i)  in the general case]
Set $T=T(\phi)$ and fix $\tau_0=0$ if $N(0)<\infty$, or any $\tau_0\in (0,T)$ otherwise. 
Pick $t_0\in (\tau_0,T)$. Then there exists $k_0\ge 1$ such that $u_k=u$ on $[0,1]\times [0,t_0]$ for all $k\ge k_0$.
By Lemma~\ref{lemz1b}, we deduce the existence of $k_1\ge k_0$ such that
\be \label{approxdrop11}
Z_{k,j}(t_0)=z(u(\cdot,t_0)-U_j)\le N(\tau_0)\quad \hbox{ for all $k,j\ge k_1$.}
\ee

Let $m\ge 1$ and assume that $t_1,\dots,t_m\in\mathcal{T}\subset [T,\infty)$, 
with $t_1<\dots<t_m$ if $m\ge 2$. Set $\tau=\frac12\min_{1\le i\le m} (t_i-t_{i-1})$.
For $1\le i\le m$, let $j_{0,i}\ge 1$ and $(\kappa_j^i)_{j\ge j_{0,i}}$ be given by Proposition~\ref{approxdrop} applied with $t=t_i$.
Choose
$$j=\max(k_1,j_{0,1},\dots,j_{0,m}),\quad k=\max(k_1,\kappa_j^1,\dots,\kappa_j^m)$$
and set $s_i=(t_i+t_{i-1})/2$ for all $i\in \{1,\dots,m\}$ and $s_{m+1}=t_m+\tau$.
Then it follows from \eqref{approxdropB}-\eqref{approxdropA} that
$$Z_{k,j}(s_i)-Z_{k,j}(s_{i+1})\ge Z_{k,j}(t_i-\tau)-Z_{k,j}(t_i+\tau)\ge 1,\quad i=1,\dots,m,$$
hence $m\le Z_{k,j}(s_1)-Z_{k,j}(s_{m+1})\le Z_{k,j}(t_0)\le N(\tau_0)$ by \eqref{approxdrop11}.
This proves the assertion.
\end{proof}

\begin{proof} [Proof of Theorem~\ref{thmz1}(ii)]
Let $t_1,t_2\in \mathcal{T}$ be such that $t_1<t_2$ and $(t_1,t_2)\cap \mathcal{T}=\emptyset$.
 We consider two cases.

{\it Case 1:} there exists $t\in(t_1,t_2)$ such that $u(0,t)>0$.
Let 
$$\tau:=\min\bigl\{s>t;\ u(0,s)=0\bigr\}\in (t,t_2].$$
Then $u(0,s)>0$ for all $s\in (t,\tau)$ and, 
by Lemma~\ref{basic-prop0}(ii), we have
$$|u(x,s)-u(0,s)-U^*(x)| \le K x^2, \qquad 0< x\le \textstyle\frac12.$$
Letting $s\to \tau$ and using the continuity of $u$, we obtain
$$|u(x,\tau)-U^*(x)| \le K x^2, \qquad 0< x\le \textstyle\frac12.$$
Therefore $u_x(0,\tau)=\infty$, hence $\tau\in\mathcal{T}$.
Since $(t_1,t_2)\cap \mathcal{T}=\emptyset$, it follows that $\tau=t_2$.
Similarly we obtain $\max\{s<t;\ u(0,s)=0\}=t_1$, 
hence $u(0,\cdot)>0$ on $(t_1,t_2)$.

{\it Case 2:} $u(0,t)=0$ for all $t\in(t_1,t_2)$.
We claim that 
\be \label{bounduxeps}
\sup_{(0,3/4]\times[t_1+\eps,t_2-\eps]} |u_x|<\infty,\quad\hbox{ for each $\eps>0$.} 
\ee
If \eqref{bounduxeps} fails, in view of \eqref{SGBUprofileUpperEst2}, \eqref{SGBUprofileUpperEst3}, there exist $\eps>0$ and 
sequences $t_j\in[t_1+\eps,t_2-\eps]$ and $x_j\to 0$ such that $u_x(x_j,t_j)\to\infty$.
 Passing to a subsequence, we may assume that $t_j\to \tau\in [t_1+\eps,t_2-\eps]$. 
By Lemma~\ref{basic-prop0}(ii), it follows that $u_x(0,\tau)=\infty$, hence $\tau\in\mathcal{T}$: a contradiction.

Finally, by \eqref{bounduxeps} and parabolic estimates
we conclude that $u$ is a classical solution on $[0,\frac12]\times(t_1,t_2)$.
 \end{proof}

\section{Proof of Theorem~\ref{conjz1}} 

It is convenient to first state and prove the following special case of Theorem~\ref{conjz1},
whose proof, based on Theorem~\ref{thm1} and on a simple application of zero number, is considerably easier but will serve as a starting point for the general case.

\begin{theorem}\label{thm2} 
Let $p>2$, $n=1$, $\Omega=(0,1)$ and set $t_0=0$.
For any integer $m\ge 1$, there exist $\phi\in X_1$ and times $0<t_1<\dots<t_{2m}\in\mathcal{T}$ 
such that $u$ is classical on $(t_{2m},\infty)$ and, for each $i\in \{1,\dots,m\}$,
$$\hbox{$u$ is classical on $(t_{2i-2},t_{2i-1})$ and of LBC type on $(t_{2i-1},t_{2i})$.}$$
\end{theorem}

\begin{proof} [Proof of Theorem~\ref{thm2}]
It follows from Theorem~\ref{thm1} that there exists $\phi\in X$
such that $u$ has at least $m$ losses and recoveries of boundary conditions.  
Moreover, by inspection of the proof, one can see that $\phi$ can be taken in $X_1$ and chosen 
to have exactly $2m$ intersections with $U^*$. 
The conclusion then follows from Theorem~\ref{thmz1}. \end{proof}

Solutions are described in Theorem~\ref{conjz1} by arbitrary finite sequences
$(\sigma_i)_{1\le i\le \ell-1}$ with values in $\{C,L\}$.
It is obvious that Theorem~\ref{conjz1} can be equivalently reformulated as follows, 
and it will be more conveniently proved under this form.

{\bf Theorem~\ref{conjz1}'.}
{\it 
Let $n=1$, $\Omega=(0,1)$ and let $d$ be a positive integer.
For any finite sequence $(\bar\sigma_i)_{1\le i\le d}\in\N^d$,
there exist $\phi\in X_1$ and times $\bar t_{d+1}>\dots>\bar t_1>0$ such that
$$\begin{aligned}
&\hbox{$u(\cdot,t)$ is classical near $x=0$ for $t\in(0,\bar t_1)\cup(\bar t_{d+1},\infty)$}\\
&\qquad\qquad\qquad\qquad\hbox{and for $t$ in the neighborhood of $\bar t_1,\dots,\bar t_{d+1}$}
\end{aligned}$$
and, for each $i\in\{1,\dots,d\}$, 

$\bullet$ if $\bar\sigma_i=0$, then $u$ is classical 
near $x=0$ in $I_i:=(\bar t_i,\bar t_{i+1})$ except for a single time;

$\bullet$ if $\bar\sigma_i\ge 1$, then $\{t\in I_i,\ u(0,t)>0\}$ is 
a nonempty open subinterval of $I_i$ minus $\bar\sigma_i-1$ times.

\noindent Moreover, $u$ is classical up to $x=1$ for all times, i.e.
\be\label{regulat1}
\hbox{$u\in C^{2,1}([\frac12,1]\times(0,\infty))$ and $u=0$ on $\{1\}\times(0,\infty)$.}
\ee
}

The integer $\bar\sigma_i$ represents the number of LBC bumps in the interval $I_i$
and the cases $\bar\sigma_i=0, 1$ or $\ge 2$ respectively correspond 
to the three possible behaviors within each interval $I_i$,
namely: GBU without LBC, single loss and recovery of boundary conditions, and
bouncing (possibly multiple).
The times $\bar t_i$ and indices $\bar\sigma_i$, along with the transition times $t_i$, are represented on Fig.~\ref{FigMixed2}
for a typical example.

\begin{figure}[h]
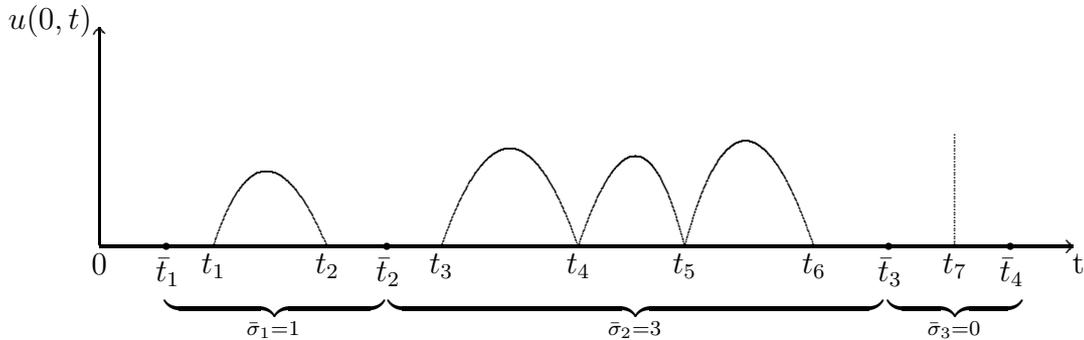

\[
\beginpicture
\setcoordinatesystem units <1cm,1cm>
\setplotarea x from -6 to 6, y from -1 to 3.5

\setdots <0pt>
\linethickness=1pt
\putrule from -5 0 to 7.8 0
\putrule from -5 0 to -5 2.9

\setdots <1pt> 

\plot 6.25 0  6.25 1.5 /

\setdots <0pt> 

\setquadratic

\plot -3.5 0 -2.8 1 -2 0  /

\plot -0.5 0 0.4 1.3 1.3 0  /
\plot 1.3 0 2.05 1.2 2.7 0  /
\plot 2.7 0 3.5 1.4 4.4 0  /

\put {t} [lt] at 7.8 -.1
\put {$u(0,t)$} [rc] at -5.1 3

\setlinear 
\setdots <0pt> 
\plot 7.725 -.075 7.8 0 /
\plot 7.725 .075 7.8 0 /
\plot -5.075 2.825 -5 2.9 /
\plot -5 2.9 -4.925 2.825 /

\put {$\underbrace{\phantom{AAAAAAAA11}}$}  [ct] at -2.7 -0.4
\put {$\underbrace{\phantom{AAAAAAAAAAAAAAAAAAAAA}}$}  [ct] at 2.05 -0.4
\put {$\underbrace{\phantom{AAAAAa}}$}  [ct] at 6.25 -0.4
\put {$^{\bar\sigma_3=0}$}   [ct] at 6.25 -1
\put {$^{\bar\sigma_2=3}$}   [ct] at 2.05 -1
\put {$^{\bar\sigma_1=1}$}   [ct] at -2.7 -1

\put {$0$}  [ct] at -5 -.1

\put {$\bar t_4$}   [ct] at 7 -.15
\put {$\bar t_3$}   [ct] at 5.4 -.15
\put {$\bar t_2$}   [ct] at -1.2 -.15
\put {$\bar t_1$}   [ct] at -4.1 -.15
\put {$t_7$}   [ct] at 6.25 -.1
\put {$t_6$}   [ct] at 4.4 -.1
\put {$t_5$}   [ct] at 2.7 -.1
\put {$t_4$}   [ct] at 1.3 -.1
\put {$t_3$}   [ct] at -0.5 -.1
\put {$t_2$}   [ct] at -2 -.1
\put {$t_1$}   [ct] at -3.5 -.1

\linethickness=1.5pt

\put {$^{^\bullet}$}  [ct] at 7 0.055
\put {$^{^\bullet}$}  [ct] at 5.4 0.055
\put {$^{^\bullet}$}  [ct] at -1.2 0.055
\put {$^{^\bullet}$}  [ct] at -4.1 0.055

\endpicture
\] 
\caption{The times $\bar t_i$ and indices $\bar\sigma_i$ (here $d=3$)}
\label{FigMixed2}
\end{figure}

{\bf Outline of proof of Theorem~\ref{conjz1}/\ref{conjz1}'.}
The proof is rather long and delicate. It is based on a modification of the proof of Theorems~\ref{thm1}
 and \ref{thm2},
along with deformation, zero number and recursion arguments.

In Theorem~\ref{thm2} (via Theorem~\ref{thm1}), 
we have constructed multibump initial data such that the corresponding solutions satisfy $u(0,t)>0$ 
on $m$ open time intervals $J_i$ (LBC bumps), separated and surrounded 
by nontrivial closed intervals where $u(0,t)=0$.
Moreover, our construction was made recursively, using a suitable rescaling for each bump,
with the spatial bump closest to the boundary being responsible for the first LBC time  interval
 $J_1$ of the solution, and so on.

An additional difficulty now is that, whereas multibump solutions with separated LBC time intervals 
should be rather stable, the phenomena of GBU without LBC or of bouncing are expected to be unstable.
To produce an arbitrary solution as in the statements of Theorem~\ref{conjz1}/\ref{conjz1}', the idea is to deform 
this multibump initial data by performing  one of the following three operations
on each space bump $\mathcal{B}_i$  of the initial data,
numbered from left to right by $i\in \{1,\dots,m\}$:
\begin{itemize}

\item[(1)] Reduce the amplitude of $\mathcal{B}_i$ by multiplying by a number in $(0,1)$,
with aim of squeezing the corresponding LBC interval
 $J_i$ to a single time $t\in\mathcal{T}$. 
The resulting $t$ produced in this way will be a GBU time without LBC,
separating two intervals where $u$ is classical;
\smallskip

\item[(2)] (if $i\le m-1$
and operation (1) is not being made on $\mathcal{B}_{i+1}$)
Link $\mathcal{B}_i$ with its right neighbor $\mathcal{B}_{i+1}$ by a suitable deformation,
with the aim of squeezing the separating interval between the LBC  intervals
$J_i$, $J_{i+1}$ 
to a single time $t\in\mathcal{T}$. 
The resulting $t$ will be a bouncing time.
Note that the operation (2) can be repeated with the next bump(s) on the right to create several consecutive rebounds;
\smallskip

\item[(3)] Leave $\mathcal{B}_i$ unchanged.
\end{itemize}

We will actually make a continuous deformation 
along operations (1) and (2) above,
leading to a suitable $q$-parameter family of initial data
$\Phi_{\mu_1,\dots,\mu_q}$ with $(\mu_1,\dots,\mu_q)\in [0,1]^q$
and the desired solution will then be obtained by iteratively selecting appropriate critical values $\mu^*_1,\dots,\mu^*_q$ of the parameters. The required continuity properties of the critical parameter functions will rely upon zero number arguments
applied to the difference of two solutions.

 In view of the proof, we need two propositions. 
The first one gives the building blocks of our multibump construction,
namely the individual bumps and the linking functions between two bumps.

\begin{proposition}\label{lemconv1}
(i) There exist $0<a_1<a_2<\textstyle\frac14$, $c_2>c_1>0$, $K_1=K_1(p)>0$ 
$K>\tilde K(a_1)$ (cf.~Proposition~\ref{lem2}(i)) and,
for all $\eps\in (0,\frac12)$, there exists a nonnegative function $\psi_\eps\in C^2(\R)$, symmetric with respect to $x=\eps$,
such that
\begin{eqnarray}
&&{\hskip -15mm}{\rm Supp}(\psi_\eps)=[(1-a_2)\eps,(1+a_2)\eps], \label{psiepsP1} \\
\noalign{\vskip 1mm}
&&{\hskip -15mm}\hbox{$\psi_\eps>Kx^\alpha$ on $[(1-a_1)\eps,(1+a_1)\eps]$,}  \label{psiepsP2} \\
\noalign{\vskip 1mm}
&&{\hskip -15mm}\|\psi_\eps\|_\infty \le K_1\eps^\alpha,  \label{psiepsP3} \\
\noalign{\vskip 1mm}
&&{\hskip -15mm}\hbox{$\psi_\eps'>0$ on $((1-a_2)\eps,\eps)$,\qquad
$\psi_\eps''\ge 0$ on $[0,1]\setminus ((1-a_1)\eps,(1+a_1)\eps)$,}  \label{psiepsP4} \\
\noalign{\vskip 1mm}
&&{\hskip -15mm}\hbox{$\psi_\eps-U^*$ has exactly one zero in $[(1-a_2)\eps,\eps)$ and 
one zero in $(\eps,(1+a_2)\eps]$,} \\
\noalign{\vskip 1mm}
&&{\hskip -15mm}\hbox{for any $\lambda\in [0,1]$, $\lambda\psi_\eps-U^*$ has at most 2 zeros in $(0,1]$.}
\label{psiepsPZ}
\end{eqnarray}
Moreover, with $c_2>c_1>0$ given by Proposition~\ref{lem2}(ii), 
there exists $\eta=\eta(p)\in(0,1)$ 
such that for any $\tau\in[1-\eta,1]$,
\be\label{psiepsLBC}
u(\tau\psi_\eps;0,t)>0 \ \hbox{ for all $t\in (c_1\eps^2,c_2\eps^2)$.}
\ee

(ii) Let $\eps\in(0,1)$ and $\bar\eps\in(0,\eps/2]$.
Then $(1+a_2)\bar\eps<(1-a_2)\eps$ (hence the supports of $\psi_{\bar\eps}$ and $\psi_\eps$ are disjoint)
and there exists a function $h=h_{\bar\eps,\eps}\in C^2([0,1])$ such that
\begin{eqnarray}
&&{\hskip -15mm}h\ge\hat\psi:=\psi_{\bar\eps}+\psi_\eps,\quad
h_{|[0,(1+a_1)\bar\eps]}=\psi_{\bar\eps},\quad h_{|[(1-a_1)\eps,1]}=\psi_\eps,\label{psiepsPZ1} \\
\noalign{\vskip 1mm}
&&{\hskip -15mm}h>Kx^\alpha\ \hbox{ and }\ (h-Kx^\alpha)''\ge 0 
\quad\hbox{in $[(1+a_1)\bar\eps,(1-a_1)\eps]$,}  \label{psiepsPZ2} \\
\noalign{\vskip 1mm}
&&{\hskip -15mm}\|h\|_\infty \le K_1\eps^\alpha.  \label{psiepsPZ2b}
\end{eqnarray}
Moreover, 
\be\label{psiepsPZ3}
\hbox{for any $\lambda\in [0,1]$, $(\lambda h+(1-\lambda)\hat\psi)-U^*$ has at most 2 zeros in $[\bar\eps,\eps]$.}
\ee
\end{proposition}

\begin{figure}[h]
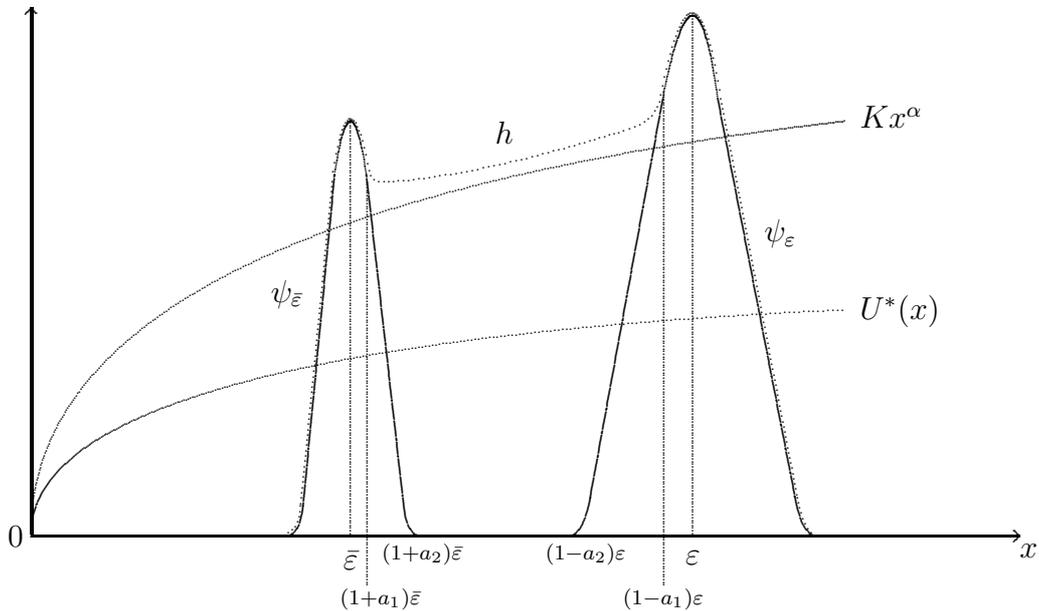

\[
\beginpicture
\setcoordinatesystem units <1cm,1cm>
\setplotarea x from -4 to 10, y from -1 to 6

\setdots <0pt>
\linethickness=1pt
\putrule from -3.5 0 to 9.5 0
\putrule from -3.5 0 to -3.5 7

\setquadratic
% base of \phi_1
\plot  3.85 0.5 3.75 0.16 3.6 0 /
\plot  6.55 0.5 6.65 0.16 6.8 0 /

% base of \phi_2
\plot 1.4 0.4 1.47 0.12 1.6 0 /
\plot 0.076 0.4 0.006 0.12 -0.124  0 /

\setlinear 
% middle of \phi_1
\plot 5.53 5.8 6.55 0.5 /
\plot 4.807 5.8 3.85 0.5 /
% middle of \phi_2
\plot 0.492 4.8 0.076 0.4 /
\plot 0.908 4.8 1.4 0.4 /

\setquadratic

% cap of \phi_1
\plot 4.807 5.8 5.2 6.9 5.53 5.8 /
% cap of \phi_2
\plot 0.492 4.8 0.7 5.5 0.908 4.8 /

% h
\setdots <2pt>
\plot 0.95 4.85 1.01 4.73 1.2 4.7 /
\plot 1.2 4.7 2.55 4.9 4.35 5.39 /
\plot 4.37 5.4 4.63 5.54 4.777 5.8 /
\plot 4.777 5.85 5.2 6.935 5.57 5.8 /
\setlinear 
\plot 5.57 5.8 6.58 0.5 /
\plot 0.46 4.7 0.044 0.4 /
\setquadratic
\plot  6.58 0.5 6.68 0.16 6.83 0 /
\plot 0.47 4.83 0.7 5.53 0.95 4.83 /
\plot 0.04 0.4 -0.03 0.12 -0.24 0 /

%U^*
\setdots <1pt>
\plot -3.5 0  -0.85 2 7.2 3 /
\plot -3.5 0  -0.85 3.5 7.2 5.5 /

% vertical dotted lines
\setlinear 
\setdots <1pt>
\plot 5.2 0 5.2 6.9 /
\plot 0.7 0 0.7 5.5 /

\plot 4.82 -0.65 4.82 5.9 /
\plot 0.92 -0.65 0.92 4.7 /

\put {$x$} [lt] at 9.5 -.1
\put {$\bar\eps$}  [ct] at 0.7 -.2
\put {$\eps$}  [ct] at 5.2 -.2
\put {$^{(1-a_1)\eps}$}  [ct] at 4.82 -.7
\put {$^{(1-a_2)\eps}$}  [ct] at 3.8 -.1
\put {$^{(1+a_1)\bar\eps}$}  [ct] at 1.1 -.7
\put {$^{(1+a_2)\bar\eps}$}  [ct] at 1.65 -.1

\put {$U^*(x)$} [lt] at 7.4 3.2
\put {$Kx^\alpha$} [lt] at 7.4 5.7
\put {$\psi_\eps$} [lt] at 6.14 4.2
\put {$\psi_{\bar\eps}$} [lt] at -0.35 3.4

\put {$h$} [lt] at 2.6 5.5
\put {$0$} [rc] at -3.6 0

\setlinear 
\setdots <0pt> 
\plot 9.425 -.075 9.5 0 /
\plot 9.425 .075 9.5 0 /
\plot -3.575 6.925 -3.5 7 /
\plot -3.5 7 -3.425 6.925 /
\endpicture
\] 

\caption{The functions in Proposition~\ref{lemconv1} (the function $h$ is the dotted curve).}
\label{FigProp}
\end{figure}

In order not to interrupt the main flow of ideas, the proof
of Proposition~\ref{lemconv1}, which is somewhat technical, is postponed to the appendix.

 In the second proposition, we prepare the preliminary sequence of bumps $\phi_1,\dots,\phi_m$.
This sequence is similar to that in the proof of Theorem~\ref{thm1}, 
but needs more precise choices of parameters.
Also, we introduce the corresponding linking functions $h_i$, as well as majorizing functions~$H_i$
that will be used in the deformation step.

  \goodbreak
  
\begin{proposition}\label{propStep1}
Let the constants $L,K_1,c_1,c_2$ and the functions $\gamma_0,\psi_\eps,h_{\bar\eps,\eps}$ be defined in 
Propositions~\ref{lem3bis} and \ref{lemconv1}, and set $c_0=(c_1+c_2)/2$.
For each $m\ge 2$, there exist numbers $0<\eps_1<\dots<\eps_m<\eps_{m+1}<\frac14$,
$\gamma_1,\dots,\gamma_{m+1}>0$ and functions $\phi_1,\dots,\phi_m$, $h_1,\dots,h_m$, $H_1,\dots,H_{m+1}\in C^2([0,1])$, 
with the following properties:
\begin{eqnarray}
&&\phi_i=\psi_{\eps_i},\qquad i\in\{1,\dots,m\}, \label{propStep1Eq1}\\
\noalign{\vskip 1mm}
&&h_i=h_{\eps_i,\,\eps_{i+1}},\qquad i\in\{1,\dots,m-1\},  \label{propStep1Eq2}\\
\noalign{\vskip 1mm}
&&h_m=\phi_m, \\
\noalign{\vskip 1mm}
&&H_i=\max(h_i,\dots,h_m), \qquad i\in\{1,\dots,m\},  \label{propStep1Eq4}\\
\noalign{\vskip 1mm}
&&H_{m+1}=0, \\
\noalign{\vskip 1mm}
&&\hbox{$\|H_1\|_\infty<2^{-\alpha}c_p$ \quad and \quad ${\rm Supp}(H_1)\subset (0,\frac12)$}, 
 \label{propStep1suppH1}\\
\noalign{\vskip 1mm}
&&\gamma_i=\min\Bigl\{\textstyle\frac14\gamma_0\bigl(\frac12\eps_i,\|H_i\|_{C^2([0,1])}\bigr),
\frac{C_1}{2L}\eps_i^2\Bigr\},\qquad i\in\{1,\dots,m+1\}, 
\label{propStep1EqGamma}\\
\noalign{\vskip 1mm}
&&\eps_{i-1}=\min\Bigl\{\textstyle\frac18\eps_i,\bigl(K_1^{-1}\gamma_i\bigr)^{1/\alpha}, 
\frac12\bigl(c_2^{-1}L\gamma_i\bigr)^{1/2}\Bigr\},\qquad i\in\{2,\dots,m+1\}. \label{propStep1EqLast}
\end{eqnarray}
Moreover, the functions $\phi_i$ have disjoint supports and 
\be\label{psiepsPZ1B}
h_i\ge \phi_i+\phi_{i+1},\quad i\in\{1,\dots,m-1\}.
\ee
 Finally, the above remains valid for $m=1$, omitting properties \eqref{propStep1Eq2} and \eqref{psiepsPZ1B}.
\end{proposition}

\begin{figure}[h]
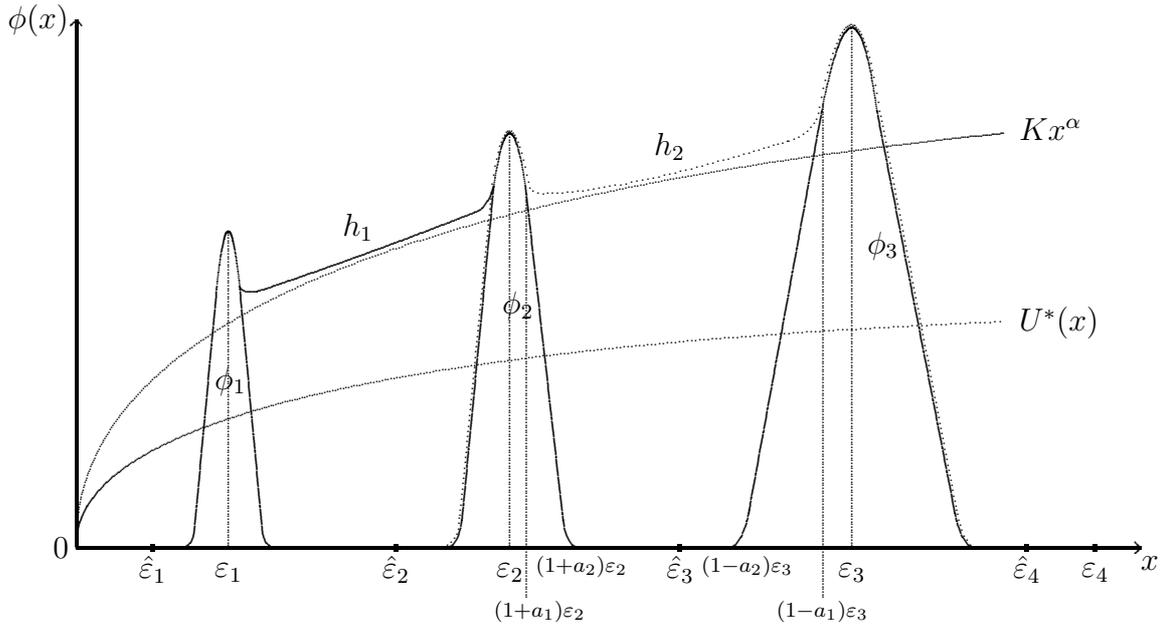

\[
\beginpicture
\setcoordinatesystem units <1cm,1cm>
\setplotarea x from -6 to 6, y from -1 to 6

\setdots <0pt>
\linethickness=1pt
\putrule from -5 0 to 9 0
\putrule from -5 0 to -5 7

\setquadratic
% base of \phi_1
\plot  3.85 0.5 3.75 0.16 3.6 0 /
\plot  6.55 0.5 6.65 0.16 6.8 0 /

% base of \phi_2
\plot 1.4 0.4 1.47 0.12 1.6 0 /
\plot 0.076 0.4 0.006 0.12 -0.124  0 /

% base of \phi_3
\plot -3.45 0.2 -3.5 0.07 -3.6 0 /
\plot -2.55 0.2 -2.5 0.07 -2.4 0 /

\setlinear 
% middle of \phi_1
\plot 5.53 5.8 6.55 0.5 /
\plot 4.807 5.8 3.85 0.5 /
% middle of \phi_2
\plot 0.492 4.8 0.076 0.4 /
\plot 0.908 4.8 1.4 0.4 /
% middle of \phi_3
\plot -3.15 3.46 -3.45 0.2 /
\plot -2.85 3.46 -2.55 0.2 /

\setquadratic

% cap of \phi_1
\plot 4.807 5.8 5.2 6.9 5.53 5.8 /
% cap of \phi_2
\plot 0.492 4.8 0.7 5.5 0.908 4.8 /
% cap of \phi_3
\plot -2.85 3.46 -3 4.2  -3.15 3.46 /

% h_1
\setdots <2pt>
\plot 0.95 4.85 1.01 4.73 1.2 4.7 /
\plot 1.2 4.7 2.55 4.9 4.37 5.4 /
\plot 4.37 5.4 4.63 5.54 4.777 5.8 /
\plot 4.777 5.85 5.2 6.935 5.57 5.8 /
\setlinear 
\plot 5.57 5.8 6.58 0.5 /
\plot 0.46 4.7 0.044 0.4 /
\setquadratic
\plot  6.58 0.5 6.68 0.16 6.83 0 /
\plot 0.47 4.83 0.7 5.53 0.95 4.83 /
\plot 0.04 0.4 -0.03 0.12 -0.24 0 /

% h_2

\setdots <0pt>
\plot -2.62 3.4 -1.2 3.9 0.24 4.45 /
\plot -2.85 3.46 -2.77 3.4 -2.62 3.4  /
\plot 0.24 4.45 0.384 4.56 0.492 4.8 /

%U^*
\setdots <1pt>
\plot -5 0  -2 2 7.2 3 /
\plot -5 0  -2 3.5 7.2 5.5 /

% vertical dotted lines
\setlinear 
\setdots <1pt>
\plot 5.2 0 5.2 6.9 /
\plot 0.7 0 0.7 5.5 /
\plot -3 0 -3 4.2 /

\plot 4.82 -0.65 4.82 5.9 /
\plot 0.92 -0.65 0.92 4.7 /

\put {$x$} [lt] at 9 -.1
\put {$\hat\eps_1$}  [ct] at -4 -.1
\put {$\eps_1$}  [ct] at -3 -.2
\put {$\hat\eps_2$}  [ct] at -0.8 -.1
\put {$\eps_2$}  [ct] at 0.7 -.2
\put {$\hat\eps_3$}  [ct] at 2.94 -.1
\put {$\eps_3$}  [ct] at 5.2 -.2
\put {$^{(1-a_1)\eps_3}$}  [ct] at 4.82 -.7
\put {$^{(1-a_2)\eps_3}$}  [ct] at 3.83 -.1
\put {$^{(1+a_1)\eps_2}$}  [ct] at 1.1 -.7
\put {$^{(1+a_2)\eps_2}$}  [ct] at 1.65 -.1
\put {$\hat\eps_4$}  [ct] at 7.5 -.1
\put {$\eps_4$}  [ct] at 8.4 -.2

\linethickness=1.5pt
\putrule from 8.4 -0.04 to 8.4  0.04
\putrule from 7.5 -0.04 to 7.5  0.04
\putrule from 2.94 -0.04 to 2.94  0.04
\putrule from -0.8 -0.04 to -0.8  0.04
\putrule from -4 -0.04 to -4  0.04

\put {$U^*(x)$} [lt] at 7.4 3.2
\put {$Kx^\alpha$} [lt] at 7.4 5.7
\put {$\phi_3$} [lt] at 5.4 4.2
\put {$\phi_2$} [lt] at 0.6 3.4
\put {$\phi_1$} [lt] at -3.18 2.4
\put {$h_2$} [lt] at 2.6 5.5
\put {$h_1$} [lt] at -1.5 4.43

\put {$\phi(x)$} [rc] at -5.1 7
\put {$0$} [rc] at -5.1 0

\setlinear 
\setdots <0pt> 
\plot 8.925 -.075 9 0 /
\plot 8.925 .075 9 0 /
\plot -5.075 6.925 -5 7 /
\plot -5 7 -4.925 6.925 /

\endpicture
\] 

\caption{The bumps and linking functions in Proposition~\ref{propStep1} (for $m=3$)}
\label{FigMulti}
\end{figure}

\begin{proof}
In view of the application of Proposition~\ref{lem3bis}, bumps will be sequentially added by moving toward the boundary $x=0$.
Therefore, we proceed by backward induction, starting from
\be\label{LBC21B0}
\eps_{m+1}\in \bigl(0,\min\bigl\{\textstyle\frac14,\frac12(c_p/K_1)^{1/\alpha}\bigr\}\bigr),\qquad H_{m+1}=0.
\ee
Take $i\in\{2,\dots,m+1\}$,  and assume that $0<\eps_i<\dots<\eps_{m+1}$, $H_i,\dots,H_{m+1}\in C^2([0,1])$ 
have already been chosen.
If $i\le m$, assume in addition that $\phi_i,\dots,\phi_m$, $h_i,\dots,h_m\in C^2([0,1])$ 
have already been chosen and that they satisfy
\be\label{condHi3}
\hbox{$H_i=\phi_i$ in $[0,(1+a_1)\eps_i]$,\qquad $H_i\ge\phi_i$ in $[0,1]$,}
\ee
where $a_1$ is given by Proposition~\ref{lemconv1}.
We let
$$ 
\gamma_i=\min\Bigl\{\textstyle\frac14\gamma_0\bigl(\frac12\eps_i,\|H_i\|_{C^2([0,1])}\bigr),
\frac{C_1}{2L}\eps_i^2\Bigr\}
$$
where the function $\gamma_0$ is given by Proposition~\ref{lem3bis},
omitting the dependence in $p,\Omega$ for conciseness, and we next choose
$$
\eps_{i-1}=\min\Bigl\{\textstyle\frac18\eps_i,\bigl(K_1^{-1}\gamma_i\bigr)^{1/\alpha}, 
\frac12 \bigl(c_2^{-1}L\gamma_i\bigr)^{1/2}
\Bigr\}.
$$
With the notation of Proposition~\ref{lemconv1}, we then set 
$$\phi_{i-1}=\psi_{\eps_{i-1}},$$
$$
h_{i-1}=\begin{cases}
\hbox{$\phi_m$,} &\hbox{ if $i=m+1$,}\\
\noalign{\vskip 1mm}
\hbox{$h_{\eps_{i-1},\,\eps_i}$,}
&\hbox{ if $i\le m$,}
\end{cases}
$$
and
\be\label{choiceHh}
H_{i-1}=\max(H_i,h_{i-1}).
\ee

Observe that, by \eqref{psiepsPZ1} and \eqref{condHi3},
we have $H_{i-1}=H_i$ on $[\eps_i,1]$, $H_{i-1}=h_{i-1}$ on $[0,\eps_i)$, and 
$H_{i-1}=\phi_i$ on $[(1-a_1)\eps_i,(1+a_1)\eps_i]$.
This in particular guarantees that $H_{i-1}\in C^2([0,1])$. 
Using \eqref{psiepsPZ1} again, we also have $H_{i-1}=\phi_{i-1}$ in $[0,(1+a_1)\eps_{i-1}]$
and, in view of \eqref{choiceHh}, $H_{i-1}\ge\phi_{i-1}$ in $[0,1]$.
Therefore, the new function $H_{i-1}$ satisfies \eqref{condHi3} with $i$ replaced by $i-1$
and we can carry out the iteration from $i=m+1$ down to $i=2$.

Finally, we define 
$\gamma_1=\min\bigl\{\textstyle\frac14\gamma_0\bigl(\frac12\eps_1,\|H_1\|_{C^2([0,1])}\bigr),
\frac{C_1}{2L}\eps_1^2 \bigr\}$.
It is then easy to check that properties \eqref{propStep1Eq1}-\eqref{propStep1EqLast} are true,
whereas the disjointness of the supports of the $\phi_i$ and \eqref{psiepsPZ1B}
respectively follow from \eqref{psiepsP1} and \eqref{psiepsPZ1}.
\end{proof}

\begin{proof}[Proof of Theorem~\ref{conjz1}']
{\bf Step~1.} {\it Preliminary multibump initial data.}
 Our starting point is a multibump initial data, defined by the sum 
$$\hat\phi:=\sum_{i=1}^m \phi_i,$$
where $m$ will be determined herefater and the compactly supported space bumps $\phi_1,\dots,\phi_m$ 
are given by Proposition~\ref{propStep1} (recall that they have disjoint supports).
Suitable deformations will then be applied to $\hat\phi$.
In the rest of the proof, we shall use the various constants and functions defined in Proposition~\ref{propStep1}.

Let us compute the number $m$ of space bumps, and the deformation type
of each bump, in terms of the given sequence $(\bar\sigma_i)$
(recall that we are proving Theorem~\ref{conjz1} under the equivalent form of Theorem~\ref{conjz1}').
 We first inductively define integers $(\kappa_i)_{1\le i\le d+1}$ by setting 
\be\label{defkappa1}
\kappa_1=1,\qquad \kappa_{i+1}=\kappa_i+\max(1,\bar\sigma_i)\ \hbox{ for $i=1,\dots,d$}.
\ee
We then set $m:=\kappa_{d+1}-1$ and define $(\xi_j)_{1\le j\le m}$ by:
\be\label{defkappa}
\begin{cases}
\hbox{$\xi_{\kappa_i}=\bar\sigma_i$} &\hbox{ if $\bar\sigma_i\le 1$,} \\
\noalign{\vskip 1mm}
\hbox{$\xi_{\kappa_i}=\dots=\xi_{\kappa_i+\bar\sigma_i-2}=2$, \ $\xi_{\kappa_i+\bar\sigma_i-1}=1$,}
&\hbox{ if $\bar\sigma_i\ge 2$.}
\end{cases}
\ee
\par\noindent The numbers $\xi_i$ will play the role of deformation indices.
Namely, for all $i\in\{1,\dots,m\}$, if
$$\xi_i=
\begin{cases}
0,&\hbox{ then the amplitude of $\phi_i$ will be reduced,} \\
\noalign{\vskip 1mm}
1,&\hbox{ then $\phi_i$ will be left unchanged,} \\
\noalign{\vskip 1mm}
2,&\hbox{ then $\phi_i$ will be linked with its right neighbor} \\ 
\end{cases}$$
(note that $\xi_m\ne 2$ by \eqref{defkappa}).
As indicated in the outline of proof, these operations are respectively aimed at producing GBU without LBC,
LBC, and bouncing.
 For the example in~Fig.~\ref{FigMixed2}
(where we had $d=3$, $\bar\sigma_1=1$, $\bar\sigma_2=3$ and $\bar\sigma_3=0$),
we need $m=5$ bumps. These bumps and their planned deformations,
as well as the indices $\kappa_i$, $\xi_i$, are depicted in Fig.~\ref{FigMixed3}.

\begin{figure}[h]
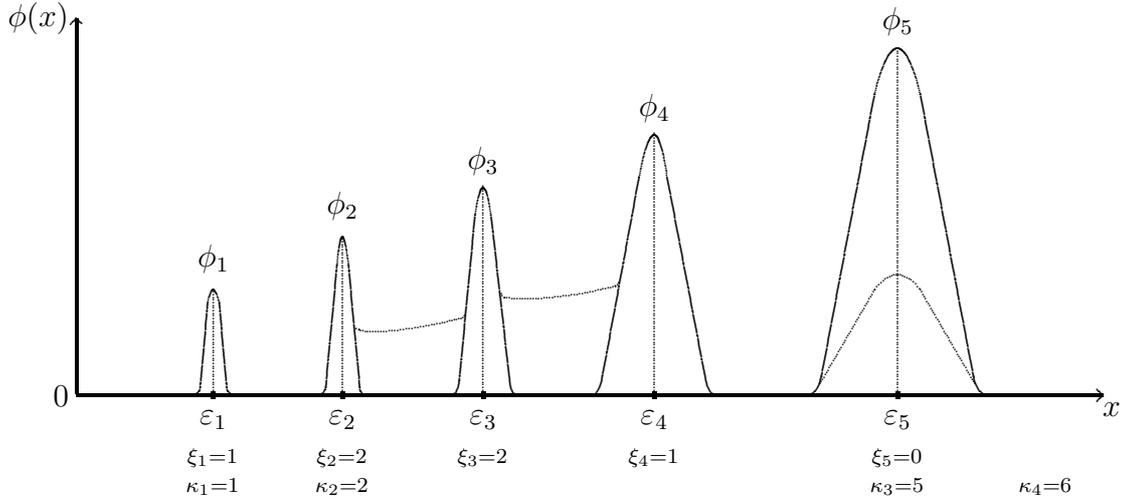

\[
\beginpicture
\setcoordinatesystem units <1cm,1cm>
\setplotarea x from -6 to 6, y from -0.5 to 6

\setdots <0pt>
\linethickness=1pt
\putrule from -5 0 to 8.5 0
\putrule from -5 0 to -5 5

\setquadratic
% base of \phi_1
\plot  4.8 0.25 4.72 0.06 4.6 0 /
\plot  6.8 0.25 6.88 0.06 7 0 /

% base of \phi_2
\plot  1.92 0.25 1.87 0.08 1.8 0 /
\plot  3.27 0.25 3.32 0.08 3.4 0 /

% base of \phi_3
\plot 0.7 0.2 0.73 0.06 0.8 0 /
\plot 0.038 0.2 0.003 0.06 -0.062  0 /

% base of \phi_4
\plot -1.72 0.1 -1.75 0.035 -1.8 0 /
\plot -1.27 0.1 -1.25 0.035 -1.2 0 /

% base of \phi_5
\plot -3.37 0.1 -3.4 0.035 -3.47 0 /
\plot -3.02 0.1 -2.99 0.035 -2.92 0 /

\setlinear 
% middle of \phi_1
\plot 5.5 4 4.8 0.25 /
\plot 6.1 4 6.8 0.25 /

% middle of \phi_1B
\setdots <1pt>
\plot 5.5 1.4 4.78 0.15 /
\plot 6.1 1.4 6.82 0.15 /

\setdots <0pt>
% middle of \phi_2
\plot 2.77 2.9 3.27 0.25 /
\plot 2.403 2.9 1.92 0.25 /
% middle of \phi_3
\plot 0.246 2.4 0.038 0.2 /
\plot 0.454 2.4 0.7 0.2 /
% middle of \phi_4
\plot -1.57 1.73 -1.72 0.1 /
\plot -1.42 1.73 -1.27 0.1 /
% middle of \phi_5
\plot -3.12 1.23 -3.02 0.1 /
\plot -3.27 1.23 -3.37 0.1 /

\setquadratic

% cap of \phi_1
\plot 5.5 4 5.8 4.6 6.1 4 /
% cap of \phi_1B
\setdots <1pt>
\plot 5.5 1.4 5.8 1.6 6.1 1.4 /
\setdots <0pt>
% cap of \phi_2
\plot 2.403 2.9 2.6 3.45 2.77 2.9 /
% cap of \phi_3
\plot 0.246 2.4 0.35 2.75 0.454 2.4 /
% cap of \phi_4
\plot -1.42 1.73 -1.5 2.1  -1.57 1.73 /
% cap of \phi_5
\plot -3.12 1.23 -3.2 1.4 -3.27 1.23 /

% h_1
\setdots <1pt>
\plot 0.59 1.35 0.6 1.315  0.63 1.3 /
\plot 0.63 1.3 1.27 1.3 2.1 1.45 /
\plot 2.09 1.45 2.12 1.47 2.14 1.51 /

% h_2
\setdots <1pt>
\plot -1.28 0.87 -1.29 0.87 -1.35 0.91 /
\plot -1.26 0.86 -0.73 0.87 0.06 1.02 /
\plot 0.08 1.03 0.1 1.05 0.12 1.07 /

% vertical dotted lines
\setlinear 
\setdots <1pt>
\plot  5.8 0  5.8  4.6 /
\plot 2.6 0 2.6 3.5 /
\plot 0.35 0 0.35 2.8 /
\plot -1.5 0 -1.5 2.1 /
\plot -3.2 0 -3.2 1.38 /

\put {$x$} [lt] at 8.5 -.1
\put {$\eps_1$}  [ct] at -3.2 -.2
\put {$\eps_2$}  [ct] at -1.5 -.2
\put {$\eps_3$}  [ct] at 0.35 -.2
\put {$\eps_4$}  [ct] at 2.6 -.2
\put {$\eps_5$}  [ct] at 5.8 -.2

\put {$\phi_1$}  [ct] at -3.2 2
\put {$\phi_2$}  [ct] at -1.5 2.7
\put {$\phi_3$}  [ct] at 0.35 3.3
\put {$\phi_4$}  [ct] at 2.6 4
\put {$\phi_5$}  [ct] at 5.8 5.15

\put {$_{\xi_1=1}$}  [ct] at -3.2 -.7
\put {$_{\xi_2=2}$}  [ct] at -1.5 -.7
\put {$_{\xi_3=2}$}  [ct] at 0.35 -.7
\put {$_{\xi_4=1}$}  [ct] at 2.6 -.7
\put {$_{\xi_5=0}$}  [ct] at 5.8 -.7

\put {$_{\kappa_1=1}$}  [ct] at -3.2 -1.1
\put {$_{\kappa_2=2}$}  [ct] at -1.5 -1.1
\put {$_{\kappa_3=5}$}  [ct] at 5.8 -1.1
\put {$_{\kappa_4=6}$}  [ct] at 7.75 -1.1

\linethickness=1.5pt
\putrule from -3.2 -0.04 to -3.2 0.04
\putrule from -1.5 -0.04 to -1.5 0.04
\putrule from 0.35 -0.04 to 0.35 0.04
\putrule from 2.6 -0.04 to 2.6 0.04
\putrule from 5.8 -0.04 to 5.8 0.04

\put {$\phi(x)$} [rc] at -5.1 5
\put {$0$} [rc] at -5.1 0

\setlinear 
\setdots <0pt> 
\plot 8.425 -.075 8.5 0 /
\plot 8.425 .075 8.5 0 /
\plot -5.075 4.925 -5 5 /
\plot -5 5 -4.925 4.925 /

\endpicture
\] 

\caption{The bumps $\phi_i$ and indices $\kappa_i$, $\xi_i$ corresponding to the example in~Fig.~\ref{FigMixed2}.
The planned deformations are represented by the dotted curves.}
\label{FigMixed3}
\end{figure}

{\bf Step~2.} {\it Deformation of initial data and their intersection properties.} 
Denote by $1\le i_1<\dots<i_q\le m$ the indices of the bumps to be deformed, 
i.e. the elements of $\{1,\dots,m\}$ such that $\xi_i\ne 1$, 
and define the parameter space $\mathcal{J}:=[0,1]^q$.
In the rest of the proof we assume $q\ge 2$. The much easier case $q=1$ can be treated by obvious modifications.
Also, for any $\mu\in\mathcal{J}$ and $j\in \{1,\dots,q\}$, we shall denote $\tilde \mu_j=(\mu_j,\dots,\mu_q)$,
so that we can write $\mu=(\mu_1,\dots,\mu_{j-1},\tilde\mu_j)$ for $j\in \{2,\dots,q\}$.

We construct a $q$-parameter deformation of $\hat\phi$, denoted by $\Phi_\mu$ by setting
\be\label{defPhimu}
\Phi_\mu=\hat\phi+\sum_{1\le \ell\le q\atop \xi_{i_\ell}=0} (\mu_\ell-1)\phi_{i_\ell}
+\sum_{1\le \ell\le q\atop \xi_{i_\ell}=2} \mu_\ell \tilde \phi_{i_\ell},
\quad\mu=(\mu_1,\dots,\mu_q)\in\mathcal{J},
\ee
where $\tilde \phi_{i_\ell}:=h_{i_\ell}-\phi_{i_\ell}-\phi_{{i_\ell}+1}\ge 0$ (owing to \eqref{psiepsPZ1B}).
Observe that the supports of the deformations 
 satisfy
\be\label{defsupportsPhimu}
\begin{cases}
{\rm Supp}(\phi_{i_\ell})\subset I^1_\ell:=\bigl[(1-a_2)\eps_{i_\ell},(1+a_2)\eps_{i_\ell}\bigr]&\hbox{ if $\xi_{i_\ell}=0$,} \\
\noalign{\vskip 1mm}
{\rm Supp}(\tilde \phi_{i_\ell})\subset I^2_\ell:=\bigl[(1+a_1)\eps_{i_\ell},(1-a_1)\eps_{i_\ell +1}\bigr]&\hbox{ if $\xi_{i_\ell}=2$,} 
\end{cases}
\ee
 by \eqref{psiepsP1}, \eqref{psiepsPZ1}, 
\eqref{propStep1Eq1}, \eqref{propStep1Eq2}, and all these intervals are disjoint
(since for all $i\in\{2,\dots,m\}$,\ $\xi_i=0\,\Rightarrow\, \xi_{i-1}\ne 2$, in view of \eqref{defkappa}).

Setting
$$\hat\eps_i=\textstyle\frac34\eps_i,\ \ 1\le i\le  m+1,$$ 
 and recalling that $\eps_{i-1}\le \frac18\eps_i$ by \eqref{propStep1EqLast}, we have
$$\hat\eps_1<\eps_1<\hat\eps_2<\dots<\hat\eps_m<\eps_m<\hat\eps_{m+1}<\eps_{m+1}.$$
We note that for each $i\in\{1,\dots,m\}$, $\Phi_\mu$ satisfies
\be\label{phimucases}
\Phi_\mu=
\begin{cases}
\hat\phi \quad\hbox{ on $[\hat\eps_i,\hat\eps_{i+1}]$}&\hbox{ if $\xi_i=1$} \\
\noalign{\vskip 1mm}
\mu_\ell \hat\phi \quad\hbox{ on $[\hat\eps_i,\hat\eps_{i+1}]$}&\hbox{ if $\xi_i=0,\ i=i_\ell,\ \ell\in\{1,\dots,q\}$} \\
\noalign{\vskip 1mm}
\mu_\ell h_i+(1-\mu_\ell)\hat\phi \quad\hbox{ on $[\eps_i,\eps_{i+1}]$}&\hbox{ if $\xi_i=2,\ i=i_\ell, \ \ell\in\{1,\dots,q\}$.} \\
\end{cases}
\ee
For simplicity, we shall denote the corresponding solutions by
$$u(\mu;x,t)=u(\Phi_\mu;x,t), \quad\mu\in\mathcal{J}.$$
 Owing to \eqref{propStep1suppH1}, we also have
\be\label{controlphimucp}
\|\Phi_\mu\|_\infty<c_p,\quad \mu\in\mathcal{J},
\ee
\be\label{controlsuppphimu}
{\rm Supp}(\Phi_\mu)\subset (0,\textstyle\frac12),\quad \mu\in\mathcal{J},
\ee
 and there exists $a>0$ such that 
\be\label{controlphimuUa}
\Phi_\mu(x)\le U_a(1-x),\quad 0\le x\le 1,\quad \mu\in\mathcal{J},
\ee
where the regular steady state $U_a$ is defined in \eqref{defUa}.
It follows from \eqref{controlphimuUa} and Proposition~\ref{compP} that $u(\mu;x,t)\le U_a(1-x)$ for all $(x,t)\in [0,1]\times[0,\infty)$.
As a consequence of Lemma~\ref{basic-prop0}(ii) (applied to the reflected solution $\tilde u(x,t)=u(1-x,t)$), we have 
\be\label{controlphimuat1a}
\hbox{$u(\mu;\cdot,\cdot)\in C^{2,1}([\frac12,1]\times(0,\infty))$}
\ee
and
\be\label{controlphimuat1}
u(\mu;1,t)=0,\quad t\ge 0,\quad \mu\in\mathcal{J}.
\ee

For all $\mu\in\mathcal{J}$ and $i\in\{1,\dots,m\}$, as a direct consequence of \eqref{phimucases} and Proposition~\ref{lemconv1}, the initial data $\Phi_\mu$ enjoy the following intersection properties
(these properties, for the example from Fig.~\ref{FigMixed2}, are illustrated in Fig.~\ref{FigSign}):
\begin{eqnarray}
&{\hskip -5mm}\xi_i=2 &\Longrightarrow\ z_{|[\eps_i,\eps_{i+1}]}(\Phi_\mu-U^*)\le 2
\ \hbox{ with $\Phi_\mu-U^*>0$ in $\{\eps_i,\eps_{i+1}\}$,} \label{intersecmu1} \\
\noalign{\vskip 1mm}
&{\hskip -5mm}\xi_i=0 &\Longrightarrow\ z_{|[\hat\eps_i,\hat\eps_{i+1}]}(\Phi_\mu-U^*)\le 2
\ \hbox{ with $\Phi_\mu-U^*<0$ in $\{\hat\eps_i,\hat\eps_{i+1}\}$,}  \label{intersecmu2}
\end{eqnarray}
and we also have
\be\label{intersecmu3}
\begin{aligned}
&\xi_i=1 \ \ \Longrightarrow\ z_{|[\eps_i,\hat\eps_{i+1}]}(\Phi_\mu-U^*)=1\\ 
&\qquad\qquad\qquad\qquad \ \hbox{ with $(\Phi_\mu-U^*)(\eps_i)>0>(\Phi_\mu-U^*)(\hat\eps_{i+1})$}
\end{aligned}
\ee
and
\be\label{intersecmu4}
\begin{aligned}
&(\xi_i\ne 0,\ \hbox{ with $\xi_{ i-1}\ne 2$ if $i\ge 2$}) \ \ \Longrightarrow\ z_{|[\hat\eps_i,\eps_i]}(\Phi_\mu-U^*)=1\\ 
 &\qquad\qquad\qquad\qquad \ \hbox{ with $(\Phi_\mu-U^*)(\hat\eps_i)<0<(\Phi_\mu-U^*)(\eps_i)$.}
\end{aligned}
\ee
In particular, by an easy induction argument, it follows that
\be\label{intersecmu5}
 z_{|[0,1]}(\Phi_\mu-U^*)\le 2m. 
\ee

\begin{figure}[h]
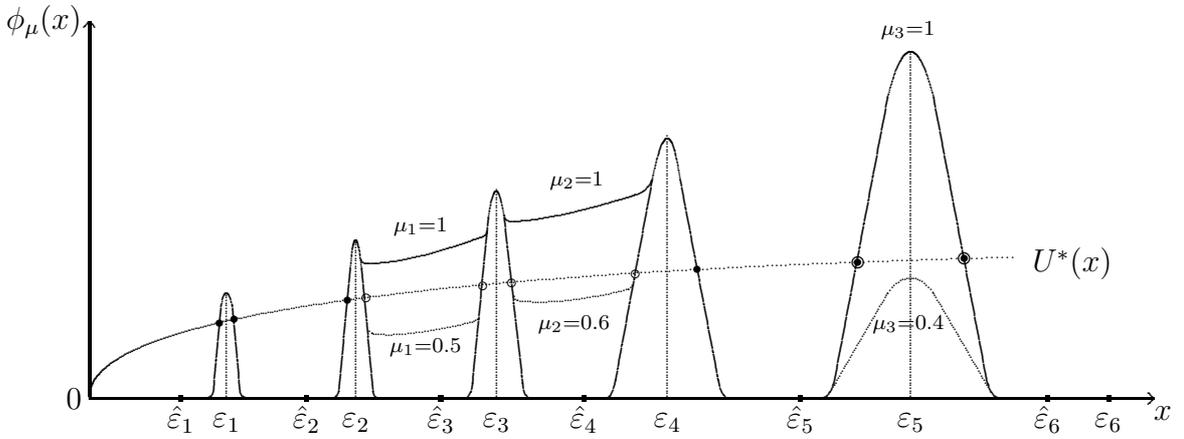

\[
\beginpicture
\setcoordinatesystem units <1cm,1cm>
\setplotarea x from -6 to 6, y from -0.5 to 6

\setdots <0pt>
\linethickness=1pt
\putrule from -5 0 to 9 0
\putrule from -5 0 to -5 5

\setquadratic
% base of \phi_1
\plot  4.8 0.25 4.72 0.06 4.6 0 /
\plot  6.8 0.25 6.88 0.06 7 0 /

% base of \phi_2
\plot  1.92 0.25 1.87 0.08 1.8 0 /
\plot  3.27 0.25 3.32 0.08 3.4 0 /

% base of \phi_3
\plot 0.7 0.2 0.73 0.06 0.8 0 /
\plot 0.038 0.2 0.003 0.06 -0.062  0 /

% base of \phi_4
\plot -1.72 0.1 -1.75 0.035 -1.8 0 /
\plot -1.27 0.1 -1.25 0.035 -1.2 0 /

% base of \phi_5
\plot -3.37 0.1 -3.4 0.035 -3.47 0 /
\plot -3.02 0.1 -2.99 0.035 -2.92 0 /

\setlinear 
% middle of \phi_1
\plot 5.5 4 4.8 0.25 /
\plot 6.1 4 6.8 0.25 /

% middle of \phi_1B
\setdots <1pt>
\plot 5.5 1.4 4.78 0.15 /
\plot 6.1 1.4 6.82 0.15 /

\setdots <0pt>
% middle of \phi_2
\plot 2.77 2.9 3.27 0.25 /
\plot 2.403 2.9 1.92 0.25 /
% middle of \phi_3
\plot 0.246 2.4 0.038 0.2 /
\plot 0.454 2.4 0.7 0.2 /
% middle of \phi_4
\plot -1.57 1.73 -1.72 0.1 /
\plot -1.42 1.73 -1.27 0.1 /
% middle of \phi_5
\plot -3.12 1.23 -3.02 0.1 /
\plot -3.27 1.23 -3.37 0.1 /

\setquadratic

% cap of \phi_1
\plot 5.5 4 5.8 4.6 6.1 4 /
% cap of \phi_1B
\setdots <1pt>
\plot 5.5 1.4 5.8 1.6 6.1 1.4 /
\setdots <0pt>
% cap of \phi_2
\plot 2.403 2.9 2.6 3.45 2.77 2.9 /
% cap of \phi_3
\plot 0.246 2.4 0.35 2.75 0.454 2.4 /
% cap of \phi_4
\plot -1.42 1.73 -1.5 2.1  -1.57 1.73 /
% cap of \phi_5
\plot -3.12 1.23 -3.2 1.4 -3.27 1.23 /

% h_1
\setdots <0pt>
\plot 0.455 2.42 0.5 2.36 0.6 2.35 /
\plot 0.6 2.35 1.27 2.45 2.18 2.7 /
\plot 2.18 2.7 2.31 2.77 2.4 2.9 /

\setdots <1pt>
\plot 0.59 1.35 0.6 1.315  0.63 1.3 /
\plot 0.63 1.3 1.27 1.3 2.1 1.45 /
\plot 2.09 1.45 2.12 1.47 2.14 1.51 /

% h_2
\setdots <0pt>
\plot -1.37 1.8 -1.41 1.81 -1.43 1.85 /
\plot -1.37 1.8 -0.73 1.87 0.18 2.15 /
\plot 0.181 2.15 0.217 2.185 0.23 2.23 /

\setdots <1pt>
\plot -1.28 0.87 -1.29 0.87 -1.35 0.91 /
\plot -1.26 0.86 -0.73 0.87 0.06 1.02 /
\plot 0.08 1.03 0.1 1.05 0.12 1.07 /

\setdots <1pt>
\plot -5 0  -2 1.25 7.2 1.875 /

% vertical dotted lines
\setlinear 
\setdots <1pt>
\plot  5.8 0  5.8  4.6 /
\plot 2.6 0 2.6 3.5 /
\plot 0.35 0 0.35 2.8 /
\plot -1.5 0 -1.5 2.1 /
\plot -3.2 0 -3.2 1.38 /

\put {$x$} [lt] at 9 -.1
\put {$\eps_1$}  [ct] at -3.2 -.2
\put {$\eps_2$}  [ct] at -1.5 -.2
\put {$\eps_3$}  [ct] at 0.35 -.2
\put {$\eps_4$}  [ct] at 2.6 -.2
\put {$\eps_5$}  [ct] at 5.8 -.2
\put {$\eps_6$}  [ct] at 8.4 -.2
\put {$\hat\eps_6$}  [ct] at 7.6 -.1
\put {$\hat\eps_5$}  [ct] at 4.35 -.1
\put {$\hat\eps_4$}  [ct] at 1.5 -.1
\put {$\hat\eps_3$}  [ct] at -0.38 -.1
\put {$\hat\eps_2$}  [ct] at -2.15 -.1
\put {$\hat\eps_1$}  [ct] at -3.8 -.1

\put {$U^*(x)$} [lt] at 7.4 2
\put {${}^{\mu_2=1}$} [lt] at 1.05 3
\put {${}^{\mu_2=0.6}$} [lt] at 0.9 1.1
\put {${}^{\mu_1=1}$} [lt] at -1 2.4
\put {${}^{\mu_1=0.5}$} [lt] at -1.05 0.8
\put {${}^{\mu_3=1}$} [lt] at 5.41 5
\put {${}^{\mu_3=0.4}$} [lt] at 5.3 1.1

\linethickness=1.5pt
\putrule from 7.6 -0.04 to 7.6 0.04
\putrule from 8.4 -0.04 to 8.4 0.04
\putrule from 4.35 -0.04 to 4.35 0.04
\putrule from 1.5 -0.04 to 1.5 0.04
\putrule from -0.38 -0.04 to -0.38 0.04
\putrule from -2.15 -0.04 to -2.15 0.04
\putrule from -3.8 -0.04 to -3.8 0.04

\put {$^{^\bullet}$}  [ct] at -3.275 1.047
\put {$^{^\bullet}$}  [ct] at -3.08 1.1
\put {$^{^\bullet}$}  [ct] at -1.59 1.35
\put {$^\circ$}  [ct] at -1.355 1.39
\put {$^\circ$}  [ct] at 0.18 1.555
\put {$^\circ$}  [ct] at 0.56 1.6
\put {$^\circ$}  [ct] at 2.185 1.715
\put {$^{^\bullet}$}  [ct] at 3.01 1.76
\put {$^{^\bullet}$}  [ct] at 5.12 1.86
\put {$^{^\bullet}$}  [ct] at 6.52 1.91
\put {$\circ$}  [ct] at 5.102 1.89
\put {$\circ$}  [ct] at 6.504 1.944

\put {$\phi_\mu(x)$} [rc] at -5.1 5
\put {$0$} [rc] at -5.1 0

\setlinear 
\setdots <0pt> 
\plot 8.925 -.075 9 0 /
\plot 8.925 .075 9 0 /
\plot -5.075 4.925 -5 5 /
\plot -5 5 -4.925 4.925 /

\endpicture
\] 

\caption{Sign changes of $\phi_\mu-U^*$ (for $m=5$, $q=3$ with $i_1=2$, $i_2=3$, $i_3=5$
and $\xi_2=\xi_3=2$, $\xi_5=0$)}
\label{FigSign}
\end{figure}

On the other hand, we observe that, for all $j\in \{1,\dots,q-1\}$ 
and $\lambda,\mu\in \mathcal{J}$, we have
\be\label{countintersecphi}
\tilde\lambda_{j+1}=\tilde\mu_{j+1} \ \Longrightarrow\  z_{|[0,1]}(\Phi_\mu-\Phi_\lambda)\le j-1. 
\ee
Indeed, in view of $\phi_i,\tilde \phi_i\ge 0$ and \eqref{defPhimu}, 
for each $\ell\in \{1,\dots,j\}$,
the function $\Phi_\mu-\Phi_\lambda$ does not change sign in the interval
$I^1_\ell$ if $\xi_{i_\ell}=0$ and $I^2_\ell$ if $\xi_{i_\ell}=2$ (cf.~\eqref{defsupportsPhimu}),
whereas $\Phi_\mu-\Phi_\lambda\equiv 0$ outside of these $j$ intervals. 

\goodbreak
{\bf Step~3.} {\it LBC and regularity properties of the deformed solutions.} 

Define the times
$$s_i=c_0\eps_i^2,\ s_i^-=c_1\eps_i^2,\ \ s_i^+=c_2\eps_i^2,
\qquad i\in\{1,\dots,m\},$$
$$\hat s_i=\textstyle\frac32L\gamma_i,\ \hat s_i^-=L\gamma_i,\ \hat s_i^+=2L\gamma_i,
\qquad i\in\{1,\dots,m+1\}.$$
 By \eqref{propStep1EqGamma}, \eqref{propStep1EqLast}, we note that (cf.~Fig.~\ref{FigTime1}) 
$$\cdots<\hat s_i^-<\hat s_i<\hat s_i^+<s_i^-<s_i<s_i^+<\hat s^-_{i+1}<\hat s_{i+1}<\hat s^+_{i+1}<\dots$$

\begin{figure}[h]
\[
\beginpicture
\setcoordinatesystem units <1cm,1cm>
\setplotarea x from -6 to 6, y from -1 to 3

\setdots <0pt>
\linethickness=1pt
\putrule from -5 0 to 7.5 0
\putrule from -5 0 to -5 2.7

\setdots <1.3pt> 
\setquadratic
\plot -0.65 0  0.5 1.5 1.65 0  /
\setdots <0pt>

\put {t} [lt] at 7.5 -.1
\put {$u(0,t)$} [rc] at -5.1 2.8

\setlinear 
\setdots <0pt> 
\plot 7.425 -.075 7.5 0 /
\plot 7.425 .075 7.5 0 /
\plot -5.075 2.625 -5 2.7 /
\plot -5 2.7 -4.925 2.625 /

\put {$0$}  [ct] at -5 -.1
\put {C for all $\mu_\ell$}  [ct] at -2.95 0.85
\put {$\overbrace{\phantom{aaaaaaaA}}$}  [ct] at -3 0.4
\put {$\hat s_i^-$}  [ct] at -3.85 -.17
\put {$\hat s_i$}  [ct] at -2.95 -.24
\put {$\hat s_i^+$}  [ct] at -2.15 -.17
\put {$s_i^-$}   [ct] at -0.5 -.15
\put {$s_i$}   [ct] at 0.5 -.29
\put {$s_i^+$}   [ct] at 1.5 -.15
\put {$\hat s_{i+1}^-$}  [ct] at 3.5 -.17
\put {$\hat s_{i+1}$}   [ct] at 4.5  -.24
\put {$\hat s_{i+1}^+$}  [ct] at 5.5 -.17
\put {C for all $\mu_\ell$}   [ct] at 4.5 0.85
\put {$\overbrace{\phantom{aaaaaaaaa}}$}  [ct] at 4.5 0.4
\put {$\underbrace{\phantom{AAAAAAAAAAAAAAAAAAAAAAAA1}}$}  [ct] at 0.7 -0.55

\put {${}^{\hbox{LBC for $\mu_\ell$ close to $1$}}$}  [ct] at 0.6 2.25
\put {${}^{\hbox{C for $\mu_\ell=0$}}$}  [ct] at 0.5 -1.2

\linethickness=1pt
\putrule from -3.9 -0.04 to -3.9 0.04
\put {$^\bullet$}  [ct] at -3 0.058
\putrule from -2.1 -0.04 to -2.1 0.04

\putrule from -0.5 -0.04 to -0.5 0.04
\putrule from 0.5 -0.04 to 0.5 0.04
\putrule from 1.5 -0.04 to 1.5 0.04

\putrule from 3.5 -0.04 to 3.5 0.04
\put {$^\bullet$}  [ct] at 4.5 0.058
\putrule from 5.5 -0.04 to 5.5 0.04

\setlinear 
\setdots <1pt> 
\plot -3 0.06 4.5 0.06 /

\endpicture
\] 

\caption{Time behavior of the deformed solutions for $i=i_\ell$ with~$\xi_{i_\ell}=0$} 
\label{FigTime1}
\end{figure}

We claim that 
for all $\mu\in\mathcal{J}$, we have:
\begin{eqnarray}
&&{\hskip -5mm}
\hbox{If $i\in\{1,m+1\}$ or if $i\in\{2,\dots,m\}$ and $\xi_{i-1}\ne 2$, then} \label{jointLBC1} \\
&&{\hskip -5mm}\quad\hbox{$u(\mu;\cdot,\cdot)$ is classical on $[\hat s_i^-,\hat s_i^+],$} \notag \\
\noalign{\vskip 2mm}
&&{\hskip -5mm}\hbox{If $i\in\{1,\dots,m\}$ and $\xi_i\ne 0$, then} \label{jointLBC2} \\
&&{\hskip -5mm}\quad\hbox{$u(\mu;0,s)>0$ for all $s\in[s_i^-,s_i^+]$,} \notag \\
\noalign{\vskip 2mm}
&&{\hskip -5mm}\hbox{If $\ell\in\{1,\dots,q\}$ and $\xi_{i_\ell}=0$, then} \label{jointLBC3} \\
&&{\hskip -5mm}\quad\begin{cases}
\hbox{$u(\mu;\cdot,\cdot)$ is classical on $[\hat s_{i_\ell},\hat s_{i_\ell+1}]$,} 
&\hbox{ if $\mu_\ell=0$} \notag \\
\noalign{\vskip 2mm}
\hbox{$u(\mu;0,s)>0$ for all $s\in[s_{i_\ell}^-,s_{i_\ell}^+]$,} &\hbox{ if $\mu_\ell$ is close to $1$,} \label{jointLBC4} \\
\end{cases} \\
\noalign{\vskip 2mm}
&&{\hskip -5mm}\hbox{If $\ell\in\{1,\dots,q\}$ and $\xi_{i_\ell}=2$ (hence $i_\ell\le m-1$), then} \\
&&{\hskip -5mm}\quad\begin{cases}
\hbox{$u(\mu;\cdot,\cdot)$ is classical on 
$[\hat s^-_{i_\ell+1},\hat s^+_{i_\ell+1}]$,} &\hbox{ if $\mu_\ell=0$} \\
\noalign{\vskip 1mm}
\hbox{$u(\mu;0,s)>0$ for all $s\in[s_{i_\ell},s_{i_\ell+1}]$,} &\hbox{ if $\mu_\ell$ is close to $1$} \notag \\
\end{cases}
\end{eqnarray}
(these behaviors are illustrated in Fig.~\ref{FigTime1}-\ref{FigTime2}).

\begin{figure}[h]
\[
\beginpicture
\setcoordinatesystem units <1cm,1cm>
\setplotarea x from -6 to 6, y from -1 to 3

\setdots <0pt>
\linethickness=1pt
\putrule from -5 0 to 7.5 0
\putrule from -5 0 to -5 2.9

\setquadratic
\setdots <1.3pt> 
\plot -4.5 0  -3 1.5 -0.25 1  /
\plot -0.25 1  0.75 0.8 1.75 1  /
\plot 1.75 1  4.5 1.5 6 0  /
\setdots <0pt>
\plot -4.1 0  -3 1.1 -1.9 0 /
\plot 3.3 0  4.5 1.1 5.7 0 /

\put {t} [lt] at 7.5 -.1
\put {$u(0,t)$} [rc] at -5.1 3

\setlinear 
\setdots <0pt> 
\plot 7.425 -.075 7.5 0 /
\plot 7.425 .075 7.5 0 /
\plot -5.075 2.825 -5 2.9 /
\plot -5 2.9 -4.925 2.825 /

\put {$0$}  [ct] at -5 -.1
\put {LBC for all $\mu_\ell$}  [ct] at -3 -1.2
\put {$\underbrace{\phantom{aaaaaaaA}}$}  [ct] at -3 -0.55
\put {$s_i^-$}  [ct] at -3.85 -.14
\put {$s_i$}  [ct] at -2.95 -.3
\put {$s_i^+$}  [ct] at -2.05 -.14
\put {$\hat s^-_{i+1}$}   [ct] at -0.5 -.17
\put {$\hat s_{i+1}$}   [ct] at 0.5 -.24
\put {$\hat s^+_{i+1}$}   [ct] at 1.6 -.17
\put {$s^-_{i+1}$}  [ct] at 3.5 -.14
\put {$s_{i+1}$}   [ct] at 4.55  -.3
\put {$s^+_{i+1}$}  [ct] at 5.55 -.14
\put {LBC for all $\mu_\ell$}  [ct] at 4.5 -1.2
\put {$\underbrace{\phantom{aaaaaaaA}}$}  [ct] at 4.5 -0.55

\put {${}^{\hbox{LBC for $\mu_\ell$ close to $1$}}$}  [ct] at 0.6 2.25
\put {${}^{\hbox{C for $\mu_\ell=0$}}$}  [ct] at 0.5 -1.2
\put {$\underbrace{\phantom{aaaaaaaaA}}$}  [ct] at 0.5 -0.55

\linethickness=1pt
\putrule from -3.9 -0.04 to -3.9 0.04
\put {$^\bullet$}  [ct] at -3 0.058
\putrule from -2.1 -0.04 to -2.1 0.04

\putrule from -0.5 -0.04 to -0.5 0.04
\putrule from 0.5 -0.04 to 0.5 0.04
\putrule from 1.5 -0.04 to 1.5 0.04

\putrule from 3.5 -0.04 to 3.5 0.04
\put {$^\bullet$}  [ct] at 4.5 0.058
\putrule from 5.5 -0.04 to 5.5 0.04

\setlinear 
\setdots <1pt> 
\plot -3 0.06 4.5 0.06 /

\endpicture
\] 

\caption{Time behavior of the deformed solutions for $i=i_\ell$ with~$\xi_{i_\ell}=2$}
\label{FigTime2}
\end{figure}

\goodbreak
Let us first check the regularity properties in \eqref{jointLBC1} and 
\eqref{jointLBC4}.
Let $i\in\{2,\dots,m+1\}$, and in case $\xi_{i-1}=2$ assume $i=i_\ell+1$ with $\mu_\ell=0$.
By \eqref{psiepsP1}, \eqref{psiepsPZ1}, \eqref{propStep1Eq2} and \eqref{propStep1Eq4}, we have
$$
\Phi_\mu\le H_i\ \hbox{ on $[\frac12\eps_i,1]$.}
$$
On the other hand, since $(1+a_2)\eps_{i-1}<\frac12\eps_i$ by \eqref{propStep1EqLast}, 
it follows from \eqref{psiepsP1}, \eqref{psiepsP3}, \eqref{psiepsPZ2b} that
$$
\Phi_\mu\le K_1\eps_{i-1}^\alpha\le\gamma_i \hbox{ on $[0,\frac12\eps_i]$}.
$$
Also, if $i=1$, then
$$\hbox{$\Phi_\mu=0$ on $[0,\frac12\eps_1]$.}$$
Therefore, recalling \eqref{propStep1EqGamma}, 
\eqref{controlsuppphimu}, Proposition~\ref{lem3bis}
guarantees that the solution $u(\mu;\cdot,\cdot)$ is classical on the time interval $\bigl[L\gamma_i,L\gamma_0(\frac12\eps_i,\|H_i\|_{C^2([0,1])})\bigr]$ $\supset [\hat s_i^-,\hat s_i^+]$.
This proves \eqref{jointLBC1} and the first part of \eqref{jointLBC4}.

Next assume $i=i_\ell$ with $\xi_{i_\ell}=0$ and $\mu_\ell=0$. 
By \eqref{psiepsP1}, \eqref{psiepsPZ1}, \eqref{propStep1Eq2} and \eqref{propStep1Eq4}, we have
$$
\Phi_\mu\le H_{i+1}\ \hbox{ on $[\frac12\eps_{i+1},1]$.}
$$
On the other hand, in case $i\ge 2$, since $(1+a_2)\eps_{i-1}<2\eps_{i-1}$ and 
$\frac12\eps_{i+1}<(1-a_2)\eps_{i+1}$, we have
\be\label{Phimunull}
\hbox{ $\Phi_\mu=0$ on $[2\eps_{i-1},\frac12\eps_{i+1}]$ 
\ \ and \ \
 $\Phi_\mu\le K_1 \eps^\alpha_{i-1}\le\gamma_i$ 
 on $[0,2\eps_{i-1})$,}
 \ee
Also \eqref{Phimunull} remains true for $i=1$ with $\eps_0:=0$ (omitting the second condition).
In both cases we thus have 
$$\Phi_\mu\le\gamma_i\ \hbox{ on $[0,\frac12\eps_{i+1}]$.}$$
Therefore, in view of \eqref{propStep1EqGamma}, 
\eqref{controlsuppphimu}, Proposition~\ref{lem3bis}
guarantees that $u(\mu;\cdot,\cdot)$ is classical on 
the time interval $\bigl[L\gamma_{i},L\gamma_0(\frac12\eps_{i+1},\|H_{i+1}\|_{C^2([0,1])})\bigr]$.
Since, by \eqref{propStep1EqGamma}, 
this interval contains $[\hat s_{i},\hat s_{i+1}]$, 
the first part of \eqref{jointLBC3} follows.

As for the LBC properties,  \eqref{jointLBC2} and the second part of \eqref{jointLBC3}
follow from \eqref{psiepsLBC}.
To check the second part of \eqref{jointLBC4}, note that
if $i=i_\ell$ with $\xi_{i_\ell}=2$ and $\mu_\ell$ is close to~$1$, then 
 \eqref{psiepsP2} and \eqref{psiepsPZ2} 
guarantee that
$$\Phi_\mu(x)\ge Kx^\alpha 
\hbox{ on $((1-a_1)x,(1+a_1)x)$ for all $x\in [\eps_{i-1},\eps_i]$.}$$
Since $K>\tilde K(a_1)$, the conclusion follows from Proposition~\ref{lem2}(ii).

\goodbreak

{\bf Step~4.} {\it Definition of the critical values.}

For any $\sigma\in\{0,2\}$, any closed subinterval $J$ of $[0,\infty)$ and continuous function $w$ on $\{0\}\times J$, we set
$$\mathcal{E}(w,J,\sigma)=
\begin{cases}
\max_J w &\hbox{ if $\sigma=0$} \\ 
\noalign{\vskip 1mm}
\min_J w &\hbox{ if $\sigma=2$} \\ 
\end{cases}$$
and we define the intervals
$$J_\ell=
\begin{cases}
 [\hat s_{i_\ell},\hat s_{i_\ell+1}]&\hbox{ if $\xi_\ell=0$,} \\
\noalign{\vskip 1mm}
[s_{i_\ell},s_{i_\ell+1}] &\hbox{ if $\xi_\ell=2$,}
\end{cases}
\qquad\ell\in \{1,\dots,q\}$$
(the intervals $J_\ell$ are represented by horizontal dotted lines in Fig.~\ref{FigTime1}-\ref{FigTime2}).

We shall iteratively define $q$ critical parameter functions 
$$\mu_1^*(\tilde\mu_2),\cdots,\mu_{q-1}^*(\tilde\mu_q), \mu_q^*$$
and $q+1$ auxiliary solutions
$$u_0^*(\tilde\mu_1,x,t),u_1^*(\tilde\mu_2,x,t),\dots,u_{q-1}^*(\tilde \mu_q;x,t), u_q^*(x,t)$$
as follows. Set 
\be\label{defujstar0}
u_0^*(\tilde\mu_1,x,t)=u(\tilde\mu_1,x,t).
\ee
Take $j\in \{1,\dots,q\}$ and assume we have already defined $u_{j-1}^*(\tilde \mu_j;x,t)$.
Then, for any $\tilde\mu_{j+1}\in[0,1]^{q-j}$, we consider the set
$$E_j(\tilde\mu_{j+1}):=\Bigl\{\mu_j\in [0,1];\ 
\mathcal{E}\bigl(u_{j-1}^*(\mu_j,\tilde \mu_{j+1};0,\cdot),J_j,\xi_{i_j}\bigr)>0\Bigr\}$$
 (understood as $E_q:=\bigl\{\mu_q\in [0,1];\ \mathcal{E}\bigl(u_{q-1}^*(\mu_q;0,\cdot),J_q,\xi_{i_q}\bigr)>0\bigr\}$
if $j=q$).
We note that $E_j(\tilde\mu_{j+1})\ne\emptyset$ since $1\in E_j(\tilde\mu_{j+1})$ by 
the second parts of \eqref{jointLBC3}-\eqref{jointLBC4}. We then define
$$\mu_j^*(\tilde\mu_{j+1})=\inf E_j(\tilde\mu_{j+1})$$
and
\be\label{defujstar}
u_j^*(\tilde \mu_{j+1};x,t)= u_{j-1}^*(\mu_j^*(\tilde\mu_{j+1}),\tilde\mu_{j+1};x,t).
\ee
 (respectively understood as $\mu_q^*=\inf E_q$ and $u_q^*(x,t)= u_{q-1}^*(\mu_q^*;x,t)$
if $j=q$).
Similarly, in what follows we shall make the convention that the variable $\tilde\mu_{q+1}$ is empty.

For further use, we immediately note by induction that
each of the functions $u_j^*(\tilde \mu_{j+1};\cdot,\cdot)$ is the solution of \eqref{vhj1} 
with initial data $\Phi_\Theta$ for some $\Theta\in\mathcal{J}$, namely:
\be\label{ustarlambdamu}
\begin{aligned}
&\qquad\hbox{for any $j\in\{0,\dots,q\}$ and $\tilde\mu_{j+1}\in[0,1]^{q-j}$, there exists}\\
&\hbox{$(\lambda_1,\dots,\lambda_j)\in [0,1]^j$ such that $ u_j^*(\tilde \mu_{j+1};\cdot,\cdot)=u(\lambda_1,\dots,\lambda_j,\tilde\mu_{j+1};\cdot,\cdot)$}.
\end{aligned}
\ee

{\bf Step~5.} {\it Properties of the critical parameter functions $\mu_j^*(\tilde\mu_{j+1})$.}

Let $Q=\overline\Omega\times(0,\infty)$.
Let $D$ be the constant given by Lemma~\ref{basic-prop0}(ii) (denoted there by $K$),
which can be chosen uniformly with respect to $\mu\in\mathcal{J}$, 
owing to $\sup_{\mu\in\mathcal{J}}\|\phi_\mu\|_{C^2([0,1])}<\infty$.
Set $V_D(x):=U^*(x)-Dx^2$.

We shall prove by joint induction that, for all $j\in \{1,\dots,q-1\}$, 
\begin{eqnarray}
&&{\hskip -1cm}\hbox{for all $\tilde\mu_{j+1}\in[0,1]^{q-j}$ and $\ell\in\{1,\dots,j\}$,} \notag \\  
&&{\hskip -1cm}\hbox{there exists $T_\ell=T_\ell(\tilde \mu_{j+1})\in J_\ell$ s.t.} \label{jointind1}  \\
&&{\hskip -1cm}\begin{cases}
u_j^*(\tilde \mu_{j+1};\cdot,T_\ell)\ge V_D \hbox{ in $(0,1)$ \ and \ }
u_j^*(\tilde \mu_{j+1};0,\cdot)=0 \ \hbox{ in $J_\ell$}
 &\hbox{ if $\xi_{i_\ell}=0$} \\
\noalign{\vskip 0.5mm}
u_j^*(\tilde \mu_{j+1};\cdot,\cdot)\ge V_D \hbox{ in $(0,1)\times J_\ell$ \ and \ }
u_j^*(\tilde \mu_{j+1};0,T_\ell)=0
&\hbox{ if $\xi_{i_\ell}=2$;} \\
\end{cases} \notag \\
\noalign{\vskip 4mm}
&&{\hskip -1cm}\hbox{the map} \ [0,1]^{q-j}\ni \tilde\mu_{j+1}\mapsto \mu_j^*(\tilde\mu_{j+1}) 
\ \hbox{ is continuous;} \label{jointind2} \\
\noalign{\vskip 4mm}
&&{\hskip -1cm} \hbox{the map} \ [0,1]^{q-j}\ni \tilde\mu_{j+1}\mapsto u_j^*(\tilde\mu_{j+1},x,t) 
\ \hbox{is continuous,} \label{jointind3} \\
&&{\hskip -1cm} \hbox{uniformly w.r.t.~$(x,t)\in Q$,} \notag
\end{eqnarray} 
and that
\begin{eqnarray}
&&{\hskip -1cm}\hbox{ for all $\ell\in\{1,\dots,q\}$, there exists $T_\ell\in J_\ell$ s.t.} \notag  \\
&&{\hskip -1cm}\begin{cases}
u_q^*(\cdot,T_\ell)\ge V_D \hbox{ in $(0,1)$ \ and \ }
u_q^*(0,\cdot)=0 \ \hbox{ in $J_\ell$}
 &\hbox{ if $\xi_{i_\ell}=0$} \\
\noalign{\vskip 0.5mm}
u_q^*(\cdot,\cdot)\ge V_D \hbox{ in $(0,1)\times J_\ell$ \ and \ }
u_q^*(0,T_\ell)=0
&\hbox{ if $\xi_{i_\ell}=2$} \\
\end{cases} \label{jointind1b}
\end{eqnarray}
(note that \eqref{jointind1b} corresponds to \eqref{jointind1} with $j=q$). 

To initialize the induction, we observe that the continuity property \eqref{jointind3} is true for $j=0$ 
(indeed, since $u_0^*(\tilde\mu_1,x,t)=u(\Phi_\mu;x,t)$, it follows from \eqref{contdep}
and the fact that $\Phi_\mu$ depends continuously in $L^\infty(\Omega)$-norm on the parameters~$\mu$).
This is the only part of the induction hypotheses \eqref{jointind1}-\eqref{jointind3} that will used in the first induction step~$j=1$.
Thus take $j\in \{1,\dots,q\}$ and, if $j\ge 2$, assume that \eqref{jointind1}-\eqref{jointind3}$_{j-1}$ 
are true.\footnote{Here and below this notation of course means that the formulae
are understood with the index~$j$ replaced by the index in subscript (here $j-1$).}

First observe that, for all $\tilde\mu_{j+1}\in[0,1]^{q-j}$, if $\xi_j=0$ (resp., $\xi_j=2$) then we have
$$\mu_j^*(\tilde\mu_{j+1})=\inf 
\Bigl\{\lambda\in [0,1];\ \max_{t\in J_j} u_{j-1}^*(\lambda,\tilde \mu_{j+1};0,t)>0\Bigr\}
\quad\hbox{(resp., $\displaystyle\min_{t\in J_j}$).}$$
We note that $\mu_j^*(\tilde\mu_{j+1})\in [0,1)$ due to \eqref{jointLBC3}-\eqref{jointLBC4}.
Therefore, there exist sequences $t_k\in J_j$ and $\lambda_+^k\to \mu_j^*(\tilde\mu_{j+1})^+$ 
as $ k \to \infty $ such that, for $ k \ge 1 $,
\be\label{prop1ujstar}
\begin{cases}
u_{j-1}^*(\lambda^k_+,\tilde \mu_{j+1};0,t_k)>0
 &\hbox{ if $\xi_{i_j}=0$} \\
\noalign{\vskip 1mm}
u_{j-1}^*(\lambda^k_+,\tilde \mu_{j+1};0,\cdot)>0\ \hbox{ in $J_j$}
&\hbox{ if $\xi_{i_j}=2$.} \\
\end{cases}
\ee
Also, if $\mu_j^*(\tilde\mu_{j+1})>0$, then there exist sequences $\hat t_k\in J_j$ 
and $\lambda^k_-\to \mu_j^*(\tilde\mu_{j+1})^-$ as $ k \to \infty $ such that, for $ k \ge 1 $,
\be\label{prop1ujstar0}
\begin{cases}
u_{j-1}^*(\lambda^k_-,\tilde \mu_{j+1};0,\cdot)=0\ \hbox{ in $J_j$}
 &\hbox{ if $\xi_{i_j}=0$} \\
\noalign{\vskip 1mm}
u_{j-1}^*(\lambda^k_-,\tilde \mu_{j+1};0,\hat t_k)=0
&\hbox{ if $\xi_{i_j}=2$.} \\
\end{cases}
\ee

To prove \eqref{jointind1}$_j$, or~\eqref{jointind1b} if $j=q$, it suffices to check it in the case $\ell=j$
(since for $\ell\in\{1,\dots,j-1\}$ it directly follows from \eqref{defujstar} and our induction hypothesis).
By Lemma~\ref{basic-prop0}(ii) and \eqref{prop1ujstar}, we have
\be\label{prop2ujstar}
\begin{cases}
u_{j-1}^*(\lambda^k_+,\tilde \mu_{j+1};\cdot,t_k)\ge V_D\ \hbox{ in $(0,1)$}
 &\hbox{ if $\xi_{i_j}=0$} \\
\noalign{\vskip 1mm}
u_{j-1}^*(\lambda^k_+,\tilde \mu_{j+1};\cdot,\cdot)\ge V_D\ \hbox{ in $J_j\times(0,1)$}
&\hbox{ if $\xi_{i_j}=2$.} \\
\end{cases}
\ee
Passing to a subsequence, there exists $T_j=T_j(\tilde\mu_{j+1})\in J_j$ such that $t_k\to T_j$,
 $\hat t_k\to T_j$ as $ k \to \infty $. 
Owing to the continuity property \eqref{jointind3}$_{j-1}$, 
we may pass to the limit $k\to\infty$ in \eqref{prop2ujstar}.
In view of \eqref{defujstar}, this yields the inequalities in \eqref{jointind1}$_j$.
If $\mu_j^*(\tilde\mu_{j+1})>0$, then we similarly get the equalities in \eqref{jointind1}$_j$ by passing to the limit in 
\eqref{prop1ujstar0} whereas, in case $\mu_j^*(\tilde\mu_{j+1})=0$, these equalities directly follow from 
\eqref{jointLBC3}-\eqref{jointLBC4}.

We next assume $j\le q-1$ and prove the upper semicontinuity in \eqref{jointind2}$_j$.
By \eqref{prop1ujstar} and the continuity property \eqref{jointind3}$_{j-1}$, for each $k\ge 1$, 
there exists $\alpha_k>0$ such that, for all $\zeta\in[0,1]^{q-j}$ with $|\tilde\mu_{j+1}-\zeta|\le\alpha_k$, we have
$u_{j-1}^*(\lambda^k_+,\zeta;0,t_k)>0$ if $\xi_j=0$, and
$u_{j-1}^*(\lambda^k_+,\zeta;0,\cdot)>0$ in $J_j$ if $\xi_j=2$.
Consequently, $\mu_j^*(\zeta)\le \lambda^k_+$ and, letting $k\to\infty$, we thus get
$$\limsup_{\zeta\to\tilde\mu_{j+1}} \mu_j^*(\zeta)\le \mu_j^*(\tilde\mu_{j+1}).$$

We now turn to the proof of the lower semicontinuity in \eqref{jointind2}$_j$,
which is more delicate. Assume for contradiction that
\be\label{propujstarcontrad}
\lambda:=\liminf_{\zeta\to\tilde\mu_{j+1}} \mu_j^*(\zeta)<\mu_j^*(\tilde\mu_{j+1}).
\ee
Arguing as above, there exist sequences $\zeta^k\to \tilde\mu_{j+1}$, $\hat\lambda^k_\pm\to \lambda$ 
as $ k \to \infty $ and $\hat t_k\in J_j$ such that, for $ k \ge 1$,
$$
\begin{cases}
u_{j-1}^*(\hat\lambda^k_-,\zeta^k;0,\cdot)=0 \hbox{ in $J_j$ \ and \ }
u_{j-1}^*(\hat\lambda^k_+,\zeta^k;0,\hat t_k)\ge V_D \hbox{ in $(0,1)$}
 &\hbox{ if $\xi_{i_j}=0$} \\
\noalign{\vskip 1mm}
u_{j-1}^*(\hat\lambda^k_+,\zeta^k;0,\cdot)\ge V_D\ \hbox{ in $(0,1)\times J_j$ \ and \ }
u_{j-1}^*(\hat\lambda^k_-,\zeta^k;0,\hat t_k)=0
&\hbox{ if $\xi_{i_j}=2$.} \\
\end{cases}
$$
Passing to a subsequence, there exists $\hat T_j\in J_j$ such that $\hat t_k\to \hat T_j$ as $ k \to \infty $. 
Using the continuity property \eqref{jointind3}$_{j-1}$ and passing to the limit $k\to\infty$, we get
\be\label{prop3ujstar}
{\hskip -4mm} \begin{cases}
u_{j-1}^*(\lambda,\tilde\mu_{j+1};0,\cdot)=0 \hbox{ in $J_j$ and }
u_{j-1}^*(\lambda,\tilde\mu_{j+1};0,\hat T_j)\ge V_D \hbox{ in $(0,1)$}
 &{\hskip -2mm}\hbox{ if $\xi_{i_j}=0$} \\
\noalign{\vskip 1mm}
u_{j-1}^*(\lambda,\tilde\mu_{j+1};0,\cdot)\ge V_D \hbox{ in $(0,1)\times J_j$ and }
u_{j-1}^*(\lambda,\tilde\mu_{j+1};0,\hat T_j)=0
&{\hskip -2mm}\hbox{ if $\xi_{i_j}=2$.} \\
\end{cases}
\ee
Denoting $\mu=\mu_j^*(\tilde\mu_{j+1})$, we set
\be\label{defvw}
\begin{aligned}
v(x,t)&:=u_{j-1}^*(\mu,\tilde\mu_{j+1};x,t)=u_j^*(\tilde \mu_{j+1};x,t), \\
w(x,t)&:=u_{j-1}^*(\lambda,\tilde\mu_{j+1};x,t).
\end{aligned}
\ee
 We then use \eqref{jointind1}$_j$ applied to $u_j^*(\tilde \mu_{j+1};x,t)$,
\eqref{prop3ujstar} and, in case $j\ge 2$,
\eqref{jointind1}$_{j-1}$ applied to $u_{j-1}^*(\lambda,\tilde\mu_{j+1};x,t)$.
It follows that, for all $\ell\in\{1,\dots,j\}$,
\be\label{prop4ujstarJ}
\begin{cases}
v(0,\cdot)=w(0,\cdot)=0 \hbox{ in $J_\ell$} &\hbox{ if $\xi_{i_\ell}=0$} \\
\noalign{\vskip 1mm}
v_x(0,\cdot)=w_x(0,\cdot)=\infty \hbox{ in $J_\ell$} &\hbox{ if $\xi_{i_\ell}=2$} \\
\end{cases}
\ee
and there exist $T_\ell,\hat T_\ell\in J_\ell$ such that
\be\label{prop4ujstarT}
\begin{cases}
v_x(0,T_\ell)=w_x(0,\hat T_\ell)=\infty &\hbox{ if $\xi_{i_\ell}=0$} \\
\noalign{\vskip 1mm}
v(0,T_\ell)=w(0,\hat T_\ell)=0 &\hbox{ if $\xi_{i_\ell}=2$.} \\
\end{cases}
\ee

To reach a contradiction, the idea is now to apply the natural scaling to $ w $ 
and to examine the dropping properties of their number of intersections with the solution $v$.
Namely, for $\beta\in (0,1)$, set 
\be\label{propujstardef}
w_\beta(x,t)=\beta^\alpha w(\beta^{-1}x,\beta^{-2}t),\quad (x,t)\in [0,\beta]\times[0,\infty)
\ee
and note that $w_\beta$ is the viscosity solution of problem \eqref{vhj1} with $\Omega=(0,\beta)$ and 
with initial data $\beta^\alpha w(\beta^{-1}x,0)$. 
Moreover, by \eqref{controlphimuat1}, we have
\be\label{wbetaat1}
w_\beta(\beta,t)=0,\quad t\ge 0.
\ee
{For $\beta\in (0,1]$, we set}
$$\mathcal{N}_\beta(t):=z_{|[0,\beta]}(v(\cdot,t)-w_\beta(\cdot,t)),\quad t>0.$$
By Proposition~\ref{lemz20A} and \eqref{wbetaat1} we have
\be\label{prop8ujstar}
\hbox{for all $\beta\in (0,1)$, \ $\mathcal{N}_\beta(t)$ is finite and nonincreasing on $[0,\infty)$.}
\ee
Also, by \eqref{countintersecphi}, \eqref{ustarlambdamu} and \eqref{defvw}, we have
\be\label{N0jminus1}
\mathcal{N}_1(0)\le j-1.
\ee

Let $\tau_0:=(0,\min(T^*_1,T^*_2))$, where $T^*_1, T^*_2$ denote the classical existence times of $v, w$.
We claim that there exist $\beta_0\in(0,1)$ and $\tau\in(0,\frac12\tau_0)$ such that, 
\be\label{prop5ujstar}
\mathcal{N}_\beta(\tau)\le j-1\quad\hbox{for all $\beta\in [\beta_0,1)$.}
\ee
To show this, we first note that, by \eqref{defPhimu}, \eqref{ustarlambdamu}, \eqref{defvw}
and since $\mu>\lambda$ due to \eqref{propujstarcontrad}, there exists $\eta\in(0,1)$ such that
 $v(\cdot,0)\ge w(\cdot,0)$ in $[\eta,1]$ and $v(\eta,0)>w(\eta,0)$.
By continuity, there exists $\tau_1\in (0,\tau_0)$ such that $v(\eta,t)>w(\eta,t)$ for all $t\in [0,\tau_1]$,
hence
\be\label{prop6ujstar}
\hbox{$v>w$ in $[\eta,1)\times (0,\tau_1]$}
\ee
by the strong maximum principle.
Also, by standard zero number properties for classical solutions we may choose 
$\tau\in (0,\tau_1)$ such that 
\be\label{prop7ujstar}
\hbox{all zeros of $v(\cdot,\tau)-w(\cdot,\tau)$ in $[0,1]$ are nondegenerate.}
\ee
Next extending $w$ by odd reflection at $x=1$, we obtain a function $\tilde w\in C^1([0,2]\times [0,\tau_0))$
and, setting $\tilde w_\beta(x,t)=\beta^\alpha \tilde w(\beta^{-1}x,\beta^{-2}t)$,
we see that $\tilde w_\beta(\cdot,\tau)\to \tilde w(\cdot,\tau)$ in $C^1([0,1])$ as $\beta\to 1^-$.
By \eqref{N0jminus1} and \eqref{prop7ujstar}, it follows that 
$$z(\tilde w_\beta(\cdot,\tau)-v(\cdot,\tau))=\mathcal{N}_1(\tau)\le\mathcal{N}_1(0)\le j-1$$ 
for $\beta$ close to $1$.
Assuming also $\beta>\eta$ and using \eqref{prop6ujstar},
along with the fact that $v>0>\tilde w_\beta$ for $x\in (\beta,1)$, we deduce \eqref{prop5ujstar}.

Next, writing $J_\ell=[t_\ell^-,t_\ell^+]$, we claim the existence of $\bar\beta\in(\beta_0,1)$ such that 
\be\label{prop13ujstar}
\hbox{for all $\beta\in (\bar\beta,1)$}, \quad \mathcal{N}_\beta(t_\ell^+)\le \mathcal{N}_\beta(t_\ell^-)-1,\quad \ \ell=1,\dots,j.
\ee
To prove the claim, assume for contradiction that there exist $\ell\in\{1,\dots,j\}$
and a sequence $\beta_k\to 1$ with $\beta_k\in (\beta_0,1)$, such that $\mathcal{N}_{\beta_k}(t_\ell^+)\ge \mathcal{N}_{\beta_k}(t_\ell^-)$.
By~\eqref{prop8ujstar}, for each~$k$, 
we have $\mathcal{N}_{\beta_k}(t)=\mathcal{N}_{\beta_k}(t_\ell^-)$ on $J_\ell$,
hence we may apply Lemma~\ref{lemz20B} to deduce the existence of $a_k\in (0,\beta_k)$ such that 
\be\label{prop15ujstar}
\hbox{$v \le w_{\beta_k}$ in $(0,a_k]\times J_\ell$ 
\quad or \quad $v \ge w_{\beta_k}$ in $(0,a_k]\times J_\ell$}.
\ee
As a consequence of \eqref{jointLBC1}, \eqref{jointLBC2},
we have ${\beta_k}^{-2}T_\ell, {\beta_k}^2\hat T_\ell\in J_\ell$ for all $k\ge k_0$ large enough.
It follows from \eqref{prop4ujstarT} and \eqref{prop15ujstar} that
\be\label{prop14ujstar}
\begin{cases}
\hbox{$w_x(0,{\beta_k}^{-2}T_\ell)=\infty$\quad or \quad $v_x(0,{\beta_k}^2\hat T_\ell)=\infty$},
 &\hbox{ if $\xi_{i_\ell}=0$} \\
\noalign{\vskip 1mm}
\hbox{$v(0,{\beta_k}^2\hat T_\ell)=0$ \quad or \quad $w(0,{\beta_k}^{-2}T_\ell)=0$},
&\hbox{ if $\xi_{i_\ell}=2$.}
\end{cases}
\ee
We thus deduce from \eqref{prop4ujstarJ} and \eqref{prop14ujstar} that,
for all $k\ge k_0$, ${\beta_k}^2\hat T_\ell\in \mathcal{T}_v$ or ${\beta_k}^{-2}T_\ell\in \mathcal{T}_w$.
But since the transition sets $\mathcal{T}_v$ and $\mathcal{T}_w$ are finite by~Theorem~\ref{thmz1}, 
this is impossible and claim \eqref{prop13ujstar} follows.

Finally, in view of \eqref{prop8ujstar} and since the intervals $J_\ell$ have disjoint interiors, \eqref{prop13ujstar} contradicts \eqref{prop5ujstar}.
Therefore hypothesis \eqref{propujstarcontrad} cannot be valid and we have proved~\eqref{jointind2}$_j$.

As for property \eqref{jointind3}$_j$ with $j\le q-1$, it follows from \eqref{jointind3}$_{j-1}$, \eqref{jointind2}$_j$ and \eqref{defujstar}.
We have thus proved \eqref{jointind1}-\eqref{jointind3}$_j$ for all $j\in \{1,\dots,q-1\}$ and \eqref{jointind1b}
by induction.

{\bf Step~6.} {\it Conclusion.}
We finally select the critical values of the parameters $\mu:=\Lambda$,
where $\Lambda\in [0,1]^q$ is defined by induction as follows:
\be\label{prop9ujstar} 
\Lambda_q=\mu_q^*  \qquad\hbox{ and }\qquad  \Lambda_{q-j}=\mu_{q-j}^*(\Lambda_{q-j+1},\dots,\Lambda_q),\quad j=1,\dots,q-1.
\ee
Let us check that the solution $u(x,t)=u(\Lambda;x,t)$,
corresponding to the initial data $\phi=\Phi_\Lambda$, has all the required properties.

We observe that, for all $j\in\{0,\dots,q\}$, we have
\be\label{prop10ujstar}
u(x,t)=u^*_j(\tilde\Lambda_{j+1};x,t),
\ee
where we recall that the notation $u_q^*(\tilde\mu_{q+1};x,t)$
is understood as $u_q^*(x,t)$ (cf.~after \eqref{defujstar}).
Indeed \eqref{prop10ujstar} follows by induction, noting that it is true 
for $j=0$ by definition (cf.~\eqref{defujstar0})
and that, if \eqref{prop10ujstar} is true at the level $j-1$ for some $j\in\{1,\dots,q\}$, then
\eqref{defujstar} and \eqref{prop9ujstar} give
$$u^*_j(\tilde\Lambda_{j+1};x,t)=
u_{j-1}^*\bigl(\mu_j^*(\tilde\Lambda_{j+1}),\tilde\Lambda_{j+1};x,t\bigr)=
u^*_{j-1}(\tilde\Lambda_j;x,t)=u(x,t).$$
Then applying \eqref{jointind1b} and \eqref{prop10ujstar} with $j=q$, it follows that, for all $\ell\in\{1,\dots,q\}$,
there exists $T_\ell\in J_\ell$ such that
\be\label{prop11ujstar}
\begin{cases}
u(0,\cdot)=0 \hbox{ in $J_\ell$ \ and \ }
u_x(0,T_\ell)=\infty
 &\hbox{ if $\xi_{i_\ell}=0$} \\
\noalign{\vskip 1mm}
u_x(0,\cdot)=\infty \hbox{ in $J_\ell$ \ and \ }
u(0,T_\ell)=0
&\hbox{ if $\xi_{i_\ell}=2$.}
\end{cases}
\ee
In particular, $T_1,\dots,T_q\in\mathcal{T}$.
Moreover, in view of \eqref{jointLBC1}-\eqref{jointLBC2}, we actually have 
$$T_\ell\in{\rm int}(J_\ell),\quad \ell\in\{1,\dots,q\}.$$
Since $\mathcal{T}$ is finite by Theorem~\ref{thmz1}(i), 
each $T_\ell$ has a neighborhood containing no other point of $\mathcal{T}$
and it follows from Theorem~\ref{thmz1}(ii) that, for some small $\eps>0$,
\be\label{prop12ujstar}
\begin{cases}\hbox{$u$ is classical on $[T_\ell-\eps,T_\ell+\eps]\setminus\{T_\ell\}$}
 &\hbox{ if $\xi_{i_\ell}=0$} \\
\noalign{\vskip 1mm}
\hbox{$u(0,\cdot)>0$ in $[T_\ell-\eps,T_\ell+\eps]\setminus\{T_\ell\}$}
&\hbox{ if $\xi_{i_\ell}=2$,}
\end{cases}
\ee
i.e.~$T_\ell$ is a GBU time without LBC if $\xi_{i_\ell}=0$ and a bouncing time if $\xi_{i_\ell}=2$.
In what follows we denote
$$\Theta=\{T_1,\dots,T_q\}=\Theta_0\cup \Theta_2,$$
where
\be\label{defTheta0}
\Theta_0=\Bigl\{t\in\Theta;\ \hbox{$u$ is classical on $[t-\eps,t+\eps]\setminus\{t\}$}\Bigr\}
\ee
and
\be\label{defTheta2}
\Theta_2=\Bigl\{t\in\Theta;\ \hbox{$u(0,\cdot)>0$ in  $[t-\eps,t+\eps]\setminus\{t\}$}\Bigr\}.
\ee

In order to conclude, it remains to enumerate the other elements of $\mathcal{T}$,
by counting the corresponding drops of the function $N(t)$.
Recall that the indices $\kappa_1,\dots,\kappa_{d+1}$ 
are defined in \eqref{defkappa1}.
We claim that for each $j\in \{1,\dots,k\}$, the interval $(\hat s_{\kappa_j},\hat s_{\kappa_{j+1}})$ 
contains {\it at least}
\be\label{countS1}
\begin{cases}
\hbox{$1$ element of $\Theta_0$} &\hbox{ if $\xi_{\kappa_j}=0$} \\ 
\noalign{\vskip 2mm}
\hbox{$2$ elements of $\mathcal{T}\setminus\Theta$} &\hbox{ if $\xi_{\kappa_j}=1$} \\
\noalign{\vskip 2mm}
\hbox{$(\kappa_{j+1}-\kappa_j-1)$ elements of $\Theta_2$ and $2$ elements of $\mathcal{T}\setminus\Theta$} 
&\hbox{ if $\xi_{\kappa_j}=2$} \\
\end{cases}
\ee
(for the example from~Fig.~\ref{FigSign}, this is illustrated in~Fig.~\ref{FigKappa}).

\begin{figure}[h]
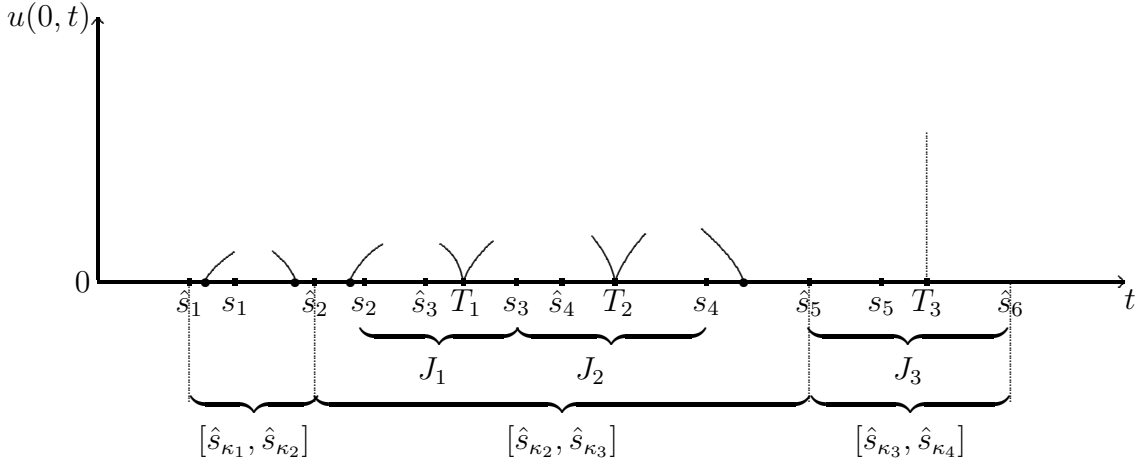

\[
\beginpicture
\setcoordinatesystem units <1cm,1cm>
\setplotarea x from -6 to 6, y from -0.5 to 3.5

\setdots <0pt>
\linethickness=1pt
\putrule from -5 0 to 8.5 0
\putrule from -5 0 to -5 3.5

\put {$\underbrace{\phantom{AAAAAaaaa1}}$}  [ct] at 5.65 -0.3
\put {$J_3$}  [ct] at 5.65 -1  %%J_1
\put {$\underbrace{\phantom{AAAAAAAA}}$}  [ct] at 1.75 -0.3
\put {$J_2$}  [ct] at 1.47 -1
\put {$\underbrace{\phantom{AAAAAAa}}$}  [ct] at -0.52 -0.3
\put {$J_1$}  [ct] at -0.6 -1  %%J_3

\put {$\underbrace{\phantom{AAAAAAAa1}}$}  [ct] at 5.68 -1.2
\put {$[\hat s_{\kappa_3},\hat s_{\kappa_4}]$}  [ct] at 5.68  -1.9

\put {$\underbrace{\phantom{AAAAAAAAAAAAAAAAAAAAA}}$}  [ct] at 1.1 -1.2
\put {$[\hat s_{\kappa_2},\hat s_{\kappa_3}]$}  [ct] at 1.1 -1.9 

\put {$\underbrace{\phantom{AAAAa1}}$}  [ct] at -2.95 -1.2
\put {$[\hat s_{\kappa_1},\hat s_{\kappa_2}]$}  [ct] at -2.95  -1.9

\setquadratic

\put {$^{\bullet}$}  [ct] at -3.58 0.05
\put {$^{\bullet}$}  [ct] at -2.4 0.05
\put {$^{\bullet}$}  [ct] at -1.67 0.05
\put {$^{\bullet}$}  [ct] at 3.5 0.05

\plot -3.6 0 -3.45 0.2 -3.2 0.4  /
\plot -2.7 0.4 -2.5 0.2 -2.4 0  /

\plot -1.7 0 -1.5 0.3 -1.25 0.5  /

\plot -0.5 0.5 -0.3 0.3 -0.2 0  /
\plot -0.2 0 -0.05 0.27 0.2 0.55  /

\plot 1.5 0.6 1.7 0.3 1.8 0  /
\plot 1.8 0 1.95 0.31 2.2 0.64  /

\plot 2.94 0.7 3.3 0.33 3.5 0  /

% vertical dotted lines
\setlinear 
\setdots <1pt>
\plot 5.9 0 5.9 2 /
\plot  4.35 0  4.35 -1.6 /
\plot  -2.15 0  -2.15 -1.6 /
\plot  -3.8 0  -3.8 -1.6 /
\plot  7 0  7 -1.6 /

\put {$t$} [lt] at 8.5 -.1
\put {$s_1$}  [ct] at -3.2 -.2
\put {$s_2$}  [ct] at -1.5 -.2
\put {$s_3$}  [ct] at 0.5 -.2
\put {$s_4$}  [ct] at 3 -.2
\put {$s_5$}  [ct] at 5.3 -.2
\put {$\hat s_6$}  [ct] at 7 -.1
\put {$\hat s_5$}  [ct] at 4.35 -.1
\put {$\hat s_4$}  [ct] at 1.1 -.1
\put {$\hat s_3$}  [ct] at -0.7 -.1
\put {$\hat s_2$}  [ct] at -2.15 -.1
\put {$\hat s_1$}  [ct] at -3.8 -.1

\put {$T_3$}  [ct] at 5.9 -.1
\put {$T_2$}  [ct] at 1.84 -.1
\put {$T_1$}  [ct] at -0.16 -.1

\linethickness=1.5pt
\putrule from -3.2 -0.04 to -3.2 0.04
\putrule from -1.5  -0.04 to -1.5  0.04
\putrule from 0.5 -0.04 to 0.5 0.04
\putrule from 3 -0.04 to 3 0.04
\putrule from 5.3 -0.04 to 5.3 0.04

\putrule from 4.35 -0.04 to 4.35 0.04
\putrule from 1.1 -0.04 to 1.1 0.04
\putrule from -0.7 -0.04 to -0.7 0.04
\putrule from -2.15 -0.04 to -2.15 0.04
\putrule from -3.8 -0.04 to -3.8 0.04

\putrule from 5.9 -0.04 to 5.9 0.04
\putrule from 1.8 -0.04 to 1.8 0.04
\putrule from -0.2 -0.04 to -0.2 0.04

\put {$u(0,t)$} [rc] at -5.1 3.5
\put {$0$} [rc] at -5.1 0

\setlinear 
\setdots <0pt> 
\plot 8.425 -.075 8.5 0 /
\plot 8.425 .075 8.5 0 /
\plot -5.075 3.425 -5 3.5 /
\plot -5 3.5 -4.925 3.425 /

\endpicture
\] 

\caption{Illustration of property \eqref{countS1}
 for the example in~Fig.~\ref{FigSign} (here $\kappa_1=1, \kappa_2=2, \kappa_3=5, \kappa_4=6$).
The bullets $_{^\bullet}$ mark additional times in $\mathcal{T}\setminus\Theta$.}
\label{FigKappa}
\end{figure}

Indeed, the first case in \eqref{countS1} is guaranteed by \eqref{prop11ujstar}-\eqref{prop12ujstar}
and the second case by \eqref{jointLBC1}-\eqref{jointLBC2} and Theorem~\ref{thmz1}(ii). 
To check the last case, by \eqref{defkappa}, we see that  
for each $\kappa\in\{\kappa_j,\dots,\kappa_{j+1}-2\}$
the interval $(s_\kappa,s_{\kappa+1})$ contains at least one element of $\Theta_2$, by 
\eqref{prop11ujstar}-\eqref{prop12ujstar},
whereas each of the intervals 
$(\hat s_{\kappa_j},s_{\kappa_j})$ and 
$(s_{\kappa_{j+1}-1},\hat s_{\kappa_{j+1}})$
contains 
at least one element of $\mathcal{T}\setminus\Theta$ by \eqref{jointLBC1}-\eqref{jointLBC2} and Theorem~\ref{thmz1}(ii).

Now, by Proposition~\ref{zerob2}, \eqref{controlphimucp} and \eqref{countS1}, for all $j\in\{1,\dots,d\}$, we have
\be\label{countS1b}
N(\hat s_{\kappa_j})-N(\hat s_{\kappa_{j+1}})\ge 
\begin{cases}
2 &\hbox{ if $\xi_{\kappa_j}=0$ or $1$} \\
\noalign{\vskip 1mm}
2(\kappa_{j+1}-\kappa_j-1)+2
&\hbox{ if $\xi_{\kappa_j}=2$} \\
\end{cases}
\ee
hence 
$$N(\hat s_{\kappa_j})-N(\hat s_{\kappa_{j+1}})\ge 2(\kappa_{j+1}-\kappa_j).$$ 
Therefore,
$$N(\hat s_1)-N(\hat s_{m+1})=N(\hat s_{\kappa_1})-N(\hat s_{\kappa_{d+1}})
\ge 2(\kappa_{d+1}-\kappa_1)=2m.$$
But $N(0)\le 2m$, due to \eqref{intersecmu5}, 
and $N(t)\in \N$ is nonincreasing by Proposition~\ref{ZeroMonot}(ii) and \eqref{controlphimucp}.
Consequently, the inequality in \eqref{countS1b} is actually an equality. 
For all $j\in\{1,\dots,d\}$, setting $I_j:=(\hat s_{\kappa_j},\hat s_{\kappa_{j+1}})$, it then
follows from \eqref{countS1} and Proposition~\ref{zerob2} that $\mathcal{T}\subset(\hat s_m,\hat s_0)$ and that
$$
\hbox{$I_j\cap \mathcal{T}$ consists of exactly}
\begin{cases}
\hbox{$1$ element of $\Theta_0$,} 
&\hbox{ if $\xi_{\kappa_j}=0$} \\ 
\noalign{\vskip 2mm}
\hbox{$2$ elements of $\mathcal{T}\setminus\Theta$,} 
&\hbox{ if $\xi_{\kappa_j}=1$} \\
\noalign{\vskip 2mm}
\hbox{$(\kappa_{j+1}-\kappa_j-1)$ elements of $\Theta_2$} &{} \\
\hbox{and $2$ elements of $\mathcal{T}\setminus\Theta$,} 
&\hbox{ if $\xi_{\kappa_j}=2$.} \\
\end{cases}
$$
On the other hand, by \eqref{defkappa}, we have $\xi_{\kappa_j}=\bar\sigma_j$ if $\bar\sigma_j\le 1$,
whereas $\xi_{\kappa_j}=2$ and $\kappa_{j+1}-\kappa_j=\bar\sigma_j$ otherwise. 
Applying Theorem~\ref{thmz1}(ii), recalling \eqref{jointLBC1}, \eqref{defTheta0}, \eqref{defTheta2}, it follows that
$$
\begin{cases}
\hbox{$u$ is classical in $I_j$ except for a single time in $\Theta_0$,} &\hbox{ if $\bar\sigma_j=0$} \\ 
\noalign{\vskip 2mm}
\hbox{$\{t\in I_j,\ u(0,t)>0\}$ is a nonempty open subinterval of $I_j$,} &\hbox{ if $\bar\sigma_j=1$} \\
\noalign{\vskip 2mm}
\hbox{$\{t\in I_j,\ u(0,t)>0\}$ is a nonempty open subinterval of $I_j$} &{}  \\
\qquad\qquad\qquad\qquad\qquad\qquad\qquad\hbox{minus $\bar\sigma_j-1$ elements of $\Theta_2$,}  &\hbox{ if $\bar\sigma_j\ge 2$.} \\
\end{cases}
$$
Consequently, the assertion in 
Theorem~\ref{conjz1}' is satisfied with $\bar t_j=\hat s_{\kappa_j}$ for $j=1,\dots,$ $d+1$. 
As for property \eqref{regulat1}, it follows from \eqref{controlphimuat1a}, \eqref{controlphimuat1}. The proof is complete.
\end{proof}

\section{Appendix}

We here prove Proposition~\ref{lemconv1}.
The proof of assertion (ii) will use the following simple lemma. 

\begin{lem}\label{lemconv0}
Let $0<X<Y<1$, $g\in C^2([0,1])$
and assume that 
\be\label{g1hyp1}
g(X)>0,\ \ g(Y)>0,
\ee
\be\label{g1hyp2}
g''>0\quad\hbox{on $[X,Y]$.}
\ee
Then there exists $\bar g\in C^2([0,1])$ such that $\bar g_{|[0,X]\cup[Y,1]}=g$, 
$$\bar g\ge g,\quad \bar g>0\ \hbox{ and }\ \bar g''\ge 0 \ \ \hbox{on $[X,Y]$.}$$
  \end{lem}

\begin{proof} 
Let $L=\frac{g(Y)-g(X)}{Y-X}$.
Owing to assumption \eqref{g1hyp2}, there exists $\eta\in (0,(Y-X)/2)$ such that
$X_1:=X+\eta$ and $Y_1:=Y-\eta$ satisfy
\be\label{g1auxil0}
g'<L\ \hbox{ on $[X,X_1]$ \quad and \quad } g'>L \ \hbox{ on $[Y_1,Y]$}.
\ee
Let $\theta\in C^1([0,1])$ satisfy $\theta(0)=1$, $\theta(1)=0$, $\theta'(0)=\theta'(1)=0$ and $\theta'\le 0$.
Next, for $\eps_1,\eps_2\in (0,\eta)$ to be chosen, set 
$$\zeta(x)=
\begin{cases}
\theta\bigl(\eps_1^{-1}(x-X)\bigr) &\hbox{ on $[X,X+\eps_1]$} \\
\noalign{\vskip 1mm}
0 &\hbox{ on $[X+\eps_1,Y-\eps_2]$} \\
\noalign{\vskip 1mm}
\theta\bigl(\eps_2^{-1}(Y-x)\bigr) &\hbox{ on $[Y-\eps_2,Y]$.} \\
\end{cases}$$
Obviously $\zeta\in C^1([X,Y])$ satisfies
\begin{eqnarray}
&&\hskip-15mm \zeta(X)=\zeta(Y)=1,\quad \zeta'(X)=\zeta'(Y)=0,\quad 0\le\zeta\le 1, \label{g1auxil1} \\
\noalign{\vskip 1mm}
&&\hskip-15mm \zeta'\le 0 \ \hbox{ on $[X,X_1]$,} \quad
\zeta=0 \ \hbox{ on $[X_1,Y_1]$,}\quad  \zeta'\ge 0\ \hbox{ on $[Y_1,Y]$}. \label{g1auxil2}
\end{eqnarray}
On the other hand, we compute
$$\int_X^{X_1} \zeta(x)(g'(x)-L)\,dx=\eps_1\int_0^1 \theta(y)(g'(X+\eps_1y)-L)\,dy\sim -c_1\eps_1\quad \eps_1\to 0$$
and
$$I(\eps_2):=\int_{Y_1}^Y \zeta(x)(g'(x)-L)\,dx=\eps_2\int_0^1 \theta(y)(g'(Y-\eps_2y)-L)\,dy\sim c_2\eps_2,\quad \eps_2\to 0,$$
where $c_1=(L-g'(X))\int_0^1 \theta(y)\,dy>0$ and $c_2=(g'(Y)-L)\int_0^1 \theta(y)\,dy>0$.
Since $I(\eps_2)$ depends continuously on $\eps_2$, 
we may thus choose $\eps_1,\eps_2\in (0,\eta)$ small such that 
\begin{eqnarray}
&&\hskip-14mm \int_X^Y \zeta(x)(g'(x)-L)\,dx=0, \label{g1auxil3} \\
&&\hskip-15mm \int_X^Y \zeta(x)|g'(x)-L|\,dx<\min(g(X),g(Y)). \label{g1auxil4}
\end{eqnarray}

Now define
$$\bar g(x)=
\begin{cases}
g(x) &\hbox{ on $[0,X]\cup[Y,1]$} \\
\noalign{\vskip 1mm}
g(X)+L(x-X)+\int_X^x \zeta(y)(g'(y)-L)\,dy &\hbox{ on $(X,Y).$} \\
\end{cases}$$
We have $\bar g'=L+\zeta(g'-L)$ and $\bar g''=\zeta'(g'-L)+\zeta g''$ on $(X,Y)$.
By \eqref{g1auxil1} and \eqref{g1auxil3}, we deduce that $\bar g$ is $C^2$ at $x=X$ and $Y$, hence on $[0,1]$,
whereas \eqref{g1hyp2}, \eqref{g1auxil0} and \eqref{g1auxil2} guarantee that $\bar g''\ge 0$ on $[X,Y]$.
Also, since $g(X)+L(x-X)\ge \min(g(X),g(Y))$ on $[X,Y]$, \eqref{g1auxil4} implies $\bar g>0$ on $[X,Y]$.

Let us finally show that $\bar g\ge g$ on $[X,Y]$.
Using \eqref{g1auxil0} and $0\le\zeta\le 1$, we get
$$\bar g(x)=g(X)+ \int_X^x [L+\zeta(y)(g'(y)-L)]\,dy \ge g(X)+ \int_X^x g'(y)\,dy=g(x)
\ \hbox{ on $[X,X_1]$}$$
and
$$\begin{aligned}
\bar g(x)
&=g(Y)+L(x-Y)-\int_x^Y \zeta(y)(g'(y)-L)\,dy \\
&=g(Y)- \int_x^Y [L+\zeta(y)(g'(y)-L)]\,dy \ge g(Y)-\int_x^Y g'(y)\,dy=g(x)
\ \hbox{ on $[Y_1,Y].$}
\end{aligned}$$
On the other hand, since $\zeta=0$ on $[X_1,Y_1]$, $\bar g(x)$ is an affine function on $[X_1,Y_1]$.
Also, by convexity, we have $g(x)\le \varphi(x):=g(X_1)+L_1(x-X_1)$ on $[X_1,Y_1]$, with $L_1=\frac{g(Y_1)-g(X_1)}{Y_1-X_1}$.
Since $\bar g, \varphi$ are affine on $[X_1,Y_1]$ with
$\bar g(X_1)\ge g(X_1)=\varphi(X_1)$ and $\bar g(Y_1)\ge g(Y_1)=\varphi(Y_1)$,
we deduce that $\bar g\ge\varphi\ge g$ on $[X_1,Y_1]$.
\end{proof}

\begin{proof} [Proof of Proposition~\ref{lemconv1}.]
(i) Let $k_0=U^*(\textstyle\frac12)$ and let
$W\in C^2([\textstyle\frac12,\textstyle\frac32])$ be symmetric with respect to $x=1$, 
with 
$$\hbox{$W(0)\le \textstyle\frac{k_0}{2}$,\quad 
$W(x)=k_0(\textstyle\frac32-x)$ on $[\textstyle\frac54,\textstyle\frac32]$,\quad
$W''<0$ on $[1,\textstyle\frac54]$, \quad $W'<0$ on $(1,\textstyle\frac32]$.}$$
Then choosing $a\in(0,\textstyle\frac14)$ sufficiently small and setting $W_a(x)=W(1+(2a)^{-1}(x-1))$ and
$I_0=[1-a,1+a]$,
the function $W_a\in C^2(I_0)$ is symmetric with respect to $x=1$ and satisfies
$$\hbox{$W_a(1+a)=0$,\qquad $W_a<U^*$ on $I_0$,}$$
$$\hbox{$W_a'<0$ on $(1,1+a]$,\qquad
$W_a'>{U^*}'>0$ on $[1-a,1-\textstyle\frac a3]$,}$$
$$\hbox{
$W_a''<{U^*}''<0$ on $[1-\textstyle\frac a3,1+\textstyle\frac a3]$,\qquad
$W_a''(x)=0$ on $I_0\setminus (1-\textstyle\frac a2,1+\textstyle\frac a2)$.}$$
Let $K>\max(\tilde K(a/4),U^*(1/2))$,
where $\tilde K$ is given by Proposition~\ref{lem2}(ii).
The function $V=\frac{8K}{k_0}W_a$ satisfies 
 \be\label{eq1lemconv1}
 \hbox{$V\ge 2K$ in $[1-\textstyle\frac a2,1+\textstyle\frac a2]$.}
 \ee
We next modify $V$ near $x=1\pm a$ so as to extend it in a convex way to a $C^2(\R)$ function with compact support.
Namely, there exist $a_2,\bar a_2$ with 
$$\textstyle\frac a3<a_1=\textstyle\frac a2<\bar a_2<a<a_2<2a$$
 and a function $\psi\in C^2(\R)$, symmetric with respect to $x=1$, such that ${\rm Supp}(\psi)=I:=[1-a_2,1+a_2]$ and
$$\hbox{$\psi=V$ on $[1-\bar a_2,1+\bar a_2]$,\qquad 
$\psi<U^*,\ \psi'<0,\ \psi''\ge 0\ $ on $[1+\bar a_2,1+a_2)$.}$$
Then, for all $\lambda\in [0,k_0/8K]$, we have $\lambda V \le W_a<U^*$ on $I$,
whereas for all $\lambda\in (k_0/8K,1]$, we have
$$\hbox{$\lambda \psi-U^*<0$ 
on $I\setminus(1-\bar a_2,1+\bar a_2)$,\qquad
$(\lambda \psi-U^*)'>0$ on $[1-\bar a_2,1-\textstyle\frac a3]$, 
}$$
$$\hbox{$(\lambda \psi-U^*)''<0$ on $[1-\textstyle\frac a3,1+\textstyle\frac a3]$,\qquad 
$(\lambda \psi-U^*)'<0$ on $[1+\textstyle\frac a3,1+a_2]$.}$$ 
It follows that, for any $\lambda\in [0,1]$, $\lambda \psi-U^*$ has at most 2 zeros in $I$.
Moreover, by \eqref{eq1lemconv1}, we have $\psi-U^*>0$ in $[1-\textstyle\frac a2,1+\textstyle\frac a2]$,
so that $\psi-U^*$ has exactly one zero in $[1-a_2,1)$ and 
one zero in $(1,1+a_2]$.
Using $\eps^\alpha U^*(\eps^{-1}x)\equiv U^*(x)$, it is easily checked that $\psi_\eps(x)=\eps^\alpha\psi(\eps^{-1}x)$ 
enjoys all properties \eqref{psiepsP1}-\eqref{psiepsPZ}.

Moreover, since $K>\tilde K(a/4)$, \eqref{psiepsLBC} is a consequence of Proposition~\ref{lem2}(ii).

(ii) First we note that, assuming $a<1/8$, we have $a_2<2a<1/4$, hence $(1+a_2)\bar\eps<5\eps/8<(1-a_2)\eps$,
so that $\psi_{\bar\eps}$ and $\psi_\eps$ have disjoint supports. The
existence of $h$ with properties \eqref{psiepsPZ1}-\eqref{psiepsPZ2} is a consequence of Lemma~\ref{lemconv0}, applied with
$$X=(1+a_1)\bar\eps,\quad Y=(1-a_1)\eps, \quad g=\psi_{\bar\eps}+\psi_\eps-Kx^\alpha,$$
and setting $h:=\bar g+Kx^\alpha$ with $\bar g$ given by the lemma.
Assumptions \eqref{g1hyp1} and \eqref{g1hyp2} of Lemma~\ref{lemconv0} follow from 
\eqref{psiepsP2}, \eqref{psiepsP4} and $(x^\alpha)''<0$.

Let us check property \eqref{psiepsPZ2b}.
Since $h-Kx^\alpha$ is convex in $[X,Y]$ by \eqref{psiepsPZ2}, we have
$$\sup_{[X,Y]}h\le KY^\alpha+ \max\bigl\{h(X)-KX^\alpha,h(Y)-KY^\alpha\bigr\}.$$
Using \eqref{psiepsP3}, \eqref{psiepsPZ1} and $\bar\eps\le \eps/2$, it follows that
$$\begin{aligned}
\sup_{[X,Y]}h
&\le\max\bigl\{h(X)+K(Y^\alpha-X^\alpha),h(Y)\bigr\}\\
&\le\max\bigl\{\psi_{\bar\eps}\bigl((1+a_1)\bar\eps\bigr)+K((1-a_1)\eps)^\alpha,\psi_\eps\bigl((1-a_1)\eps\bigr)\bigr\}\\
&\le\max\bigl\{K_1\bar\eps^\alpha+K\eps^\alpha,K_1\eps^\alpha\bigr\}
\le\max\bigl\{2^{-\alpha}K_1+K,K_1\bigr\}\eps^\alpha=K_1\eps^\alpha,
\end{aligned}$$
assuming $K_1\ge (1-2^{-\alpha})^{-1}K$ in \eqref{psiepsP3} without loss of generality.
Since also $h\le K_1\eps^\alpha$ in $[0,X]\cup[Y,1]$ by \eqref{psiepsP3} and \eqref{psiepsPZ1}, we deduce \eqref{psiepsPZ2b}.

Let us finally check \eqref{psiepsPZ3}. For $\lambda\in [c_p/K,1]$, 
\eqref{psiepsP2}, \eqref{psiepsPZ1} and the first part of \eqref{psiepsPZ2} yield
 \be\label{proofapp1}
 \lambda h+(1-\lambda)\hat\psi-U^*\ge (\lambda K-c_p)x^\alpha>0
\quad\hbox{in $[\bar\eps,\eps]$.}
\ee
Thus assume $\lambda\in [0,c_p/K)$. By \eqref{psiepsP2}, \eqref{psiepsPZ1}, we have
 \be\label{proofapp2}
\lambda h+(1-\lambda)\hat\psi-U^*=\psi_\eps+\psi_{\bar\eps}-U^*>0
\quad\hbox{in $[\bar\eps,(1+a_1)\bar\eps]\cup[(1-a_1)\eps,\eps]$}
\ee
whereas, by the second parts of \eqref{psiepsP4} and \eqref{psiepsPZ2}, we get
 \be\label{proofapp3}
 \bigl[\lambda h+(1-\lambda)\hat\psi-U^*\bigr]''
\ge [(\lambda K-c_p)x^\alpha]''>0
\quad\hbox{in $[(1+a_1)\bar\eps,(1-a_1)\eps]$.}
\ee
Combining \eqref{proofapp1}-\eqref{proofapp3}, we deduce \eqref{psiepsPZ3} and 
this concludes the proof.
\end{proof}

{\bf Acknowledgement.} 
NM is supported by the JSPS Grant-in-Aid for Scientific Research (B) (No.20H01814).
PhS is partially supported by the Labex MME-DII (ANR11-LBX-0023-01).


\begin{thebibliography}{99}
\bibitem{A96} 
{\sc N. Alaa,}
Weak solutions of quasilinear parabolic equations with measures as initial data,
\textit{Ann. Math. Blaise Pascal} 3 (1996), 1-15.
              
\bibitem{ABG89} 
{\sc N.D. Alikakos, P.W. Bates, C.P. Grant,}
Blow up for a diffusion-advection equation,
\textit{Proc. Roy. Soc. Edinburgh Sect. A } 113 (1989), 181-190.
		
\bibitem{An88}
{\sc S. Angenent,}
The zero set of a solution of a parabolic equation,
\textit{J. Reine Angew. Math. 390} (1988), 79-96.
		
\bibitem{ARS04} 
{\sc J.M. Arrieta, A. Rodriguez-Bernal, Ph. Souplet,}
Boundedness of global solutions for nonlinear parabolic equations
involving gradient blow-up phenomena,
 \textit{Ann. Sc. Norm. Super. Pisa Cl. Sci.} (5) 3 (2004), 1-15.              
                        		
\bibitem{AS20}
{\sc A. Attouchi, Ph. Souplet,}
Gradient blow-up rates and sharp gradient estimates for diffusive Hamilton-Jacobi equations,
\textit{Calculus of Variations and PDE} 59, (2020) 153.

\bibitem{Baras-Cohen:JFA71} 
{\sc P. Baras and L. Cohen,}
Complete blow-up after $ T_{\max} $ for the solution of a semilinear heat equation,
\textit{J. Functional Analysis} 71 (1987), 142-174.

\bibitem{BB}
{\sc G. Barles, J. Burdeau,}
The Dirichlet problem for semilinear second-order degenerate elliptic equations and 
applications to stochastic exit time control problems,
\textit{ Comm. Partial Differential Equations} 20 (1995), 129-178.

\bibitem{BDL04} 
{\sc G. Barles, F. Da Lio,}
On the generalized Dirichlet problem for viscous Hamilton-Jacobi equations,
\textit{J. Math. Pures Appl.} 83 (2004), 53-75.

\bibitem{BSW} 
{\sc M. Ben-Artzi, Ph. Souplet, F.B. Weissler,}
The local theory for viscous Hamilton-Jacobi equations in Lebesgue spaces,
\textit{J. Math. Pures Appl.} 81 (2002), 343-378.

\bibitem{BKL} 
{\sc S. Benachour, G. Karch, Ph. Lauren\c cot,}
Asymptotic profiles of solutions to viscous Hamilton-Jacobi equations,
\textit{J. Math. Pures Appl.} 83 (2004), 1275-1308.

\bibitem{CG96} 
{\sc G.R. Conner, C.P. Grant,}
Asymptotics of blowup for a convection-diffusion equation with conservation,
\textit{Differential Integral Equations} 9 (1996), 719-728.            

\bibitem{CIL} 
{\sc M. Crandall, H. Ishii, P.-L. Lions,}
User's guide to viscosity solutions of second order partial differential equation,
\textit{Bull. Amer. Math. Soc.} 27 (1992), 1-67.

\bibitem{Esteve}
{\sc C. Esteve}, 
Single-point gradient blow-up on the boundary for diffusive Hamilton-Jacobi equation in domains with
non-constant curvature,
\textit{J. Math. Pures Appl.} 137 (2020),  143-177.                         
                          
\bibitem{FLa} 
{\sc M. Fila, J. Lankeit}, 
Continuation beyond interior gradient blow-up in a semilinear parabolic equation,
\textit{Math. Ann.} 377 (2020), 317-333.

\bibitem{FL94} 
{\sc M. Fila, G.M. Lieberman,}
Derivative blow-up and beyond for quasilinear parabolic equations,
\textit{Differential Integral Equations} 7 (1994), 811-821.              
              
\bibitem {FPS19}
{\sc R. Filippucci, P. Pucci, Ph. Souplet}, 
A Liouville-type theorem in half space and its application to the gradient blow up behavior for 
superquadratic diffusive Hamilton-Jacobi equations.
\textit{Comm. Partial Differential Equations} 45 (2020), 321-349.
               
\bibitem{FMP05} 
{\sc M. Fila, H. Matano, P. Pol\'a\v cik,}
Immediate regularization after blow-up,
\textit{SIAM J. Math. Anal.} 37 (2005), 752-776.
 
 \bibitem{FlSo}
{\sc W.H. Fleming, H.M. Soner,}
Controlled Markov Processes and Viscosity Solutions, in: Appl. Math., Springer-Verlag, New-York, 1993.       
              
\bibitem{Fri} 
{\sc A. Friedman,}
Partial differential equations of parabolic type, Prentice-Hall, 1964.

         
\bibitem{GV97}
{\sc V.A. Galaktionov, J.L. V\'azquez,}
Continuation of blow-up solutions of nonlinear heat equations in several space dimensions,
\textit{Comm. Pure Appl. Math.} 50 (1997), 1-67.

\bibitem{GH08} 
{\sc J.-S. Guo, B. Hu,}
Blowup rate estimates for the heat equation with a nonlinear gradient source term,
\textit{Discrete Contin. Dyn. Syst. 20} (2008), 927-937.    
      
\bibitem{Herrero-Velazquez:CRASP319} 
{\sc M.A. Herrero, J.J.L. Vel\'{a}zquez,}
Explosion de solutions des \'{e}quations paraboliques semilin\'{e}aires supercritiques, 
\textit{C. R. Acad. Sci. Paris} 319 (1994), 141-145.

\bibitem{HM04} 
{\sc M. Hesaaraki, A. Moameni,}
Blow-up positive solutions for a family of nonlinear parabolic equations in general domain in $\mathbb{R}^N$,
\textit{Michigan Math. J.} 52 (2004), 375-389.           

\bibitem{LS}
{\sc Y.-X. Li, Ph. Souplet,}
Single-point gradient blow-up on the boundary for diffusive Hamilton-Jacobi equations in planar domains,
\textit{Comm. Math. Phys.} 293 (2009), 499-517.
                                                  
\bibitem{MM09}
{\sc H. Matano, F.  Merle}, 
 Classification of type I and type II behaviors for a supercritical nonlinear heat equation,
\textit{J. Funct. Anal.} 256 (2009),  992-1064.
             
\bibitem{Mizoguchi:MZ239} 
 {\sc N. Mizoguchi,}
On the behavior of solutions for a semilinear parabolic equation with supercritical nonlinearity, 
\textit{Math. Z.} 239 (2002), 215-229.

\bibitem{Mizoguchi:JFA220}
 {\sc N. Mizoguchi,}
Various behaviors of solutions for a semilinear heat equation after blowup, 
\textit{J. Functional Analysis} 220 (2005), 214-227.

\bibitem{Miz1} 
{\sc N. Mizoguchi,}
Multiple blowup of solutions for a semilinear heat equation,
\textit{Math. Ann.} 331 (2005), 461-473.

 \bibitem{Miz2} 
 {\sc N. Mizoguchi,}
Multiple blowup of solutions for a semilinear heat equation II,
\textit{J. Differential Equations} 231 (2006) 182-194.

\bibitem{Miz-Vaz}
{\sc N. Mizoguchi, J. L. Vazquez} 
Multiple blowup for nonlinear heat equations at different places and different times, 
\textit{Indiana Univ. Math. J.} 56 (2007), 2859-2886.

\bibitem{Ni-Sacks-Tavantzis:JDE54} 
{\sc W.-M. Ni, P.E. Sacks, J. Tavantzis,}
On the asymptotic behavior of solutions of certain quasi-linear parabolic equations, 
\textit{J. Differential Equations} 54 (1984), 97-120.

\bibitem{PS} 
{\sc A. Porretta, Ph. Souplet,}
The profile of boundary gradient blow-up for the diffusive Hamilton-Jacobi equation,
\textit{International Math. Research Notices}  17 (2017), 5260-5301.

 \bibitem{PS2} 
{\sc A. Porretta, Ph. Souplet,}
Analysis of the loss of boundary conditions for the diffusive Hamilton-Jacobi equation,
\textit{Ann. Inst. H. Poincar\'e Anal. Non Lin\'eaire} 34 (2017), 1913-1923.
   
\bibitem{PS3} 
{\sc A. Porretta, Ph. Souplet,}
Blow-up and regularization rates, loss and recovery of boundary conditions for the superquadratic viscous Hamilton-Jacobi equation,
\textit{J. Math. Pures Appl.} 133 (2020) 66-117.

\bibitem{PZ} 
{\sc A. Porretta, E. Zuazua,}
Null controllability of viscous Hamilton-Jacobi equations,
\textit{Ann. Inst. H. Poincar\'e Anal. Non Lin\'eaire} 29 (2012), 301-333.
 
\bibitem{QR16} 
{\sc A. Quaas, A. Rodr\'\i guez,}
Loss of boundary conditions for fully nonlinear parabolic equations with superquadratic gradient terms,
\textit{J. Differential Equations} 264 (2018), 2897-2935.
 
\bibitem{QS07} 
{\sc P. Quittner, Ph. Souplet,}
Superlinear parabolic problems. Blow-up, global existence and steady states. Second Edition.
Birkh\"{a}user Advanced Texts: Basel Textbooks, Birkh\"{a}user Verlag, Basel, 2019.

\bibitem{S02} 
{\sc Ph. Souplet,}
Gradient blow-up for multidimensional nonlinear parabolic equations with general boundary conditions,
\textit{Differential Integral Equations} 15 (2002), 237-256.

\bibitem{SV}
{\sc Ph. Souplet, J.L. V\'azquez}, 
Stabilization towards a singular steady state with gradient blow-up for a convection-diffusion problem,
\textit{Discrete Contin. Dynam. Systems} 14 (2006), 221-234.

\bibitem{SZ06} 
{\sc Ph. Souplet,  Q.S. Zhang,}
Global solutions of inhomogeneous Hamilton-Jacobi equations,
\textit{J. Anal. Math.} 99 (2006), 355-396.
              
\bibitem{ZL13} 
{\sc Z.-C. Zhang, Z. Li,}
A note on gradient blowup rate of the inhomogeneous Hamilton-Jacobi equations,
\textit{Acta Math. Sci. Ser. B (Engl. Ed.)} 33 (2013), 678-686.
 
\end{thebibliography}
\end{document}